\documentclass[a4paper,12pt]{article}
\usepackage{currfile,datetime}
\usepackage{amsthm,amsmath,amssymb,bbm,geometry,epsfig,listings,
paralist,
enumerate,comment,nicefrac,verbatim,multirow,xcolor}
\usepackage{mathrsfs}  

\usepackage[pagewise]{lineno}[4.41]

\usepackage[linktocpage,hidelinks,colorlinks=false,linkcolor=red,urlcolor=false]{hyperref}

\usepackage[final,notref,notcite]{showkeys}

\usepackage[capitalise,compress]{cleveref} 

\crefname{equation}{}{}

\newtheorem{lemma}{Lemma}[section]

\newtheorem{theorem}[lemma]{Theorem}

\newtheorem{corollary}[lemma]{Corollary}

\newtheorem{setting}[lemma]{Setting}

\theoremstyle{definition}
\newtheorem{example}[lemma]{Example}
\newtheorem{remark}[lemma]{Remark}

\crefname{subsection}{Subsection}{Subsections}
\crefname{enumi}{Item}{Items}
\creflabelformat{enumi}{(#2\textup{#1}#3)}

\renewcommand{\gets}{\curvearrowleft}
\newcommand{\defeq}{\curvearrowleft}

\newcommand{\1}{\ensuremath{\mathbbm{1}}}
\providecommand{\N}{{\ensuremath{\mathbbm{N}}}}
\providecommand{\Z}{{\ensuremath{\mathbbm{Z}}}}
\providecommand{\R}{{\ensuremath{\mathbbm{R}}}}

\providecommand{\E}{{\ensuremath{\mathbbm{E}}}}

\renewcommand{\P}{{\ensuremath{\mathbbm{P}}}}
\newcommand{\F}{{\ensuremath{\mathbbm{F}}}}
\newcommand{\trace}{\mathrm{tr}}
\newcommand{\uniform}{\ensuremath{\mathcal{R}}}
\newcommand{\unif}{\ensuremath{\mathfrak{r}}}

\newcommand{\vastl}[2]{\left#2 \rule{0pt}{#1}\kern-1ex\right.}
\newcommand{\vastr}[2]{\left. \rule{0pt}{#1}\kern-1ex\right#2}
\newcommand{\funcF}{F}
\newcommand{\LipConstF}{L}
\newcommand{\sppr}{\ensuremath{Y}}
\newcommand{\exactpr}{\ensuremath{X}}
\newcommand{\bigV}{V}
\newcommand{\funcG}{g}
\newcommand{\smallV}{v}
\newcommand{\lyaV}{\varphi}

\newcommand{\smallU}{u}

\newcommand{\nfrac}[2]{#1/#2}
\newcommand{\nfracXB}[2]{#1/\left(#2\right)}

\newcommand{\unit}[2]{\mathbf{e}^{#1}_{#2}}
\newcommand{\mlinear}[3]{{L}^{(#1)}(#2,#3)}
\newcommand{\totalD}{\mathsf{D}}

\newcommand{\sigmaAlgebra}{{\sigma}}

\newcommand{\xeqref}[1]{}

\newcommand{\HS}{}
\newcommand{\betaA}{\beta}
\newcommand{\kappaA}{\kappa}
\newcommand{\size}[1]{|#1|}
\newcommand{\rdown}[2]{{\left \llcorner#1\right \lrcorner}_{#2}}
\newcommand{\setCalP}[2]{\mathcal{P}(#1,#2)}
\newcommand{\stTau}[3]{\boldsymbol{\tau}}
\newcommand{\funcUo}{{U}}
\newcommand{\funcUi}{{\bar{U}}}
\newcommand{\startT}{t_0}
\newcommand{\startX}{x_0}

\newcommand{\funcEta}{c_0}
\newcommand{\setD}{D}
\newcommand{\smallF}{f}

\newcommand{\betaT}{\beta}
\renewcommand{\epsilon}{\varepsilon}
\newcommand{\growrate}{\rho}

\newcommand{\lyaPsi}{\psi}
\newcommand{\FE}{\mathrm{FE}}

\title{
Multilevel Picard approximations of high-dimensional semilinear partial differential equations with\\
locally monotone coefficient functions}

\author
{Martin Hutzenthaler$^{1}$ \\
 Tuan Anh Nguyen$^{2}$\bigskip\\
\small{$^1$ Faculty of Mathematics, University of Duisburg-Essen,}\\
\small{Essen, Germany; e-mail: \texttt{martin.hutzenthaler}\textcircled{\texttt{a}}\texttt{uni-due.de}}\\
\small{$^2$ Faculty of Mathematics, University of Duisburg-Essen,}\\
\small{Essen, Germany; e-mail: \texttt{tuan.nguyen}\textcircled{\texttt{a}}\texttt{uni-due.de}}
}

\begin{document}

\maketitle
{\makeatletter
\let\@makefnmark\relax
\let\@thefnmark\relax
\@footnotetext{\emph{Key words and phrases:}
curse of dimensionality, high-dimensional
PDEs, 
multilevel
Picard approximations, multilevel Monte Carlo
method, locally monotone, tamed Euler-type approximation%
}
\@footnotetext{\emph{AMS 2010 subject classification}: 65M75} 
\makeatother}

\begin{abstract}
The full history recursive multilevel Picard approximation method for semilinear 
pa\-ra\-bo\-lic partial differential equations (PDEs)
is the
only method which provably overcomes the curse of dimensionality for general time horizons
if the coefficient functions and the nonlinearity are globally Lipschitz continuous
and the nonlinearity is gradient-independent.
In this article we extend this result to locally monotone coefficient functions.
Our results cover a range of semilinear PDEs with polynomial coefficient functions.
\end{abstract}


\section{Introduction}

High-dimensional second-order partial differential equations (PDEs) are abundant in many important areas including financial engineering, economics, quantum mechanics, statistical physics, etc; see e.g. the surveys~\cite{ElKarouiPengQuenez1997,bachmayr2016tensor}.
The challenge in the numerical approximation of high-dimensional nonlinear PDEs
lies in the possible
curse of dimensionality
by which we mean that the complexity of the problem goes up exponentially as a function of the dimension or of the inverse prescribed accuracy.
Most approximation methods of nonlinear PDEs suffer from this curse,
including
sparse grid methods (e.g., \cite{smolyak1963quadrature}),
sparse polynomial approximation (e.g., \cite{cohen2010convergence}),
and
BSDE-methods (e.g., \cite{BallyPages2003a,BouchardTouzi2004,GobetLemorWarin2005,
BenderDenk2007,
fahim2011probabilistic,
zhang2013sparse,
bender2017primal,
BriandLabart2014,
guo2015monotone};
see also the literature discussion in~\cite{e2019multilevel}).
Branching diffusion approximations do not suffer from the curse;
see, e.g.,~\cite{Skorohod1964,HenryLabordere2012,HenryLabordereTanTouzi2014,HenryLabordereEtAl2016}. However, these approximations are only applicable for small time horizons
and small terminal conditions;
see also the discussion in \cite[Section 4.7]{e2019multilevel}.
Recently various deep learning-based methods have been proposed for numerical approximations of PDEs;
see, e.g., \cite{BeckBeckerGrohsJaafariJentzen2018,BeckEJentzen2017,BCJ18,ChanMikaelWarin2018,EHJ17,EY18,EGJS18,FTT17,GHJvW18,Hen17,hure2020deep,KLY17,KremsnerSteinickeSzoelgyenyi2020,Mis18,NM18,Rai18,SS17} or the
overview article \cite{beck2020overview}.
There is empirical evidence that these deep learning-based methods work well at least for medium prescribed accuracies;
see, e.g., the simulations in \cite{ChanMikaelWarin2018,EHJ17,HJE18,BeckEJentzen2017,BeckBeckerGrohsJaafariJentzen2018}.
However, stochastic optimization methods may get trapped in local minima and
there exists no theoretical convergence result; cf., e.g., \cite{DK21}.

To the best of our knowledge the only approximation method which has been mathematically proved to
overcome the curse of dimensionality for certain semilinear PDEs
is the full history recursive multilevel Picard (MLP) method introduced
in \cite{e2021multilevel} and analyzed, e.g., in
\cite{hutzenthaler2021overcoming,
hutzenthaler2020multilevel,
hutzenthalerKruse2020multilevel,
hutzenthaler2020overcoming,
hutzenthaler2020overcomingB,
beck2020overcoming,
hutzenthaler2021strong,
giles2019generalised,
beck2020overcoming}.
In this article we extend the analysis of MLP approximations
to the case of semilinear PDEs with locally monotone coefficient functions
and globally Lipschitz continuous, gradient-independent nonlinearities.
The case of locally monotone coefficients is particularly important
since many equations from applications satisfy such a condition;
see, e.g., \cite[Section 4]{CHJ14}.
Building on the analysis of the case of Lipschitz coefficients
in \cite{hutzenthaler2020multilevel}, the nonlinearity
($f$ in \eqref{eq:intro:PDE}) is not difficult to deal with.
So we focus now on linear PDEs (\eqref{eq:intro:PDE} with $f\equiv 0$).
Linear PDEs with locally monotone coefficient functions have
been approximated in the literature with essentially optimal rate; see, e.g.,
\cite{hutzenthaler2020perturbation}
in combination with the multilevel Monte Carlo method in \cite{g08b}.
However, the analysis in  \cite{hutzenthaler2020perturbation} is not explicit
in the dimension and it remained unclear under which conditions
it is possible to approximate linear PDEs with
locally monotone coefficients without curse.
The key contribution of this article is to derive explicit error bounds so that
dependencies on the dimension become clear.
In particular, we observe that on the right-hand side of the
one-sided linear growth condition \cref{eq:onesided.linear.growth} below
the additive part may grow polynomially in the dimension whereas
it is sufficient to assume that
the prefactor of $\|x\|^2$ is bounded in the dimension.
The following theorem illustrates our main result, Theorem~\ref{m01} below,
in the case of coefficient functions which satisfy the global monotonicity
condition.
\begin{theorem}\label{thm:intro}
Consider the notation in~\cref{s05b}, let
$T,\delta\in  (0,\infty)$,
  $b,c,\beta,\eta\in [1,\infty)$,
let
$\smallF\colon \R\to\R$ be globally Lipschitz continuous,
for every $d\in\N$ 
let
$\mu_d\in C( \R^{d} ,\R^{d})$, $\sigma_d=(\sigma_{d,1},\ldots,\sigma_{d,d})\in C( \R^{d},\R^{d\times d})$,
$g_d\in C(\R^d,\R)$,
$u_d\in C^{1,2}([0,T]\times\R^d,\R)$,
for every $d\in\N$, 
$h\in(0,T]$
 let ${\setD}_h^d\subseteq \R^{d}$ satisfy
 that
$\setD_h^d=\{x\in \R^d\colon h\leq 0.5,
d^{3\eta}(\|x\|^2+d^{\eta})^{8\beta}\leq \exp ( |\ln (h)|^{\nicefrac{1}{2}} )\}
$,
assume for all $d\in\N$ that $\sup_{t\in[0,T],x\in\R^d} \frac{\lvert u_d(t,x)\rvert}{(1+\lVert x\rVert^2)^{8\beta}}<\infty$,
assume for all 
$d\in\N$,
$x,y\in\R^d$, $t\in[0,T]$,  
 $h\in (0,T]$   that
\begin{gather}
\left\langle x-y,\mu_d(x)-\mu_d(y)\right\rangle
+3\left\|\sigma_d(x)-\sigma_d(y)\right\|^2\leq c\|x-y\|^2\label{eq:global.monotonicity}\\
\langle x,\mu_d(x)\rangle+\tfrac{16\beta-1}{2}\|\sigma_d(x)\|^2\leq c\bigl[\|x\|^2+d^{\eta}\bigr]
\label{eq:onesided.linear.growth}
\\
\lvert g_d(x)-g_d(y)\rvert+
 \lVert\mu_d(x)-\mu_d(y)\rVert+
\lVert\sigma_d(x)-\sigma_d(y)\rVert
\leq cd^{\eta}\|x-y\|\bigl[\lVert x\rVert^2+\lVert y\rVert^2+d^{\eta}\bigr]^\beta,\label{eq:local.Lipschitz.continuity}
\\
\max\left\{\lvert g_d(x)\rvert,\lVert\mu_d(x)\rVert, \lVert\sigma_d(x)\rVert^2 \right\}
\leq cd^{\eta}\bigl[\lVert x\rVert^2+d^{\eta}\bigr]^{\beta},
\label{eq:polynomial.growth}
\\(\tfrac{\partial}{\partial t}u_d)(t,x)+\left\langle\mu_d(x), (\nabla_xu_d)(t,x)\right\rangle+\tfrac{1}{2} \trace\left(
\sigma_{d} (x)
(\sigma_{d}(x))^*
(\mathrm{Hess}_xu_d)(t,x)\right)=
-f(u_d(t,x)), \label{eq:intro:PDE}
\end{gather} 
and $u_d(T,x)=g_d(x)$,
let 
$  \Theta = \bigcup_{ n \in \N } \Z^n$,
let $(\Omega,\mathcal{F},\P, (\F_t)_{t\in[0,T]})$ be a filtered probability space which satisfies the usual conditions\footnote{Let $T \in (0,\infty)$ and let ${\bf \Omega} = (\Omega,\mathcal{F},\P, (\F_t)_{t\in[0,T]})$ be 
a filtered probability space. 
Then we say that ${\bf \Omega}$
satisfies the usual conditions if and only if 
it holds for all $t \in [0,T)$ that $\{ A\in \mathcal F: \P(A)=0 \} \subseteq \F_t 
= \cap_{ s \in (t,T] } \F_s$.},
let $\unif^\theta\colon \Omega\to[0,1]$, $\theta\in \Theta$, be independent random variables which are uniformly distributed on $[0,1]$, 
 let $W^{d,\theta}\colon [0,T]\times\Omega \to \R^{d}$,
$d\in\N$, $\theta\in\Theta$, be independent standard $(\F_{t})_{t\in[0,T]}$-Brownian motions with continuous sample paths,
assume that $\sigmaAlgebra(\{\unif^\theta\colon\theta\in\Theta\})$ and
$\sigmaAlgebra(\{W^{d,\theta}_t\colon d\in\N,\theta\in\Theta,t\in[0,T]\})$ are independent,
for every $d,N\in\N$,
$\theta\in\Theta$,
$t\in[0,T)$,  $x\in\R^d$
let
$
(Y^{d,N,\theta}_{t,s}(x,\omega))_{s\in[t,T],\omega\in\Omega}\colon[t,T]\times\Omega\to\R^d$ 
 satisfy  for all $k\in\{0,1,\ldots,N\}$,
 $s\in(\frac{kT}{N},\frac{(k+1)T}{N}]\cap(t,T]$ that $Y_{t,t}^{d,N,\theta}(x)=x$ and 
\begin{align}\small\color{blue}
\begin{split}\label{eq:intro:tamed.Euler}
&Y_{t,s}^{d,N,\theta}(x) = 
Y_{t,\max\{t,\frac{kT}{N}\}}^{d,N,\theta}(x)
+ \1_{{\setD}_{\size{\delta}}}\bigl(Y_{t,\max\{t,\frac{kT}{N}\}}^{d,N,\theta}(x)\bigr)\\
&
\cdot\left[\mu_d\bigl(Y_{t,\max\{t,\frac{kT}{N}\}}^{d,N,\theta}(x)\bigr)\left(s-\max\{t,\tfrac{kT}{N}\}\right)+
\frac{\sigma_d\bigl(Y_{t,\max\{t,\frac{kT}{N}\}}^{d,N,\theta}(x)\bigr)
\bigl(W^{d,\theta}_{s} -W^{d,\theta}_{\max\{t,\frac{kT}{N}\}} \bigr)}{1+\Bigl\lVert\sigma_d\bigl(Y_{t,\max\{t,\frac{kT}{N}\}}^{d,N,\theta}(x)\bigr)\bigl(W_{s}^{d,\theta} -W^{d,\theta}_{\max\{t,\frac{kT}{N}\}} \bigr)\Bigr\rVert^2}\right],
\end{split}\end{align}
let
$ 
  {\bigV}_{ n,M}^{d,\theta} \colon [0, T] \times \R^d \times \Omega \to \R
$, 
$d,n,M\in\N_0$, $\theta\in\Theta$, satisfy
for all $d,M \in \N$, $n\in \N_0$, $\theta\in\Theta $,
$ t \in [0,T]$, $x\in\R^d $
that  
\begin{align}\label{eq:method}
  &{\bigV}_{n,M}^{d,\theta}(t,x)
=
  \tfrac{ \1_{ \N }( n )}{M^n}
 \textstyle\sum\limits_{i=1}^{M^n} \displaystyle
      g_d\bigl(\sppr^{d,M^M,(\theta,0,-i),x}_{t,T}\bigr)
 \\
 \nonumber
&  +
  \textstyle\sum\limits_{\ell=0}^{n-1} \displaystyle \left[ \tfrac{(T-t)}{M^{n-\ell}}
   \textstyle\sum\limits_{i=1}^{M^{n-\ell}}\displaystyle
      \big(\smallF\circ {\bigV}_{\ell,M}^{d,(\theta,\ell,i)}-\1_{\N}(\ell)\,\smallF\circ{\bigV}_{\max\{\ell-1,0\},M}^{d,(\theta,-\ell,i)}\big)
\big(t+(T-t)\unif^{(\theta,\ell,i)},\sppr_{t,t+(T-t)\unif^{(\theta,\ell,i)}}^{d,M^M,(\theta,\ell,i),x}\big)
    \right]\!,
\end{align} 
and
for every $d,n,M\in\N$ let  $\FE_{d,n,M}\in\N$ be the number of function evaluations of $(f,g_d,\mu_d,\sigma_d)$ needed to compute one realization of ${\bigV}_{n,M}^{d,0}(0,0)\colon\Omega\to\R^d$ (cf. \eqref{c01} for a precise definition). Then there exist $\mathfrak{c}\in \R$, $ \mathfrak{n}\colon\N\times(0,1]\to\N$ such  that
for all $d\in\N$, $\epsilon\in(0,1]$
it holds that $\big(\E\bigl[|\smallU_d(0,0)-{\bigV}_{\mathfrak{n}(d,\epsilon),\mathfrak{n}(d,\epsilon)}^{d,0}(0,0)|^2\big]\big)^{ 1/2 }\leq \epsilon$ and $\FE_{d,\mathfrak{n}(d,\epsilon),\mathfrak{n}(d,\epsilon)}
\leq \mathfrak{c}
d^{16\eta\cdot (4+\delta)}\epsilon^{-(4+\delta)}$. 
\end{theorem}
Theorem \ref{thm:intro} follows 
from Theorem~\ref{m01}\footnote{%
applied for every $d\in\N$
 with $m\defeq d$,
$b\defeq d^{\eta}c2^{9\beta}\max\{1,T^{3/2},\lvert Tf(0)\rvert\}$,
$\gamma\defeq 0$, $L\defeq \sup_{v,w\in\R\colon v\neq w}\frac{\lvert f(w)-f(v)\rvert}{\lvert v-w\rvert} $,
$\alpha\defeq0$, $p\defeq 8\beta$, 
$\bar{U}\defeq (0)_{x\in\R^d}$,
$\kappa\defeq 1$,
$\varphi\defeq (\R^d\ni x\mapsto [\|x\|^2+d^{\eta}]^{8\beta}\in[1,\infty))$,
$U\defeq (8Tc)_{x\in\R^d}$,
$g\defeq g_d$,
$f\defeq ([0,T]\times\R^d\times \R\ni (t,x,w)\mapsto f(w)\in\R)$,
$\mu\defeq \mu_d$,
$\sigma\defeq \sigma_d$,
$(D_h)_{h\in(0,T]}\defeq(D_h^d)_{h\in(0,T]}$,
$(W^{\theta})_{\theta\in\Theta}\defeq (W^{d,\theta})_{\theta\in\Theta}$
in the notation of Theorem~\ref{m01}},
 Lemma~\ref{s28}\footnote{applied
for every $d\in\N$ 
with $m\defeq d$, $p\defeq 8\beta$, $a\defeq d^{\eta}$
in the notation of  Lemma~\ref{s28}},
and 
the Feynman-Kac formula.
Next, we comment on the statement of Theorem \ref{thm:intro}.
For all $\delta\in(0,1)$
the computational effort of this method to achieve an accuracy of size $\epsilon\in (0,\infty)$ 
grows at most like 
  $\epsilon^{-4-\delta}$ times a polynom
of the dimension $d\in\N$. 
%
Here, we define the errors as the $L^2$-distances  between   
the exact solutions of the PDE \eqref{eq:intro:PDE} and the MLP approximations \eqref{eq:intro:tamed.Euler} at $(0,0)$ and
we measure computational effort through
the number of function evaluations of all paramater functions of the
PDE \eqref{eq:intro:PDE} needed to compute an MLP approximation at a fixed
space-time point.
The MLP approximation method is based on the idea
\begin{enumerate}[(a)]
  \item  
 to reformulate the PDE \eqref{eq:intro:PDE} as stochastic fixed-point equation $u_d=\Phi_d(u_d)$ with a suitable
function $\Phi_d$,
\item
to approximate the fixed point $u_d$ through
Picard iterates $(u_d^{(n)})_{n\in\N_0}$,
\item
to write $u_d$ as telescoping series
\begin{equation}  \begin{split}
u_d=u_d^{(0)}+\sum_{n=1}^{\infty}(u_d^{(n+1)}-u_d^{(n)})
=u_d^{(0)}+\sum_{n=1}^{\infty}\big(\Phi_d(u_d^{(n)})-\Phi_d(u_d^{(n-1)})\big),
\end{split}     \end{equation}
and
\item to approximate the series by a finite sum and the
temporal and spatial integrals in the summands by
Monte Carlo averages with fewer and fewer independent
samples as $n$ increases;
\end{enumerate}
see, e.g., \cite{hutzenthaler2020overcoming} for more details.
The approximations \eqref{eq:intro:tamed.Euler} of the stochastic differential equation
associated with the linear part of the PDE \eqref{eq:intro:PDE} are tamed Euler-type approximations as proposed
in \cite{hjk12,hutzenthaler2015numerical,hutzenthaler2018exponential}.
We note that classical Euler-Maruyama approximations cannot be used if the coefficient
functions grow superlinearly; see \cite{hjk11,HutzenthalerJentzenKloeden2013}.
The drift coefficient $\mu_d$ and the diffusion coefficient $\sigma_d$ 
are assumed in Theorem~\ref{thm:intro} to satisfy the global monotonicity condition
\eqref{eq:global.monotonicity} and the one-sided linear growth assumption condition
\eqref{eq:onesided.linear.growth},
these coefficients $\mu_d,\sigma_d$ and the terminal condition $g_d$ are assumed
to satisfy the local Lipschitz condition \eqref{eq:local.Lipschitz.continuity}
and the polynomial growth condition \eqref{eq:polynomial.growth},
and the nonlinearity $f$ is assumed to be globally Lipschitz continuous.
A central observation of this article is that the constant $c$ 
in \eqref{eq:global.monotonicity} and in \eqref{eq:onesided.linear.growth} 
appears in the exponent of our upper bounds (cf. \cref{m01}).
Therefore, to avoid the curse of dimensionality we have to assume that      $c$ grows at most logarithmically
in the dimension or does not depend on the dimension as in \cref{thm:intro}.
Furthermore, due to the term $d^{\eta}$ 
in \eqref{eq:onesided.linear.growth}--\eqref{eq:polynomial.growth},
 our result includes, e.g., the case of additive noise where $\sigma_d= \1_{d\times d}$.  
We note that
Theorem~\ref{m01} below assumes much weaker conditions than \eqref{eq:global.monotonicity}--\eqref{eq:intro:PDE}. In particular, the restrictive global monotonicity condition \eqref{eq:global.monotonicity}
is weakened to the local monotonicity condition \eqref{w01}.

The remainder of this article is organized as follows.
Section 2 focuses on the linear PDE-part.
Lemma~\ref{s05} establishes moment estimates, Lemma~\ref{r33} provides
exponential moment estimates, and
 Lemma~\ref{g03} proves strong error estimates for tamed Euler-type approximations.
Finally, our main result, Theorem~\ref{m01} below, establishes an error analysis
 for MLP approximations of PDEs with locally monotone coefficient functions.

\subsection{Notation}
\label{s05b}
Let $\|\cdot \|\colon \bigcup_{m,n\in\N}\R^{m\times n}\to[0,\infty)  $  satisfy for all $m,n\in\N$,
$
b=(b_{ij})_{\substack{i\in [1,m]\cap\N,\,j\in [1,n]\cap\N}}\in \R^{m\times n}
$  that $\|b\|^2=\sum_{i=1}^{m}\sum_{j=1}^{n}|b_{ij}|^2$,
let $\langle\cdot ,\cdot \rangle \colon \bigcup_{n\in\N}\R^n \times\R^n\to\R$ satisfy for all
$n\in \N$, $x=(x_i)_{i\in [1,n]\cap\N},y=(y_i)_{i\in [1,n]\cap\N}\in\R^n$  that $\langle x, y\rangle=\sum_{i=1}^{n}x_iy_i$,
and
for every $m,n,\ell\in\N$ let $\mlinear{m}{\R^n}{\R^\ell}$ be the set of $m$-linear mappings from $\R^n\times\ldots\times\R^n$ ($m$ times) to $\R^\ell$ and let $\|\cdot \|_{\mlinear{m}{\R^n}{\R^\ell}}\colon \mlinear{m}{\R^n}{\R^\ell} \to\R $  satisfy 
for all $A\in \mlinear{m}{\R^n}{\R^\ell}$ that
\begin{align}\label{s36}
\|A \|_{\mlinear{m}{\R^n}{\R^\ell}} = \sup\bigl\{\|A(x_1,x_2,\ldots,x_m)\|\colon
\forall\,i\in [1,m]\cap\N\colon  x_i\in \R^n\;\text{and}\;
\|x_i\|  =1\bigr\}.
\end{align}

\section{Strong approximation theory for SDEs with locally monotone coefficient functions}
Results in the literature on moments and strong convergence rates of implementable approximations of
SDEs do not clarify the dependence of upper bounds on the dimension.
It is a key contribution of this article to provide upper bounds
for moments, exponential moments, and errors of approximations of SDEs;
see Lemma~\ref{s05}, Lemma~\ref{r33}, and
 Lemma~\ref{g03} below for details.

\renewcommand{\lyaV}{{V}}

\subsection{Moment estimates for tamed Euler-type approximations}
A key observation of this article is
that the functions $V\defeq(\R^d\ni x\mapsto(\|x\|^2+d)^{p+1}$, $d,p\in\N$,
satisfy the condition \eqref{s01b} below with dimension-independent $c$
and satisfy the Lyapunov-type condition \eqref{r05bx} below for a number
of interesting SDEs.
Note that 
the functions $V\defeq(\R^d\ni x\mapsto(\|x\|^2+1)^{p+1}$, $d,p\in\N$,
do not satisfy \eqref{r05bx} in the case of $\mu=0$ and $\sigma(x)=\1_{d\times d}$
with dimension-independent $c$.
Our proof of \cref{s05} below adapts a number of arguments from~\cite[Chapter 2]{HutzenthalerJentzen2014Memoires}.

\begin{lemma}[Moment estimates]\label{s05}
Consider the notation in~\cref{s05b},
let $d,m\in\N$,
$T\in  (0,\infty)$, 
  $p\in[3,\infty)$, $b,c\in [1,\infty)$,  $\betaA \in[0,\infty)$,
$ 
\kappaA\in (0, p/(3\betaA+4)]$,
$\lyaV\in C^3(\R^d, [1,\infty))$,
for every $s\in [0,T)$ let $\setCalP{s}{T}$ be the set given by
$
\setCalP{s}{T}=\{(t_0,t_1,\ldots,t_n)\in\R^{n+1}\colon n\in\N,s=t_0<t_1<\ldots<t_n=T\}
$,
for every 
$\theta\in\bigcup_{s\in [0,T)} \setCalP{s}{T}$,
$n\in\N$
 with $\theta=(t_0,t_1,\ldots,t_n)$ let
 $\rdown{\cdot}{\theta}\colon [t_0,t_n]\to \R $ satisfy 
for all $t\in(t_0,t_n]$
 that 
 $\rdown{t}{\theta}=\sup(\{t_0,t_1,\ldots,t_n\}\cap[t_0,t))
$ and $\rdown{t_0}{\theta}=t_0$,
let
$\mu\colon \R^{d} \to\R^{d}$, $\sigma=(\sigma_1,\ldots,\sigma_m)\colon \R^{d}\to\R^{d\times m}$
 be Borel measurable, 
  let $({\setD}_{h})_{h\in(0,T]}\subseteq \mathcal{B}(\R^{d})$,  
let $(\Omega,\mathcal{F},\P, (\F_t)_{t\in[0,T]})$ be a filtered probability space which satisfies the usual conditions,
 let $W\colon [0,T]\times\Omega \to \R^{m}$ be a standard $(\F_{t})_{t\in[0,T]}$-Brownian motion with continuous sample paths,
for every $t\in[0,T)$,  $x\in\R^d$,
$\theta\in \setCalP{t}{T}$  let
$
(Y^{\theta,x}_{t,s}(\omega))_{s\in[t,T],\omega\in\Omega}\colon[t,T]\times\Omega\to\R^d$ 
 satisfy  for all $s\in(t,T]$  that 
\begin{align}\label{g01b}Y_{t,t}^{\theta,x}=x\quad\text{and}\quad
Y_{t,s}^{\theta,x} =\biggl[ z + \1_{{\setD}_{\size{\theta}}}(z)\left[\mu(z)(s-\rdown{s}{\theta})+
\frac{\sigma(z)(W_s -W_{\rdown{s}{\theta}} )}{1+\left\|\sigma(z)(W_s -W_{\rdown{s}{\theta}} )\right\|^2}\right]\biggr]\biggr|_{z=Y_{t,t\vee\rdown{s}{\theta}}^{\theta,x}}
\end{align}
assume  for all $\ell\in\{1,2,3\}$, $x\in\R^d$   that
\begin{gather}
\left\|(\totalD^\ell \lyaV)(x)\right\|_{\mlinear{\ell}{\R^d}{\R}}
\leq c| \lyaV(x)|^{1-\frac{\ell }{p}},\label{s01b}
\\
(\totalD \lyaV(x))(\mu(x))+\frac{1}{2} \sum_{k=1}^{m} (\totalD^2\lyaV(x))(\sigma_k(x),\sigma_k(x))\leq c\lyaV(x),\label{r05bx}
\\
\|\mu(x)\| \leq b|\lyaV(x)|^{\frac{\betaA+1}{p}},\quad\text{and}\quad
\|\sigma(x)\|_{\HS}^2\leq b|\lyaV(x)|^{\frac{\betaA+2}{p}}
,\label{r05x}
\end{gather}
assume for all $h\in(0,T]$, $x\in {\setD}_{h}$ that $\lyaV(x)\leq c(b^{3}h)^{-\kappaA}$,
and 
let $\rho \in\R$ satisfy that
\begin{align}\label{r40}
\rho=(5c^{1+\nfrac{1}{\kappaA}}p)^{3p}.
\end{align}
Then 
it holds for all $t\in[0,T)$, $s\in[t,T]$, $x\in\R^d$, $\theta\in\setCalP{s}{T}$ that
$\E \!\bigl[\lyaV\bigl(Y_{t,s}^{\theta,x}\bigr)\bigr]\leq e^{\rho(s-t)}\lyaV(x)$.
\end{lemma}
\begin{proof}[Proof of \cref{s05}]
First,  \cite[Lemma~2.12]{HutzenthalerJentzen2014Memoires} proves for all $x,y\in\R^d$
that
\begin{align}\label{s06b}
\lyaV(x+y)\leq \lyaV(x)+c^p2^{p-1}\|y\||\lyaV(x)|^{1-\frac{1}{p}}+ c^p 2^{p-1}\|y\|^p.
\end{align}
Furthermore, the chain rule and \eqref{s01b} yield for all $x,y\in \R^d$, $t\in[0,1]$ that
\begin{align}\begin{split}
\left|\tfrac{d}{dt}(\lyaV(x+ty))\right|
&=\left|((\totalD \lyaV)(x+ty))y\right|\\
&\leq 
 \left\|(\totalD \lyaV)(x+ty)\right\|_{\mlinear{1}{\R^d}{\R}}
\|y\|   \leq c|V(x+ty)|^{1-\frac{1}{p}}
\|y\|  .
\end{split}
\end{align}
This and  \cite[Lemma~2.11]{HutzenthalerJentzen2014Memoires} (applied for $x,y\in\R^d$ with
$T\defeq 1 $, $ c\defeq  c\|y\| $, $y\defeq ([0,1]\ni s\mapsto \lyaV(x+s y)\in\R)$ in the notation of \cite[Lemma~2.11]{HutzenthalerJentzen2014Memoires}) imply   for all $x,y\in\R^d$ that 
\begin{align}
\label{s06a}
\lyaV(x+y)\leq \left[|\lyaV(x)|^{1/p}+\frac{c\|y\|}{p}\right]^p. 
\end{align}
This, the fundamental theorem of calculus, \eqref{s01b},  and the fact that
$\forall\,a_1,a_2\in[0,\infty),r\in [0,\infty)\colon (a_1+a_2)^r\leq 2^{(r-1)\vee0}(|a_1|^r+|a_2|^r)
$
 imply  for all $i\in\{1,2\}$, $x,y\in\R^d$  that
\begin{align}\begin{split}
&\left\|(\totalD^{i}\lyaV)(x)-(\totalD^{i}\lyaV)(y)\right\|_{\mlinear{i}{\R^d}{\R}}
\leq \int_0^1 \left\|(\totalD^{i+1}\lyaV)(x+t(y-x))\right\|_{\mlinear{i+1}{\R^d}{\R}}\|y-x\|  dt\\
&\leq \int_0^1 
c |\lyaV(x+t(y-x))|^{\frac{p-i-1}{p}}
\|y-x\|  dt
\leq 
\int_0^1 c \left[|\lyaV(x)|^{\frac{1}{p}}+\frac{c\|y-x\|}{p} \right]^{p-i-1}\|y-x\|  dt \\
&\leq 
( 2c)^{p-1} \left[|\lyaV(x)|^{\frac{p-i-1}{p}}\|y-x\|+\frac{\|y-x\|^{p-i}}{p} \right] .\\
\end{split}
\end{align}
Hence, \eqref{s36}, \eqref{r05x}, and the triangle inequality  show for all $x,y\in\R^d$ that 
\begin{align}\begin{split}
 &\left|\bigl[(\totalD \lyaV)(x)-(\totalD \lyaV)(y)\bigr](\mu(x) )+\frac{1}{2}
\sum_{k=1}^{m}\left[(\totalD^2 \lyaV)(x)-(\totalD^2 \lyaV)(y)\right]\!(\sigma_k(x),\sigma_k(x))
\right|\\
&\leq \left\|(\totalD \lyaV)(x)-(\totalD \lyaV)(y)\right\|_{\mlinear{1}{\R^d}{\R}}\left\|(\mu(x) )\right\|
+\frac{1}{2}\left\|(\totalD^2 \lyaV)(x)-(\totalD^2 \lyaV)(y)\right\|_{\mlinear{2}{\R^d}{\R}}
\|\sigma(x)\|_{\HS}^2\\
&\leq \sum_{i=1}^{2}\left[
( 2c)^{p-1} \left[|\lyaV(x)|^{\frac{p-i-1}{p}}\|y-x\|+\frac{\|y-x\|^{p-i}}{p}\right]b |\lyaV(x)|^\frac{\betaA+i}{p}\right]\\
&=
(2c)^{p}\tfrac{b}{c}\left[|\lyaV(x)|^{\frac{p+\betaA-1}{p}}\|y-x\|+\frac{1}{2p}
\sum_{i=1}^{2}
|\lyaV(x)|^{\frac{\betaA+i}{p}}\|y-x\|^{p-i}\right]
.\label{s13}
\end{split}
\end{align}
Next, Jensen's inequality,
the Burkholder-Davis-Gundy inequality (see, e.g.,  \cite[Lemma~7.7]{dz92}, applied with $r\defeq(r\vee2)/2$ and
$\Phi\defeq ([0,T]\ni s\mapsto \sigma(x)\in\R^{d\times m})$ in the notation of \cite[Lemma~7.7]{dz92}) imply for all $r\in [1,\infty)$,  $t\in[0,T]$, $x\in\R^d$ that
\begin{align}\label{s18c}\begin{split}
\left(\E\!\left[\big.\!\left\| \sigma(x)W_t\right\|  ^{r} \right]\right)^{\!\nicefrac{1}{r}}
&\leq 
\left(\Bigl.\E\!\left[\left\| \sigma(x)W_t\right\|^{r\vee2} \right]\right)^{\!\nfracXB{1}{r\vee2}}\\&
\leq 
\sqrt{\frac{(r\vee2)((r\vee2)-1)}{2} }\sqrt{t}\|\sigma(x)\| _{\HS}\leq r \sqrt{t}\|\sigma(x)\| _{\HS}.
\end{split}\end{align}
The fact that $\kappaA\leq p/(3\betaA)$
shows that for all  $t\in(0,T]$, $x\in\R^d$ with $\lyaV(x)\leq c(b^{3}t)^{-\kappaA} $ it holds that
\begin{align}c^{-1/\kappaA}b^{3} t|\lyaV(x)|^{\frac{\betaA}{p}}\leq 
 c^{-\frac{1}{\kappaA}}|\frac{c}{\lyaV(x)}|^{\frac{1}{\kappaA}}|\lyaV(x)|^{\frac{\betaA}{p}}= 
|\lyaV(x)|^{\frac{\betaA}{p}-\frac{1}{\kappaA}}\leq
 |\lyaV(x)|^{-\frac{2\betaA}{p}}\leq1.
\end{align}
This, the triangle inequality, \eqref{s18c}, the assumption that $c\geq1$, and \eqref{r05x}  imply  that
for all  $r\in [1,\infty)$, $t\in(0,T]$, $x\in\R^d$ with $\lyaV(x)\leq  c (b^{3}t)^{-\kappaA} $ it holds that
\begin{align}\begin{split}
&\left(\E\!\left[\Big.\!\left\|\mu(x)t+ \sigma(x)W_t\right\|  ^r \right]\right)^{\!\nfrac{1}{r}}\leq 
t\|\mu(x)\|  +
\left(\E\!\left[\Big.\!\left\| \sigma(x)W_t\right\|  ^{r} \right]\right)^{\nfrac{1}{r}}\\
&
\leq 
bt |\lyaV(x)|^{\frac{\betaA+1}{p}}+r\sqrt{t}\sqrt{b}|\lyaV(x)|^{\frac{\betaA+2}{2p}}
\\
&
\leq c^{\frac{1}{\kappaA}}
|\lyaV(x)|^{\frac{1}{p}}
\left[
c^{-\frac{1}{\kappaA}}bt |\lyaV(x)|^{\frac{\betaA}{p}}+r\left(c^{-\frac{1}{\kappaA}}b{t}|\lyaV(x)|^{\frac{\betaA}{p}}\right)^{\nicefrac{1}{2}}\right]
\\
&
\leq c^{\frac{1}{\kappaA}}
|\lyaV(x)|^{\frac{1}{p}}
(1+r)
\left(c^{-\frac{1}{\kappaA}}b{t}|\lyaV(x)|^{\frac{\betaA}{p}}\right)^{\nicefrac{1}{2}}
= c^{\frac{1}{\kappaA}}
|\lyaV(x)|^{\frac{1}{p}}
(1+r)
\left(c^{-\frac{1}{\kappaA}}b^3{t}|\lyaV(x)|^{\frac{\betaA}{p}}\right)^{\nicefrac{1}{2}}\tfrac{1}{b}
\\
&\leq \tfrac{c^{\frac{1}{\kappaA}}}{b}|\lyaV(x)|^{\frac{1}{p}}(1+r)
\left(
| \lyaV(x)|^{-\frac{2\betaA}{p}}\right)^{\nicefrac{1}{2}}
= \tfrac{c^{\frac{1}{\kappaA}}}{b}(1+r)|\lyaV(x)|^{\frac{1-\betaA}{p}}
.\label{s18b}
\end{split}\end{align}
Next, the assumption that $p\geq 3$ shows for all $i\in\{1,2\}$ that
\begin{align}\begin{split}
\frac{p+\betaA-1}{p}+ \frac{1-\betaA}{p}= 1\quad\text{and}\quad
\frac{\betaA+{i}}{p}+\frac{(1-\betaA)(p-{i})}{p}= 
\frac{(i+1)\betaA+ p-\betaA p}{p}\leq 1.\\
\end{split}\end{align}
Combining
 \eqref{s13},
\eqref{s18b} (applied for  
$q\in\{1,p-1,p-2\}$ with $r\defeq q$ in the notation of \eqref{s18b}),
and the fact that $p\geq 3$
then shows that for all $t\in(0,T]$, $x\in\R^d$ with $\lyaV(x)\leq c(b^{3}t)^{-\kappaA} $ it holds that
\begin{align}
&\E\!\left[
\left|\small
\begin{matrix}
((\totalD \lyaV)(x)-(\totalD \lyaV)(y))(\mu(x) )+\displaystyle\frac{1}{2}
\sum_{k=1}^{m}\left[(\totalD^2 \lyaV)(x)-(\totalD^2 \lyaV)(y)\right]\!(\sigma_k(x),\sigma_k(x))\end{matrix}
\right|\Bigr|_{y= x+\mu(x)t+\sigma(x)W_t }\right]\nonumber \\
&\leq 
(2c)^{p}\!\left[
|\lyaV(x)|^{\frac{p+\betaA-1}{p}}\E\!\left[\Big.\|\mu(x)t+\sigma(x)W_t\|\right]
+\left[\frac{1}{2p}
\sum_{i=1}^{2}
|\lyaV(x)|^{\frac{\betaA+i}{p}}\E\!\left[ \Big.\|\mu(x)t+\sigma(x)W_t\|^{p-i}\right]\right]\right]\nonumber \\
&
\leq(2c)^{p}\tfrac{b}{c}\!\left[
|\lyaV(x)|^{\frac{p+\betaA-1}{p}} 2\tfrac{c^{\frac{1}{\kappaA}}}{b}
|\lyaV(x)|^{\frac{1-\betaA}{p}}
+\frac{1}{2p}
\sum_{i=1}^{2}\left[
|\lyaV(x)|^{\frac{\betaA+i}{p}} \left[\tfrac{c^{\frac{1}{\kappaA}}}{b}(p-i+1)\right]^{p-i}
|\lyaV(x)|^{\frac{(1-\betaA)(p-i)}{p}}\right]\right]
\nonumber \\
&\leq (2c)^p c^{\frac{1}{\kappaA}p}p^p\left(\frac{2}{p^p}+\frac{p}{2p}+\frac{p-1}{2p}\right)\lyaV(x)\leq (2c^{1+\frac{1}{\kappaA}}p)^p\lyaV(x).
\label{s18}
\end{align}
It\^o's formula, a telescoping sum argument, the triangle inequality, and \eqref{r05bx}  then ensure  that for all $t\in(0,T]$, $x\in\R^d$ with $\lyaV(x)\leq c(b^{3}t)^{-\kappaA} $ it holds that
{\allowdisplaybreaks
\begin{align}
&\E \!\left[\big. \lyaV(x+\mu(x)t+\sigma(x)W_t))\right]\nonumber \\
&=\lyaV(x)+\int_0^t
\E\left[\left.\left[
(\totalD f(y))(\mu(x))+\frac{1}{2} \sum_{k=1}^{m}(\totalD^2f(y))(\sigma_k(x),\sigma_k(x))\right]\right|_{y=x+\mu(x)s+\sigma(x)W_s}\right]ds\nonumber \\
&\leq \lyaV(x)+ \left[(\totalD \lyaV(x))(\mu(x))+\frac{1}{2} \sum_{k=1}^{m} (\totalD^2\lyaV(x))(\sigma_k(x),\sigma_k(x))\right]t\nonumber \\&\qquad+\int_0^t
\E\!\vastl{25pt}[\biggl|
\begin{matrix}
((\totalD \lyaV)(x)-(\totalD \lyaV)(y))(\mu(x) )\\
+\frac{1}{2}
\sum_{k=1}^{m}\left[(\totalD^2 \lyaV)(x)-(\totalD^2 \lyaV)(y)\right]\!(\sigma_k(x),\sigma_k(x))
\end{matrix}
\biggr|\Bigr|_{y= x+\mu(x)s+\sigma(x)W_s }\vastl{25pt}]ds\nonumber \\
&\leq  \lyaV(x)\left[1+ c t+\int_0^t
 (2c^{1+\frac{1}{\kappaA}}p)^p\,ds \right]
\leq\lyaV(x)\left(1+(3c^{1+\frac{1}{\kappaA}}p)^pt\right)
.
\label{s19}
\end{align}}%
Next, the fact that
$ 
\kappaA\in [0, p/(3\betaA+4)]$ implies that
\begin{align}
\frac{1}{2}- \kappaA\left[\frac{3(\betaA+2)}{2p}-\frac{1}{p}\right]\geq 
\frac{1}{2}- \frac{p}{3\betaA+4}\frac{(3\betaA+4)}{2p}=0.
\end{align}
The fact that 
$\forall\,x\in \R^d\colon
\left\|x -x/(1+\|x\|^2)\right\|  \leq \|x\|^{3}  
$,
\eqref{s18b}, the fact that 
$1+3p\leq 3.5p$,
the fact that 
$1\leq \lyaV$,
the fact that 
$b\geq 1$,
and
the fact that $\kappaA\in (0, p/(3\betaA+4)]$,
then
prove that for all $t\in(0,T]$, $x\in\R^d$
with $\lyaV(x)\leq c  (b^{3}t)^{-\kappaA}$ it holds that
$c^{-\nfrac{1}{\kappaA}}b^{3}t\leq 1$ and
\begin{align}
&\left(\E\!\left[
\left\|\sigma(x)W_t
-
\frac{\sigma(x)W_t}{1+\|\sigma(x)W_t\|  ^2}\right\|  ^p\right]\!\right)^{\!\nicefrac{1}{p}}
\leq
 \left(\Big.\!\E\!\left[\big.\!\left\|\sigma(x)W_t\right\|  ^{3p}\right]\right)^{\!\frac{1}{p}}=
\left[
\left(\Big.\!\E\!\left[\big.\!\left\|\sigma(x)W_t\right\|  ^{3p}\right]\right)^{\!\frac{1}{3p}}\right]^3
\nonumber \\
&\leq 
\left[ (1+3p) \tfrac{c^{\frac{1}{\kappaA}}}{b}\left(c^{-\frac{1}{\kappaA}}bt|\lyaV(x)|^{\frac{\betaA}{p}}\right)^{\!\frac{1}{2}}
|\lyaV(x)|^{\frac{1}{p}}\right]^{3}
= 
\left[(1+3p)\tfrac{c^{\frac{1}{\kappaA}}}{b}\right]^{3}
\left(c^{-\frac{1}{\kappaA}}bt\right)^{\frac{3}{2}}
|\lyaV(x)|^{\frac{3(\betaA+2)}{2p}-\frac{1}{p}}|\lyaV(x)|^{\frac{1}{p}}\nonumber \\
&
\leq 
\left[{3.5p}c^{\frac{1}{\kappaA}}\right]^{3}
\left(c^{-\frac{1}{\kappaA}}bt\right)^{\frac{3}{2}}
(c (b^{3}t)^{-\kappaA})^{\frac{3(\betaA+2)}{2p}-\frac{1}{p}}|\lyaV(x)|^{\frac{1}{p}}\nonumber \\
&
=
\left[{3.5p}\right]^{3}c^{\frac{2}{\kappaA}}
t
\cdot
\Big(c^{-\frac{1}{\kappaA}} b^3 t\Big)^{\frac{1}{2}-\kappaA\left[\frac{3(\betaA+2)}{2p}-\frac{1}{p}\right]}|\lyaV(x)|^{\frac{1}{p}}
\leq 
\left[{3.5p}\right]^{3}c^{\frac{2}{\kappaA}}
t
|\lyaV(x)|^{\frac{1}{p}}.
\label{s24}\end{align}
This, H\"older's inequality,  and \cref{s19} imply that for all $t\in(0,T]$, $x\in\R^d$
with $\lyaV(x)\leq c(b^3t)^{-\kappaA}$ it holds that
\begin{align}\begin{split}
&\E\!\left[\left\|\sigma(x)W_t
-
\frac{\sigma(x)W_t}{1+\|\sigma(x)W_t\|  ^2}\right\|  |\lyaV(x +\mu(x)t+\sigma(x)W_t)|^{1-\frac{1}{p}}\right]\\
&
\leq 
\left(\!\E\!\left[\left\|\sigma(x)W_t
-
\frac{\sigma(x)W_t}{1+\|\sigma(x)W_t\|  ^2}\right\|^p  \right]\right)^{\!\frac{1}{p}}\left(\E\!
\left[|\lyaV(x +\mu(x)t+\sigma(x)W_t)|\big.\right]\Big.\!\right)^{\frac{p-1}{p}}\\
&\leq
\left[
\left[{3.5p}\right]^{3}
c^{\frac{2}{\kappaA}}t
|\lyaV(x)|^{\frac{1}{p}}\right]
\left[\lyaV(x)\left(1+(3c^{1+\frac{1}{\kappaA}}p)^pt\right)\right]^{\frac{p-1}{p}}
\leq
\left[{3.5p}\right]^{3}
c^{\frac{2}{\kappaA}}t
\lyaV(x)\left(1+(3c^{1+\frac{1}{\kappaA}}p)^pt\right).
\end{split}\end{align}
This, the assumption that
$V\geq1$,
\eqref{s06b}, 
\eqref{s19}, \eqref{s24}, the fact that
$
\forall\,a_1,a_2,t\in [0,\infty)\colon (1+a_1t)(1+a_2t)\leq  e^{(a_1+a_2)t}
$, 
the fact that $p\geq 3$,
the fact that
$3^p+2^p(3.5)^{3p}\leq [3+2(3.5)^{3}]^p\leq 5^{3p}
$,
and the definition of $\rho$ in \eqref{r40}
demonstrate that for all $t\in(0,T]$, $x\in\R^d$
with $\lyaV(x)\leq  c(b^3t)^{-\kappaA}$ it holds that $t\leq c^{\nfrac{1}{\kappaA}}$ and
{\allowdisplaybreaks
\begin{align}
&\E\!\left[\lyaV\!\left(x+\mu(x)t+\frac{\sigma(x)W_t}{1+\|\sigma(x)W_t\|  ^2}\right)\right]
\leq \E\!\left[\big. \lyaV(x +\mu(x)t+\sigma(x)W_t)\right]\nonumber \\& \quad + 2 ^{p-1}c^p\E\!\left[\left\|\sigma(x)W_t
-
\frac{\sigma(x)W_t}{1+\|\sigma(x)W_t\|  ^2}\right\|  |\lyaV(x +\mu(x)t+\sigma(x)W_t)|^{1-\frac1p}\right]\nonumber \\& \quad +2 ^{p-1}c^p\E\!\left[
\left\|\sigma(x)W_t
-
\frac{\sigma(x)W_t}{1+\|\sigma(x)W_t\|  ^2}\right\|  ^p\right]\nonumber \\
&\leq\lyaV(x)\left(1+(3c^{1+\frac{1}{\kappaA}}p)^pt\right)\\
&\qquad+
2 ^{p-1}c^p
\left[{3.5p}\right]^{3}
c^{\frac{2}{\kappaA}}t
\lyaV(x)\left(1+(3c^{1+\frac{1}{\kappaA}}p)^pt\right)\nonumber 
+2 ^{p-1}c^p
\left(\left[{3.5p}\right]^{3}
c^{\frac{2}{\kappaA}}t
|\lyaV(x)|^{\frac{1}{p}}\right)^p\nonumber \\
&\leq \lyaV(x)\left(1+ (3c^{1+\frac{1}{\kappaA}}p)^{p} t\right)+
2 ^{p}c^p
\left[{3.5p}\right]^{3p}
c^{\frac{2}{\kappaA}}t
\lyaV(x)\left(1+ (3c^{1+\frac{1}{\kappaA}}p)^{p} t\right)\nonumber \\
&=\lyaV(x)\left(1+ (3c^{1+\frac{1}{\kappaA}}p)^{p} t\right)
\left(1+2 ^{p}c^{p+\frac{2}{\kappaA}}
\left[{3.5p}\right]^{3p}
t\right)\nonumber \leq \lyaV(x)\exp \left(\left[(3c^{1+\frac{1}{\kappaA}}p)^{p}+2 ^{p}c^{p+\frac{2}{\kappaA}}
\left[{3.5p}\right]^{3p}\right]t \right)\nonumber \\
&
\leq \lyaV(x)\exp \left((c^{1+\frac{1}{\kappaA}}p)^{3p}\left[3^p+2^p(3.5)^{3p}\right]t \right)
\leq 
\lyaV(x)\exp \left((5c^{1+\frac{1}{\kappaA}}p)^{3p}t \right)
= \lyaV(x)e^{\rho t}.
\label{s26}\end{align}}%
This, the tower property, the Markov property of $W$, and the fact that
$\forall\, s\in[0,T],t\in[s,T], B\in\mathcal{B}(\R^d)\colon \P((W_{t}-W_{s})\in B)=\P( W_{t-s}\in B)$ imply for all
 $x\in\R^d$, $t\in[0,T]$, $s\in [t,T]$, $ \theta\in\setCalP{t}{T}$
  that  
\begin{align}\small\begin{split}
&\E\!\left[\lyaV(Y_{t,s}^{\theta,x})\1_{\setD_{\size{\theta}}}(Y_{t,\rdown{s}{\theta}}^{\theta,x})\right]=\E\! \left[\E \left[\lyaV(Y_{t,s}^{\theta,x})\1_{\setD_{\size{\theta}}}(Y_{t,\rdown{s}{\theta}}^{\theta,x})\big|\F_{\rdown{s}{\theta}}\right] \right]
\\&=\E\!\left[\E \!\left[\left(\lyaV\!\left(z+
\mu(z)(s-\rdown{s}{\theta})+
\frac{\sigma(z)(W_{s} -W_{\rdown{s}{\theta}} )}{1+\left\|\sigma(z)(W_{s} -W_{\rdown{s}{\theta}} )\right\|^2} \right)\1_{\setD_{\size{\theta}}}(z)
\right)
\Bigr|_{z=Y_{t, \rdown{s}{\theta}}^{\theta,x}} \middle|\F_{\rdown{s}{\theta}}\right]\right]\\
&
=\E\!\left[\E \!\left[\lyaV \!\left(z+
\mu(z)(s-\rdown{s}{\theta})+
\frac{\sigma(z)(W_{s-\rdown{s}{\theta}} )}{1+\left\|\sigma(z)(W_{s-\rdown{s}{\theta}} )\right\|^2} \right)\1_{\setD_{\size{\theta}}}(z) \right]\Bigr|_{z=Y_{t, \rdown{s}{\theta}}^{\theta,x}}\right]\\
&\leq e^{\rho(s- \rdown{s}{\theta})}
\E\!\left[
\lyaV\bigl(Y_{t, \rdown{s}{\theta}}^{\theta,\theta}(x)\bigr)
\1_{\setD_{\size{\theta}}}\bigl(Y_{t, \rdown{s}{\theta}}^{\theta,x}\bigr)
\right].\label{s27}
\end{split}\end{align}
Next, \eqref{g01b} shows for all
 $x\in\R^d$, $t\in[0,T]$, $s\in [t,T]$, $ \theta\in\setCalP{t}{T}$ 
that 
\begin{align}
\bigl\{\omega\in\Omega\colon Y_{t,\rdown{s}{\theta}}^{\theta,x}(\omega)\in\R^d\setminus \setD_{\size{\theta}}\bigr\}\subseteq\bigl\{\omega\in\Omega\colon Y_{t,s}^{\theta,x}(\omega)=Y_{t,\rdown{s}{\theta}}^{\theta,x}(\omega)\bigr\} .
\end{align}
This and \eqref{s27} imply   for all
 $x\in\R^d$, $t\in[0,T]$, $s\in [t,T]$,
$ \theta\in\setCalP{t}{T}$
  that 
\begin{align}
\begin{split}
\E \!\left[V\bigl(Y_{t,s}^{\theta,x}\bigr )\right]&= \E \!\left[V\bigl(Y_{t,s}^{\theta,x} \bigr)\1_{{\setD}_{\size{\theta}}}\bigl(Y_{t, \rdown{s}{\theta}}^{\theta,x}\bigr)\right]+
\E \!\left[V(Y_{t,s}^{\theta,x} )\1_{\R^d\setminus{\setD}_{\size{\theta}}}\bigl(Y_{t, \rdown{s}{\theta}}^{\theta,x}\bigr)\right]\\
&= \E \!\left[V\bigl(Y_{t,s}^{\theta,x} \bigr)\1_{{\setD}_{\size{\theta}}}\bigl(Y_{t, \rdown{s}{\theta}}^{\theta,x}\bigr)\right]+
\E \!\left[V\bigl(Y_{t,\rdown{s}{\theta}}^{\theta,x} \bigr)\1_{\R^d\setminus{\setD}_{\size{\theta}}}\bigl(Y_{t, \rdown{s}{\theta}}^{\theta,x}\bigr)\right]\\
&\leq e^{\rho (s- \rdown{s}{\theta})}\Biggl[\E\!\left[V\bigl(Y_{t, \rdown{s}{\theta}}^{\theta,x}\bigr)\1_{{\setD}_{\size{\theta}}}(Y_{t, \rdown{s}{\theta}}^{\theta,x})\right] +
\E \!\left[V\bigl(Y_{t,\rdown{s}{\theta}}^{\theta,x} \bigr)\1_{\R^d\setminus{\setD}_{\size{\theta}}}\bigl(Y_{t,\rdown{s}{\theta}}^{\theta,x}\bigr)\right]\Biggr]\\
&=  
e^{\rho (s-\rdown{s}{\theta})} \E\!\left[V(Y_{t,\rdown{s}{\theta}}^{\theta,x})\right].
\end{split}\end{align}
An induction argument 
and \eqref{g01b}
then show  for all
 $x\in\R^d$, $t\in[0,T]$, $s\in [t,T]$, 
$ \theta\in\setCalP{t}{T}$ that
$\E\bigl[V\bigl(Y_{t,s}^{\theta,x}\bigr )\bigr]\leq e^{\rho(s-t)}V(x)$.
This completes the proof  of \cref{s05}.
\end{proof}
\subsection{Strong error estimates for approximations of SDEs}
We establish strong error estimates in Lemma \ref{g03} below.
First we introduce the setting, Setting \ref{g01}, in which we work in the rest of this section
and then we establish exponential moment estimates which are uniform in the dimension.
\renewcommand{\lyaV}{\varphi}
\begin{setting}\label{g01}
Consider the notation in~\cref{s05b}, let $d,m\in\N$,
$T\in  (0,\infty)$, 
  $b,c,\betaT,\gamma,r\in [1,\infty)$,  $\alpha\in[0,\infty)$,
$p\in[4r \betaT,\infty)$,
$\funcUi\in C(\R^{d},[0,\infty))$, 
$ 
\kappaA\in (0, p/(3\betaT+1)]$,
$\lyaV\in C^3(\R^d, [1,\infty))$,
$\funcUo\in C^3(\R^{d},[0,\infty))$,
let 
$\mu\colon \R^{d} \to\R^{d}$, $\sigma=(\sigma_1,\ldots,\sigma_m)\colon \R^{d}\to\R^{d\times m}$
 be Borel measurable functions, 
let $\rho \in\R$ satisfy that
\begin{align}\label{r40b}
\rho=(5c^{2+\frac{1}{\kappaA}}p)^{3p},
\end{align}
for every $t\in [0,T)$ let $\setCalP{t}{T}$ be the set given by
$
\setCalP{t}{T}=\{(t_0,t_1,\ldots,t_n)\in\R^{n+1}\colon n\in\N,t=t_0<t_1<\ldots<t_n=T\}
$,
for every 
$\theta\in\bigcup_{t\in [0,T)} \setCalP{t}{T}$, $n\in\N$,
$(t_0,t_1,\ldots,t_n)\in\R^{n+1}$
 with $\theta=(t_0,t_1,\ldots,t_n)$ let
$\size{\theta}\in(0,T] $, $\rdown{\cdot}{\theta}\colon [t_0,t_n]\to \R $ satisfy 
for all $s\in(t_0,t_n]$
 that 
\begin{align}
 \size{\theta} =\max_{i\in [0,n-1]\cap\N_0}
 |t_{i+1}-t_i|,
\quad \rdown{t_0}{\theta}=t_0,
\quad\text{and}\quad\rdown{s}{\theta}=\sup(\{t_0,t_1,\ldots,t_n\}\cap[t_0,s)),\label{r40c}
\end{align}
let $(\Omega,\mathcal{F},\P, (\F_t)_{t\in[0,T]})$ be a filtered probability space which satisfies the usual conditions,
  let 
$\setD\colon (0,T]\to   \mathcal{B}(\R^{d})$,
 let $W\colon [0,T]\times\Omega \to \R^{m}$ be a standard $(\F_{t})_{t\in[0,T]}$-Brownian motion with continuous sample paths,
for every $t\in[0,T]$,  $x\in\R^d$ let $
(X_{t,s}^x(\omega))_{t\in[s,T],\omega\in\Omega}\colon[s,T]\times\Omega\to\R^d$  be adapted stochastic processes with continuous sample paths which satisfy that for all $s\in[t,T]$ it holds $\P$-a.s. that 
\begin{align}\label{g01a}
\int_{t}^s \|\mu(X_{t,r}^x)\|+\| \sigma(X_{t,r}^x)\|_{\HS}^2\,dr<\infty\quad\text{and}\quad
X_{t,s}^x= x+\int_{t}^s \mu(X_{t,r}^x)\,dr +\int_t^s \sigma(X_{t,r}^x)\,dW_r,
\end{align}
for every $t\in[0,T)$,  $x\in\R^d$,
$\theta\in \setCalP{t}{T}$  let
$
(Y^{\theta,x}_{t,s}(\omega))_{s\in[t,T],\omega\in\Omega}\colon[t,T]\times\Omega\to\R^d$ 
 satisfy that  for all $s\in(t,T]$ it holds that $Y_{t,t}^{\theta,x}=x$ and 
\begin{align}\label{g01bb}
Y_{t,s}^{\theta,x} = Y_{t,\rdown{s}{\theta}}^{\theta,x} + \1_{{\setD}_{\size{\theta}}}(Y_{t,\rdown{s}{\theta}}^{\theta,x})\left[\mu(Y_{t,\rdown{s}{\theta}}^{\theta,x})(s-\rdown{s}{\theta})+
\frac{\sigma(Y_{t,\rdown{s}{\theta}}^{\theta,x})(W_{s} -W_{\rdown{s}{\theta}} )}{1+\left\|\sigma(Y_{t,\rdown{s}{\theta}}^{\theta,x})(W_{s} -W_{\rdown{s}{\theta}} )\right\|^2}\right],
\end{align}
assume for all $\ell\in\{1,2,3\}$, $x,y\in\R^d$, $t\in[0,T]$, $s\in[t,T]$ that
\begin{gather}\label{s01}
\left\|(\totalD^\ell \lyaV)(x)\right\|_{\mlinear{\ell}{\R^d}{\R}}
\leq c| \lyaV(x)|^{1-\frac{\ell }{p}},
\\[5pt]
(\totalD \lyaV(x))(\mu(x))+\frac{1}{2} \sum_{k=1}^{m} (\totalD^2\lyaV(x))(\sigma_k(x),\sigma_k(x))\leq c\lyaV(x),\label{r05b}
\\[5pt]
\max\left\{\|\mu(x)\|, \|\sigma(x)\|_{\HS}^2, \|x\|\right\}\leq b|\lyaV(x)|^{\frac{\betaT}{p}},\label{r05}
\\[5pt]
 \|\mu(x)-\mu(y)\|\vee
\|\sigma(x)-\sigma(y)\|_{\HS}\leq b\|x-y\|\left[|\lyaV(x)|^{\frac{\beta}{p}}+|\lyaV(y)|^{\frac{\beta}{p}}\right],
\label{r05a}
\\[5pt]
\begin{split}
&{\left\langle x-y, \mu(x)-\mu(y)\right\rangle+(2r-1) \|\sigma(x)-\sigma(y)\|_{\HS}^2}\\&
\leq {\|x-y\|^2}
 \left[\frac{\funcUo(x)+\funcUo(y)}{8rT}+\frac{\funcUi(x)+\funcUi(y)}{8r}\right],\end{split}
\label{r02}
\\
\label{r11}
\left\|(\totalD^\ell \funcUo)(x)\right\|_{\mlinear{\ell}{\R^d}{\R}}\leq c| \funcUo(x)|^{1-\frac{1 }{c}},\quad 
|\funcUi(x)|\leq c(1+|\funcUo(x)|^\gamma),
\\
\label{r12}
|\funcUi (x)-\funcUi (y)|\leq 
c\left[1+|\funcUo (x)|^\gamma+|\funcUo (y)|^\gamma\right]\|x-y\|,\quad\text{and}\quad
\\
\begin{split}
(\totalD \funcUo (x))(\mu (x))&+\frac{1}{2} \sum_{k=1}^{m} (\totalD^2\funcUo (x))(\sigma_{k}(x),\sigma_{k}(x))\\&+
\frac{1}{2}e^{\alpha T}\left\|\sigma (x)^*(\nabla \funcUo) (x)\right\|^2+\funcUi (x)\leq \alpha\funcUo (x),
\end{split}\label{r35}
\end{gather}
 assume for all 
$h\in (0,T]$, $x\in {\setD}_{h}$
 that 
\begin{align}\label{r32}
\max\left\{\|\mu (x)\|,\|\sigma (x)\|,\funcUo (x),b^{3\kappaA}\lyaV(x)\right\}\leq
\min\Big\{ch^{-\frac{1}{28}},ch^{-\frac{1}{8+8\gamma}},ch^{-\kappaA},r\log\big(\tfrac{1}{h}\big)\Big\},
\end{align} 
and
let $\funcEta\colon\R\to[1,\infty)$ satisfy for all $a\in \R$ that
\begin{align}\label{r36}
\funcEta(a)=\exp \left(
\exp\left( 2\left[720 \max\{T,\alpha,1\}(ce^{\alpha T})^3 \right]^{ (720(ce^{\alpha T})^3\max\{T,1\}+7)\gamma}\right)|\min\{|a|,1\}|^{1/8}\right).
\end{align}

\end{setting}
The following lemma, Lemma \ref{r33}, suitably applies \cite[Proposition~2.14]{jentzen2018exponential}
to obtain exponential moment estimates.
\begin{lemma}[Exponential moments]
\label{r33}Assume \cref{g01} and let $\startT\in[0,T)$, $\startX\in\R^d$, $\theta\in\setCalP{\startT}{T}$.
Then  it holds for all $s\in[\startT,T]$ that
\begin{align}\begin{split}
&\E\!\left[\exp \left(e^{\alpha (T-s)}\funcUo(Y_{\startT,s}^{\theta,\startX})+\int_{\startT}^{s}
\1_{{\setD}_{\size{\theta}}}(Y_{\startT,\rdown{t}{\theta}}^{\theta,\startX})
e^{\alpha (T-t)}
\funcUi(Y_{\startT,t}^{\theta,\startX})\,dt \right) \right]\\&\leq 
\funcEta(\size{\theta}) \exp \left({\funcUo(\startX)}{e^{\alpha (T-\startT)}}\right)
\end{split}\label{r33a}
\end{align}
and
\begin{align}\begin{split}\label{r33b}
&\E\!\left[
\exp \left(e^{\alpha(T-s)}\funcUo(X_{\startT,s}^{\startX})+\int_{\startT}^{s}
e^{\alpha(T-t)}
\funcUi(X_{\startT,t}^{\startX})\,dt \right) \right]\leq 
\exp \left(\funcUo(\startX)e^{\alpha (T-\startT)}\right).
\end{split}
\end{align}
\end{lemma}

\begin{proof}[Proof of \cref{r33}]
Throughout this proof let 
$\delta,\varsigma\in\R$,
 $n\in\N$, $(t_1,t_2,\ldots,t_n)\in\R^{n}$,
$\tilde{\theta}\in\setCalP{0}{T-\startT}$,
$\tilde{Y}=(\tilde{Y}_t(\omega))_{t\in [0,T-\startT],\omega\in\Omega}\colon[0,T-\startT]\times\Omega\to\R $
 satisfy for all $t\in [0,T-\startT]$ that
\begin{align}\begin{split}
&{\delta}=\frac{1}{28},\quad  \varsigma= \frac{1}{8+8\gamma}, \quad
\theta= (\startT,t_1,\ldots,t_n),\\
&\tilde{\theta}=(0,t_1-\startT,\ldots,t_n-\startT),\quad\text{and}\quad
\tilde{Y}_t= Y_{\startT,\startT+t}^{\theta,\startX} .\end{split}
\end{align}
Then it holds that
\begin{align}\small\label{r34b}
\frac{1-14\delta}{2+2\gamma}=\frac{1}{4+4\gamma}, \quad
\varsigma\in \left(0,\frac{1-14\delta}{2+2\gamma}\right),\quad\text{and}\quad
\varsigma+\varsigma\gamma +7\delta-\frac12= \frac{1}{8}+\frac{1}{4}-\frac12=-\frac{1}{8},
\end{align}
and
it holds for all
$t\in [0,T-\startT]$ that
\begin{align}
\size{\theta}=\size{\tilde{\theta}},\quad
 \startT+\rdown{t}{\tilde{\theta}}=\rdown{\startT+t}{\theta},\quad\text{and}\quad
Y_{\startT,\rdown{\startT+t}{\theta}}^{\theta,\startX}=
Y_{\startT,\startT+\rdown{t}{\tilde{\theta}}}^{\theta,\startX}=
\tilde{Y}_{\rdown{t}{\tilde{\theta}}}.\label{r34c}
\end{align}
Then \eqref{g01bb} 
implies
for all  $t\in [0,T-\startT]$
that 
\begin{align}\begin{split}
\tilde{Y}_t&=Y_{\startT,\startT+t}^{\theta,\startX} = Y_{\startT,\rdown{\startT+t}{\theta}}^{\theta,\startX}+\1_{{\setD}_{\size{\theta}}}(Y_{\startT,\rdown{\startT+t}{\theta}}^{\theta,\startX})\\&\qquad\cdot \left[\mu(Y_{\startT,\rdown{\startT+t}{\theta}}^{\theta,\startX})(\startT+t-\rdown{\startT+t}{\theta})+
\frac{\sigma(Y_{\startT,\rdown{\startT+t}{\theta}}^{\theta,\startX})(W_{\startT+t} -W_{\rdown{\startT+t}{\theta}} )}{1+\left\|\sigma(Y_{\startT,\rdown{\startT+t}{\theta}}^{\theta,\startX})(W_{\startT+t} -W_{\rdown{\startT+t}{\theta}} )\right\|^2}\right]\\
&=
\tilde{Y}_{\rdown{t}{\tilde{\theta}}}+
\1_{{\setD}_{\size{\tilde{\theta}}}}(\tilde{Y}_{\rdown{t}{\tilde{\theta}}})\left[\mu(\tilde{Y}_{\rdown{t}{\tilde{\theta}}})(t-\rdown{t}{\tilde{\theta}})+
\frac{\sigma(\tilde{Y}_{\rdown{t}{\tilde{\theta}}})(W_{\startT+t} -W_{\startT+\rdown{t}{\tilde{\theta}}} )}{1+\left\|\sigma(\tilde{Y}_{\rdown{t}{\tilde{\theta}}})(W_{\startT+t} -W_{\startT+\rdown{t}{\tilde{\theta}}} )\right\|^2}\right].\label{r34d}
\end{split}\end{align}
Next, \eqref{r32} implies for all 
$h\in (0,T]$, $x\in {\setD}_{h}$
 that 
\begin{align}
\max\left\{\|\mu (x)\|,\|\sigma (x)\|,\funcUo (x)\right\}\leq
\min \left\{ch^{-\frac{1}{28}},ch^{-\frac{1}{8+8\gamma}}\right\}
=c\min \left\{h^{-\delta},h^{-\varsigma}\right\}.\label{r34}
\end{align}
This, the substitution rule,
\eqref{r34b}--\eqref{r34d},  \eqref{r11}--\eqref{r35}, the fact that $\gamma\ge1$, 
\cite[Proposition~2.14]{jentzen2018exponential}   
(applied with
$H\defeq\R^d$, $U\defeq\R^m$,
$T\defeq (T-\startT)$,
$\rho\defeq\alpha$, 
$\delta\defeq \delta$,
$c\defeq ce^{\alpha(T-\startT)}$, $\gamma\defeq\gamma$, $\varsigma\defeq\varsigma$,
$F\defeq \mu$, $B\defeq\sigma$, $V\defeq  e^{\alpha(T-\startT)}\funcUo$,
$\bar{V}\defeq  e^{\alpha (T-\startT)}\funcUi$,
$S\defeq\left( (0,T-\startT]\ni s\mapsto\mathrm{Id}_{\R^d}\in \mlinear{1}{\R^d}{\R^d}\right) $ where $\mathrm{Id}_{\R^d}$ is the identity function on $\R^d$,
$(D_h)_{h\in(0,T]}\defeq ({\setD}_{h})_{h\in (0,T-\startT]}$,
$(\mathcal{F}_t)_{t\in[0,T]}\defeq (\F_{\startT+t})_{t\in [0,T-\startT]}$,
$(W_t)_{t\in[0,T]}\defeq (W_{\startT+t}-W_{\startT})_{t\in[0,T-\startT]}$,
$\theta\defeq\tilde{\theta}$, $(Y^\theta_t)_{t\in [0,T]}\defeq (\tilde{Y}_t)_{t\in[0,T-\startT]}$ in the notation of \cite[Proposition~2.14]{jentzen2018exponential}),
the fact that
$ Y_{\startT,\startT}^{\theta,\startX}=\startX$, and \eqref{r36}
 show   for all $s\in[\startT,T]$  that 
{\allowdisplaybreaks
\begin{align}
&\E\!\left[\exp \left(e^{\alpha (T-s)}\funcUo(Y_{\startT,s}^{\theta,\startX})+\int_{\startT}^{s}
\1_{{\setD}_{\size{\theta}}}(Y_{\startT,\rdown{t}{\theta}}^{\theta,\startX})
e^{\alpha (T-t)}
\funcUi(Y_{\startT,t}^{\theta,\startX})\,dt \right) \right]\nonumber \\
&=  \E\!\left[\exp \left(\frac{{e^{\alpha (T-\startT)}\funcUo(Y_{\startT,s}^{\theta,\startX})}}{e^{\alpha (s-\startT)}}+\int_{0}^{s-\startT}
\1_{{\setD}_{\size{\theta}}} (Y_{\startT,\rdown{\startT+t}{\theta}}^{\theta,\startX})
{e^{\alpha(T-t-\startT)}\funcUi(Y_{\startT,\startT+t}^{\theta,\startX})}
\,dt \right) \right]\nonumber \\
&=
\E\!\left[\exp \left(\frac{e^{\alpha (T-\startT)}\funcUo(\tilde{Y}_{s-\startT})}{e^{\alpha(s-\startT)}}+\int_{0}^{s-\startT}
\frac{{\1_{{\setD}_{\size{\tilde{\theta}}}} (\tilde{Y}_{\rdown{t}{\tilde{\theta}}})
e^{\alpha (T-\startT)}
\funcUi(\tilde{Y}_{t})}
}{e^{\alpha t}}\,dt \right) \right]\nonumber \\
&\leq 
\E\!\left[
\exp \left(e^{\alpha(T-\startT)}\funcUo(Y_{\startT,\startT}^{\theta,\startX})\right)\right]\nonumber \\
&\qquad\qquad\cdot
\exp \left(\frac{
\exp\left( 2\left[720 \max\{T-\startT,\alpha,1\}(ce^{\alpha T})^3 \right]^{ (720(ce^{\alpha T})^3\max\{T-\startT,1\}+7)\gamma}\right)}{|\min\{\size{\tilde{\theta}},1\}|^{\varsigma+\varsigma\gamma +7\delta-\frac12}}\right)\nonumber \\
&\leq \exp \left(
\funcUo(\startX)e^{\alpha(T- \startT)}\right)\nonumber \\&\qquad\qquad\exp \left(
\exp\left( 2\left[720 \max\{T,\alpha,1\}(ce^{\alpha T})^3 \right]^{ (720(ce^{\alpha T})^3\max\{T,1\}+7)\gamma}\right)|\min\{\size{\theta},1\}|^{1/8}\right)
\nonumber \\
&=\exp \left(
e^{\alpha(T- \startT)}\funcUo(\startX)
\right)\funcEta(\size{\theta}).
\end{align}}%
This implies \eqref{r33a}. Next, 
the substitution rule,
\eqref{r11}--\eqref{r35}, the fact that $\gamma\ge1$, and   \cite[Corollary~2.4]{CHJ14} 
(applied 
for $s\in[\startT,T]$
with 
$d\defeq d$, 
$m\defeq m$,
$T\defeq T-\startT$, 
$O\defeq\R^d$, 
$\mu\defeq\mu$, 
$\sigma\defeq\sigma$,
$(\mathcal{F}_t)_{t\in(0,T]}\defeq (\F_{\startT+t})_{t\in (0,T-\startT]}$,
$(W_t)_{t\in[0,T]}\defeq (W_{\startT+t}-W_{\startT})_{t\in(0,T-\startT]}$,
$\alpha\defeq\alpha$,   
$U\defeq  e^{\alpha (T-\startT)}\funcUo$, $\bar{U}\defeq \funcUi e^{\alpha (T-\startT)}$, 
$\tau\defeq(\Omega\ni \omega\mapsto s-\startT\in [0,T-\startT] )$,
$(X_{t})_{t\in[0,T]}\defeq (X_{\startT,\startT+t})_{t\in [0,T-\startT]}$ in the notation of \cite[Corollary~2.4]{CHJ14}) imply for all $s\in[\startT,T]$  that 
\begin{align}\begin{split}
&\E\!\left[
\exp \left(e^{\alpha(T-s)}\funcUo(X_{\startT,s}^{\startX})+\int_{\startT}^{s}
e^{\alpha(T-t)}
\funcUi(X_{\startT,t}^{\startX})\,dt \right) \right]\\
&=\E\!\left[
\exp \left(\frac{e^{\alpha(T-\startT)}\funcUo(X_{\startT,s}^{\startX})}{e^{\alpha (s-\startT)}}+\int_{0}^{s-\startT}
\frac{e^{\alpha(T-\startT)}\funcUi(X_{\startT,\startT+t}^{\startX})}{e^{\alpha t}}
\,dt \right) \right]\\
&\leq
\E\!\left[
\exp \left(e^{\alpha (T-\startT)}\funcUo(X_{\startT,\startT}^{\startX})\right)\right]= 
\exp \left(e^{\alpha (T-\startT)}\funcUo(\startX)\right).
\end{split}
\end{align} This implies \eqref{r33b}.
The proof of \cref{r33} is thus completed.
\end{proof}

The following lemma, Lemma \ref{g03}, establishes strong error estimates
and implies strong convergence with rate $1/2$. This strong convergence rate is well known
in the literature; see, e.g., \cite{hutzenthaler2020perturbation}.
The main contribution of Lemma \ref{g03} is to derive an explicit upper bound
which allows to explore its dependence on the dimension.
\begin{lemma}[Strong error]\label{g03}
Assume \cref{g01} and let $\startX\in\R^d$, 
$\startT\in[0,T)$, 
 $\theta\in \setCalP{\startT}{T}$.
Then it holds that 
\begin{align}\begin{split}
&\sup_{s\in[\startT,T]}\left(\E\!\left[\left\|X_{t_0,s }^{\startX}-Y_{t_0,s }^{\theta,{\startX}}\right\|^r\right]\right)^{\nicefrac{1}{r}} \\&\leq
 74\sqrt{2}(T\vee1)r^{3/2}b^2\exp\left((2+\tfrac{\rho}{2r})(T-\startT)\right)\left[\funcEta(\size{\theta})
\exp \left(e^{\alpha(T-\startT)}
\funcUo(\startX)\right)
\lyaV({\startX})\right]^{\frac{1}{2r}}
\sqrt{\size{\theta}}
.
\end{split}\end{align}

\end{lemma}
\begin{proof}[Proof of \cref{g03}]
\renewcommand{\unit}[2]{\mathbf{e}_{#2}}
\newcommand{\unitT}[2]{{\mathbf{e}}_{#2}^{\mathsf{T}}}
\newcommand{\id}{\mathrm{Id}_{\R^d}}Throughout this proof
let $\id\colon\R^d\to\R^d $ be the identity on $\R^d$,
let $\psi\colon\R^d\to \R^d$ be the function which satisfies for all $y\in\R^d$ that
\begin{align}\label{r19b}
\psi(y)=\nfracXB{y}{1+\|y\|^2}
,\end{align}
let $\unit{m}{k}\in \R^{m\times 1}$, $k\in[1,m]\cap\N$, be the $m$-dimensional standard basis vectors of $\R^{m\times1}$, let
$\unitT{m}{k}$, $k\in[1,m]\cap\N$, satisfy for all $k\in[1,m]\cap\N$ that
$\unitT{m}{k}\in\R^{1\times m}$ is the transposed matrix of $\unit{m}{k}$,
let 
$Z\colon [\startT,T]\times\Omega\to \R^d$
 be the function which satisfies for all $t\in [\startT,T]$ that
\begin{align}
Z_t= \sigma(Y_{\startT,\rdown{t}{\theta}}^{\theta,\startX})(W_{t}-W_{\rdown{t}{\theta}}),\label{r19}
\end{align}
let $\stTau{\startT}{\startX}{\theta}\colon \Omega\to[\startT,T]$ satisfy that
\begin{align}
\stTau{\startT}{\startX}{\theta}= \inf(\{\rdown{s}{\theta}\colon s\in [\startT,T], Y_{\startT,\rdown{s}{\theta}}^{\theta,\startX}\notin {\setD}_{\size{\theta}}\} \}\cup\{T\}),\label{r23}
\end{align}
and
let $a\colon [\startT,T]\times\Omega\to \R^d$, $b\colon [\startT,T]\times\Omega\to\R^{d\times m}$  be the functions which
satisfy for all $t\in [\startT,T]$ that
\begin{align}\begin{split}\label{r18}
&a_{t}=\mu(Y_{\startT,\rdown{t}{\theta}}^{\theta,\startX})+\sum_{k=1}^{m}\left[ \frac{1}{2}
[(\totalD^2\psi)(Z_s) ]\left(\sigma(Y_{\startT,\rdown{t}{\theta}}^{\theta,\startX})\unit{m}{k},\sigma(Y_{\startT,\rdown{t}{\theta}}^{\theta,\startX})\unit{m}{k}\right)\right]
\quad\text{and}\quad\\
&b_{t}=\sum_{k=1}^{m}
[(\totalD\psi)({Z}_s)]\left(\sigma(Y_{\startT,\rdown{t}{\theta}}^{\theta,\startX})\unit{m}{k}\right)\unitT{m}{k}.
\end{split}\end{align}
The fact that
$\forall\, \mathfrak{a}\in\R^{d\times m}\colon \mathfrak{a}=\sum_{k=1}^{m}\mathfrak{a}\unit{m}{k}\unitT{m}{k}$ and the fact that
\begin{align}\label{r30}
\forall\, s\in[\startT,T],t\in (\rdown{s}{\theta},s)\colon \rdown{t}{\theta}=\rdown{s}{\theta}\end{align} 
then show that for all $s\in[\startT,T]$ it holds $\P$-a.s. that
\begin{align}\label{r29}
Z_s= \sum_{k=1}^{m}
\int_{\rdown{s}{\theta}}^{s}\bigl(\sigma(Y_{\startT,\rdown{t}{\theta}}^{\theta,\startX})\unit{m}{k}\bigr)\unitT{m}{k}\,dW_t\quad\text{and}\quad
\mu(Y_{\startT,\rdown{s}{\theta}}^{\theta,\startX})(s-\rdown{s}{\theta})=
\int_{\rdown{s}{\theta}}^{s}\mu(Y_{\startT,\rdown{t}{\theta}}^{\theta,\startX})\,dt.
\end{align}%
Furthermore, 
\eqref{r23} 
and \eqref{r30} imply for all $t\in[\startT,T)$ that 
\begin{align}
\left\{Y_{\startT,\rdown{t}{\theta}}^{\theta,\startT}\in {\setD}_{\size{\theta}}\right\}=\left\{t< \stTau{\startT}{\startX}{\theta}\right\}.\label{r29b}
\end{align}%
This, \eqref{g01bb}, \eqref{r23}, \eqref{r19b},
It\^o's formula, \eqref{r19},  \eqref{r18}, and \eqref{r29} show that for all  $s\in[\startT,T]$ it holds $\P$-a.s. that
\begin{align}\begin{split}
&Y_{\startT,s}^{\theta,\startX} 
= Y_{\startT,\rdown{s}{\theta}}^{\theta,\startX}+\1_{{\setD}_{\size{\theta}}}(Y_{\startT,\rdown{s}{\theta}}^{\theta,\startX})\left[\mu(Y_{\startT,\rdown{s}{\theta}}^{\theta,\startX})(s-\rdown{s}{\theta})+\psi({Z}_{s})\right]\\
&= 
Y_{\startT,\rdown{s}{\theta}}^{\theta,\startX}
+\sum_{k=1}^{m}\int_{\rdown{s}{\theta}}^{s}
\1_{{\setD}_{\size{\theta}}}(Y_{\startT,\rdown{t}{\theta}}^{\theta,\startX})
[(\totalD\psi)({Z}_{t})]\left(\sigma(Y_{\startT,\rdown{t}{\theta}}^{\theta,\startX})\unit{m}{k}\right)\unitT{m}{k}\,dW_t
\\& +
\int_{\rdown{s}{\theta}}^{s}
\1_{{\setD}_{\size{\theta}}}(Y_{\startT,\rdown{t}{\theta}}^{\theta,\startX})\left[
\mu(Y_{\startT,\rdown{t}{\theta}}^{\theta,\startX})+\frac{1}{2}\sum_{k=1}^{m}
[(\totalD^2\psi)({Z}_{t} )]\left(\sigma(Y_{\startT,\rdown{t}{\theta}}^{\theta,\startX})\unit{m}{k},\sigma(Y_{\startT,\rdown{t}{\theta}}^{\theta,\startX})\unit{m}{k}\right)\right]dt\\
&
= Y_{\startT,\rdown{s}{\theta}}^{\theta,\startX} 
 +\int_{\rdown{s}{\theta}}^s
\1_{\left\{t<\stTau{\startT}{\startX}{\theta}\right\}}
a_t\,dt+\int_{\rdown{s}{\theta}}^s\1_{\left\{t<\stTau{\startT}{\startX}{\theta}\right\}}b_t\,dW_t.
\end{split}\end{align}
An induction argument then shows that
for all $s\in[\startT,T]$ it holds $\P$-a.s. that
\begin{align}\begin{split}
Y_{\startT,s}^{\theta,\startX} 
&= 
\startX +\int_{\startT}^s\1_{\left\{t<\stTau{\startT}{\startX}{\theta}\right\}}
a_t\,dt+
\int_{\startT}^s\1_{\left\{t<\stTau{\startT}{\startX}{\theta}\right\}}
b_t\,dW_t .\end{split}
\label{r28}
\end{align}
Next, \eqref{r05}, Jensen's inequality, a Lyapunov-type estimate (see, e.g., \cite[Lemma~2.2]{CHJ14}) combined with \eqref{g01a} and \eqref{r05b}, 
the fact that $c\leq \rho$, 
and \cref{s05} (applied with $\betaA\defeq \betaT-1$ in the notation of \cref{s05}) combined with 
\eqref{s01}--\eqref{r05}, \eqref{r40b},  the assumption that 
$ 
\kappaA\in [0, p/(3\betaT+1)]$, and the fact that
 $\forall\,h\in(0,T], x\in {\setD}_{h}\colon \lyaV(x)\leq c(b^3h)^{-\kappaA}$ (see \eqref{r32})
imply that for all $t\in[\startT,T]$, $q\in [1,\infty
)$ with $\betaT q\leq p$ it holds that 
\begin{align}\begin{split}
&\left(
\E\!\left[ \left\|\mu(Y_{\startT,t}^{\theta,\startX})\right\|^{q}\right]\right)^{\frac{1}{q}}\vee
\left(\E\!\left[\left[
\left\|\sigma(Y_{\startT,t}^{\theta,\startX})\right\|_{\HS}^{2}\vee1\right]^{q}\right]\right)^{\frac{1}{q}}\vee
\left(\E\!\left[\bigl\|Y_{t_0,t}^{\theta,{\startX}}\bigr\|^{q}\right]\right)^{\frac{1}{q}}
\\
&\vee\left(
\E\!\left[ \left\|\mu(X_{\startT,t}^{\startX})\right\|^{q}\right]\right)^{\frac{1}{q}}\vee
\left(\E\!\left[\left[
\left\|\sigma(X_{\startT,t}^{\startX})\right\|_{\HS}^{2}\vee1\right]^{q}\right]\right)^{\frac{1}{q}}\vee
\left(\E\Bigl[
\left\|X_{t_0,t}^{\startX}\right\|^{q}\Bigr]\right)^{\frac{1}{q}}
\\
&
\leq b \left[
\left(
\E\!\left[\bigl|
\lyaV(Y_{\startT,t}^{\theta,\startX})\bigr|^{\frac{{\betaT} q}{p}}\right]\right)^{\frac{1}{q}}
\vee
\left(
\E\!\left[\left|
\lyaV(X_{\startT,t}^{\startX})\right|^{\frac{{\betaT} q}{p}}\right]\right)^{\frac{1}{q}}\right]\\
&
\leq b\left[
\left(
\E\bigl[
\lyaV(Y_{\startT,t}^{\theta,\startX})\bigr]\right)^{\frac{\betaT}{p}}
\vee 
\Bigl(
\E\bigl[
\lyaV(X_{\startT,t}^{\startX})\bigr]\Bigr)^{\frac{\betaT}{p}}\right]
\leq be^{\rho{\betaT} (t-\startT)/p}(\lyaV(\startX))^{{\betaT} /p}.\label{r20}
\end{split}\end{align}
Next, 
\eqref{g01bb} and the fact that Brownian motions have independent increments prove for all $t\in [\startT,T]$ that
$\sigma(Y_{\startT,\rdown{t}{\theta}}^{\theta,\startX})$ and
$W_{t}-W_{\rdown{t}{\theta}}$ are independent.
This,
\eqref{r19}, 
the Burkholder-Davis-Gundy inequality (see, e.g., Lemma~7.7 in~\cite{dz92}), and
 \eqref{r20} 
combined with the fact that $\forall\,t\in \R\colon |t|\leq |t|^2\vee1$
imply that for all   $t\in[\startT,T]$, $q\in [2,\infty)$  with $\betaT  q\leq p$
it holds that
\begin{align}\label{r18c}\begin{split}
&\left(\E\!\left[\left\| Z_{t}\right\|  ^{q} \right]\right)^{\nicefrac{1}{q}}
\leq 
\left(\E\!\left[\left\|  \sigma(Y_{\startT,\rdown{t}{\theta}}^{\theta,\startX})(W_{t}-W_{\rdown{t}{\theta}})\right\|^{q} \right]\right)^{\frac{1}{q}}\\
&=
\left(\E\!\left[\left\|  \sigma(Y_{\startT,\rdown{t}{\theta}}^{\theta,\startX})\right\|^{q} \right]\right)^{\frac{1}{q}}
\left(\E\!\left[\left\|  W_{t}-W_{\rdown{t}{\theta}}\right\|^{q}\right]\right)^{\frac{1}{q}}\\
&
\leq be^{\rho\betaT  (t-\startT)/p}(\lyaV(\startX))^{\betaT  /p}\sqrt{\size{\theta}}
\sqrt{q(q-1)/2 }.
\end{split}\end{align}
This  and the fact that $\forall\,y\in\R^d\colon \|(\totalD^2\psi)(y)\|_{\mlinear{2}{\R^d}{\R^d}}\leq 14(1\wedge\|y\|)$ (see \cite[Lemma~3.1]{HJ20}) imply that for all   $t\in[\startT,T]$, $q\in [2,\infty)$  with $\betaT  q\leq p$
it holds that
\begin{align}\begin{split}
&\left(\E\!\left[
\left\| 
(\totalD^2\psi)(Z_{t})\right\|_{\mlinear{2}{\R^d}{\R^d}}^{q}\,ds\right]\right)^{\frac{1}{q}}
\leq 14 be^{\rho\betaT (t-\startT)/p}(\lyaV(x))^{\betaT /p}\sqrt{\size{\theta}}
\sqrt{q(q-1)/2 }.
\end{split}\end{align}
This, the fact that for all $\mathfrak{m}\in\mlinear{2}{\R^d}{\R^d}$, $\mathfrak{a}\in \R^{d\times m}$ it holds that
\begin{align}
\sum_{k=1}^{m}
\|\mathfrak{m}(\mathfrak{a} \unit{m}{k},\mathfrak{a}\unit{m}{k})\|\leq 
\sum_{k=1}^{m}\left[
\|\mathfrak{m} \|_{\mlinear{2}{\R^d}{\R^d}}\|\mathfrak{a}\unit{m}{k}\|^2\right]=
\|\mathfrak{m}\|_{\mlinear{2}{\R^d}{\R^d}}\|\mathfrak{a}\|^2_{\HS},
\end{align}
 H\"older's inequality, and \eqref{r20} 
then
 prove that  for all   $t\in[\startT,T]$, $q\in [1,\infty)$  with $  2\betaT q\leq p$ it holds that
\begin{align}\begin{split}
& \left(\E\!\left[
\left\|\frac{1}{2} \sum_{k=1}^{m}
[(\totalD^2\psi)(Z_{t}) ]\left(\sigma(Y_{\startT,\rdown{t}{\theta}}^{\theta,\startX})\unit{m}{k},\sigma(Y_{\startT,\rdown{t}{\theta}}^{\theta,\startX}\unit{m}{k})\right)\right\|^{q}\right]\right)^{\frac{1}{q}}\\
&\leq \frac{1}{2}
\left[
\left(\E\!\left[
\left\| 
(\totalD^2\psi)(Z_{t})\right\|_{\mlinear{2}{\R^d}{\R^d}}^{q}
\left\|\sigma(Y_{\startT,\rdown{t}{\theta}}^{\theta,\startX})\right\|_{\HS}^{2q}\right]\right)^{\frac{1}{q}}\right]\\
&\leq \frac{1}{2}
\left(\E\!\left[
\left\| 
(\totalD^2\psi)(Z_{t})\right\|_{\mlinear{2}{\R^d}{\R^d}}^{2q}\right]\right)^{\frac{1}{2q}}
\left(\E\!\left[
\left\|\sigma(Y_{\startT,\rdown{t}{\theta}}^{\theta,\startX})\right\|_{\HS}^{4q}\right]\right)^{\frac{1}{2q}}\\
&\leq \frac{1}{2}
\left[14 be^{\rho\betaT (t-\startT)/p}(\lyaV(\startX))^{\betaT /p}\sqrt{\size{\theta}}
\sqrt{q(2q-1)}\right]\left[be^{\rho\betaT (t-\startT)/p}(\lyaV(\startX))^{\betaT /p}\right]\\
&\leq 7b^2e^{2\rho\betaT (t-\startT)/p}(\lyaV(\startX))^{2\betaT /p}\sqrt{\size{\theta}}
\sqrt{q(2q-1) }.
\end{split}\label{r21}\end{align}
Next, 
\eqref{g01bb}, \eqref{r19b}, \eqref{r19}, the fact that
$\forall\,y\in\R^d\colon \|\psi(y) \|\leq \|y\|$, \eqref{r20}, and \eqref{r18c} show 
that for all $t\in[\startT,T]$, 
$q\in [2,\infty)$  with $\betaT q\leq p$ it holds
that
\begin{align}\begin{split}
&\left(
\E\!\left[\left\|
Y_{\startT,\rdown{t}{\theta}}^{\theta,{\startX}}-
Y_{\startT,t}^{\theta,{\startX}}\right\|^{q}\right]\right)^{\frac{1}{q}}
\leq
\left(
\E\!\left[ \left\|\mu(Y_{\startT,\rdown{t}{\theta}}^{\theta,{\startX}})(t-\rdown{t}{\theta})+\psi({Z}_{t})\right\|^{q}\right]\right)^{\frac{1}{q}}
\\
&\leq 
\left(
\E\!\left[ \left\|\mu(Y_{\startT,\rdown{t}{\theta}}^{\theta,{\startX}})\right\|^{q}\right]\right)^{\frac{1}{q}}{\size{\theta}}+\left(\E\!\left[\left\| Z_{t}\right\|  ^{q}\right]\right)^{\nicefrac{1}{q}}
\\
&\leq be^{\rho\betaT (t-\startT)/p}(\lyaV({\startX}))^{\betaT /p}  \sqrt{T\size{\theta}}+
be^{\rho\betaT (t-\startT)/p}(\lyaV(\startX))^{\betaT /p}\sqrt{\size{\theta}}
\sqrt{q(q-1)/2 }\\
&\leq b(\sqrt{T}+1)e^{\rho\betaT (t-\startT)/p}(\lyaV(\startX))^{\betaT /p}\sqrt{\size{\theta}}\sqrt{q(q-1)/2 }\\
&\leq 2b(T\vee1)e^{\rho\betaT (t-\startT)/p}(\lyaV(\startX))^{\betaT /p}\sqrt{\size{\theta}}\sqrt{q(q-1)/2 }.
\end{split}\end{align}
This, \eqref{r05a}, H\"older's inequality,  the triangle inequality, and \eqref{r20} 
imply  that for all $t\in[\startT,T]$, 
$q\in [1,\infty)$  with $ 2\betaT q\leq p$ it holds
that
\begin{align}\begin{split}
&\left(
\E\!\left[
\left\|
\mu(Y_{\startT,\rdown{t}{\theta}}^{\theta,{\startX}})-
\mu(Y_{\startT,t}^{\theta,{\startX}})\right\|^{q}\right]\right)^{\frac{1}{q}}\vee\left(
\E\!\left[
\left\|\sigma(Y_{\startT,\rdown{t}{\theta}}^{\theta,{\startX}})-\sigma(Y_{\startT,t}^{\theta,{\startX}})\right\|_{\HS}^{q}\right]\right)^{\frac{1}{q}}\\
&
\leq 
b\left[
\left(
\E\!\left[\left|
V(Y_{\startT,\rdown{t}{\theta}}^{\theta,{\startX}})\right|^{\frac{\beta 2q}{p}}\right]\right)^{\frac{1}{2q}}+
\left(
\E\!\left[\left|
V(Y_{\startT,t}^{\theta,{\startX}})\right|^{\frac{\beta 2q}{p}}\right]\right)^{\frac{1}{2q}}\right]\left(
\E\!\left[
\left\|
Y_{\startT,\rdown{t}{\theta}}^{\theta,{\startX}}-
Y_{\startT,t}^{\theta,{\startX}}\right\|^{2q}\right]\right)^{\frac{1}{2q}}\\
&\leq b\left[2e^{\rho\betaT (t-\startT)/p}(\lyaV({\startX}))^{\betaT /p} \right]
\left[
2b(T\vee1)e^{\rho\betaT (t-\startT)/p}(\lyaV({\startX}))^{\betaT /p}  \sqrt{\size{\theta}}
\sqrt{q(2q-1) }\right]\\
&= 
4b^2(T\vee1)e^{2\rho\betaT (t-\startT)/p}(\lyaV({\startX}))^{2\betaT /p}  \sqrt{\size{\theta}}
\sqrt{q(2q-1) }.\label{r22}
\end{split}\end{align}
This, \eqref{r18}, the triangle inequality, and
\eqref{r21} show 
that for all $t\in[\startT,T]$, 
$q\in [1,\infty)$  with $2\betaT q\leq p$ it holds
that
\begin{align}\begin{split}
&\left[
\E\!\left[
\left\|a_{t}-\mu(Y^{\theta,\startX}_{\startT,t})\right\|^{q}\right]\right]^{\frac{1}{q}}
\leq 
\left(
\E\!\left[
\left\|
\mu(Y_{\startT,\rdown{t}{\theta}}^{\theta,{\startX}})-
\mu(Y_{\startT,t}^{\theta,{\startX}})\right\|^{q}\right]\right)^{\frac{1}{q}}\\
&\qquad\qquad\qquad
+ \left(\E\!\left[
\left\|\frac{1}{2} \sum_{k=1}^{m}
[(\totalD^2\psi)(Z_{t}) ]\left(\sigma(Y_{\startT,\rdown{t}{\theta}}^{\theta,{\startX}})\unit{m}{k},\sigma(Y_{\startT,\rdown{t}{\theta}}^{\theta,{\startX}})\unit{m}{k}\right)\right\|^{q}\right]\right)^{\frac{1}{q}}\\
&\leq 11(T\vee1)b^2e^{2\rho\betaT (t-\startT)/p}(\lyaV({\startX}))^{2\betaT /p}  \sqrt{\size{\theta}}
\sqrt{q(2q-1) }.
\end{split}\label{r20b}
\end{align}
Next, for all $\mathfrak{m}\in\mlinear{1}{\R^d}{\R^d} $, $\mathfrak{a}\in\R^{d\times m}$ it holds that
\begin{align}
\sum_{k=1}^{m}\left[
\mathfrak{m}(\mathfrak{a}\unit{m}{k})\unitT{m}{k}\right]= 
\begin{bmatrix}
\mathfrak{m}(\mathfrak{a}\unit{m}{1})&\mathfrak{m}(\mathfrak{a}\unit{m}{2})&\ldots&\mathfrak{m}(\mathfrak{a}\unit{m}{m})\\
\end{bmatrix}\in \R^{d\times m}.
\end{align}
This shows  for all $\mathfrak{m}\in\mlinear{1}{\R^d}{\R^d} $, $\mathfrak{a}\in\R^{d\times m}$  that
\begin{align}
\left\|
\sum_{k=1}^{m}
\mathfrak{m}(\mathfrak{a}\unit{m}{k})\unitT{m}{k}\right\|_{\HS}^2=\sum_{k=1}^{m}\|\mathfrak{m}(\mathfrak{a}\unit{m}{k})\|^2\leq \sum_{k=1}^{m}\|\mathfrak{m}\|^2_{\mlinear{1}{\R^d}{\R^d}}\|\mathfrak{a}\unit{m}{k}\|^2= \|\mathfrak{m}\|^2_{\mlinear{1}{\R^d}{\R^d}}
\|\mathfrak{a}\|^2_{\HS}.
\end{align}
This,
\eqref{r18}, the fact that $\forall\,\mathfrak{a}\in\R^{d\times m}\colon \sum_{k=1}^{m}\mathfrak{a}\unit{m}{k}\unitT{m}{k}=\mathfrak{a}$,  the fact that
\begin{align}\forall\,y\in\R^d\colon\|\totalD \psi(y)-\id\|_{\mlinear{1}{\R^d}{\R^d}}\leq 3\left(\|y\|\wedge1\right)^2\leq 3\|y\|
\end{align}
 (see \cite[Lemma~3.1]{HJ20}), H\"older's inequality, \eqref{r18c}, and \eqref{r20}  
 imply that
for all   $t\in[\startT,T]$, $q\in [1,\infty)$  with $2\betaT q\leq p$  it holds that
\begin{align}\begin{split}
&\left(
\E\!\left[
\left\|b_{t}-\sigma(Y_{\startT,\rdown{t}{\theta}}^{\theta,{\startX}})\right\|_{\HS}^{q}\right]\right)^{\frac{1}{q}}= \left(
\E\!\left[\left\|
\sum_{k=1}^{m}
\left[(\totalD\psi)({Z}_{t})-\id \right]\left(\sigma(Y_{\startT,\rdown{t}{\theta}}^{\theta,{\startX}})\unit{m}{k}\right)\unitT{m}{k}\right\|_{\HS}^q\right]\right)^{\frac{1}{q}}\\
&\leq 
\left(
\E\!\left[\left[
\left\|(\totalD\psi)({Z}_{t})-\id\right\|_{\mlinear{1}{\R^d}{\R^d}}
\left\|\sigma(Y_{\startT,\rdown{t}{\theta}}^{\theta,{\startX}})\right\|_{\HS}
\right]^q\right]\right)^{\frac{1}{q}}\\
&
\leq 
\left(
\E\!\left[\left[3
\left\|Z_{t}\right\|
\left\|\sigma(Y_{\startT,\rdown{t}{\theta}}^{\theta,{\startX}})\right\|_{\HS}
\right]^q\right]\right)^{\frac{1}{q}}
\leq 3
\left(
\E\!\left[
\left\| Z_{t} \right\|
^{2q}\right]\right)^{\frac{1}{2q}}
\left(
\E\!\left[
\left\|\sigma(Y_{\startT,\rdown{t}{\theta}}^{\theta,{\startX}})\right\|_{\HS}
^{2q}\right]\right)^{\frac{1}{2q}}\\
&\leq 3\left[be^{\rho\betaT (t-\startT)/p}(\lyaV({\startX}))^{\betaT /p}\sqrt{\size{\theta}}
\sqrt{q(2q-1)}\right]\left[be^{\rho\betaT (t-\startT)/p}(\lyaV({\startX}))^{\betaT /p}\right]\\
&= 3b^2e^{2\rho\betaT (t-\startT)/p}(\lyaV({\startX}))^{2\betaT /p}\sqrt{\size{\theta}}
\sqrt{q(2q-1)}.
\end{split}
\end{align}
This, the triangle inequality, and \eqref{r22}  show that for all   $t\in[\startT,T]$, $q\in [1,\infty)$  with $2\betaT q\leq p$ it holds that 
\begin{align}
&\left(
\E\!\left[
\|b_{t}-\sigma(Y_{\startT,t}^{\theta,{\startX}})\|_{\HS}^{q}\right]\right)^{\frac{1}{q}}\leq 
\left(
\E\!\left[
\left\|b_{t}-\sigma(Y_{\startT,\rdown{t}{\theta}}^{\theta,{\startX}})\right\|_{\HS}^{q}\right]\right)^{\frac{1}{q}}
+\left(
\E\!\left[
\left\|\sigma(Y_{\startT,\rdown{t}{\theta}}^{\theta,{\startX}})-\sigma(Y_{\startT,t}^{\theta,{\startX}})\right\|_{\HS}^{q}\right]\right)^{\frac{1}{q}}
\nonumber \\
&
\leq 
7(T\vee1)b^2e^{2\rho\betaT (t-\startT)/p}(\lyaV({\startX}))^{2\betaT /p}\sqrt{\size{\theta}}
\sqrt{q(2q-1)}.\label{r20c}
\end{align}
Next, observe that Jensen's inequality and Tonelli's theorem imply  for all
 $s\in (\startT,T]$, $\mathfrak{X}\in \{(X^{\startX}_{\startT,t})_{t\in[\startT,T]},(Y^{\theta,\startX}_{\startT,t})_{t\in [\startT,T]}\}$ that
\begin{align}\begin{split}
&
\E\! \left[\exp \left(\int_{\startT}^{s\wedge\stTau{\startT}{\startX}{\theta}}
\frac{\funcUo(\mathfrak{X}_{t})}{T}\,dt \right) \right]
\leq 
\E\! \left[\exp \left(\frac{1}{s-\startT}\int_{\startT}^{s}
\funcUo(\mathfrak{X}_{t})\,dt \right) \right]\\
&
\leq 
\E\! \left[\frac{1}{s-\startT}\int_{\startT}^{s}\exp \left(
\funcUo(\mathfrak{X}_{t}) \right) dt\right]
=
\frac{1}{s-\startT}\int_{\startT}^{s}
\E\! \left[\exp \left(
\funcUo(\mathfrak{X}_{t})\right) \right]
dt
\leq 
\sup_{t\in [\startT,T]}\E\! \left[\exp \left(
\funcUo(\mathfrak{X}_{t}) \right) \right].
\end{split}
\end{align}
This, \eqref{r02},  H\"older's inequality, \eqref{r29b}, the fact that
$\min\{\funcUo,\funcUi\}\geq 0$,
and \cref{r33} imply 
that 
{\allowdisplaybreaks\begin{align}
&
\sup_{s\in[\startT,T]}\left|\E\! \left[\exp\! \left(
\int_{\startT}^{s\wedge\stTau{\startT}{\startX}{\theta} }\left.2r\left[\tfrac{
{\left\langle x-y, \mu(x)-\mu(y)\right\rangle+(2r-1) \|\sigma(x)-\sigma(y)\|_{\HS}^2}}{\|x-y\|^2}\right]^+
\right|_{\substack{x=X_{\startT,t}^{\startX},y=Y_{\startT,t}^{\theta,\startX}}}
dt
\right)\right]\right|^{\frac{1}{2r}}
\nonumber \\
&\leq \sup_{s\in[\startT,T]}  \left(
\E\! \left[\exp \left(\int_{\startT}^{s\wedge\stTau{\startT}{\startX}{\theta}}
\frac{\funcUo(X_{\startT,t}^{\startX})+\funcUo(Y_{\startT,t}^{\theta,\startX})}{4T}+
\frac{\funcUi(X_{\startT,t}^{\startX})+\funcUi(Y_{\startT,t}^{\theta,\startX})}{4}\,dt \right) \right]\right)^{\frac{1}{2r}}\nonumber \\
&\leq\sup_{s\in[\startT,T]} \vastl{25pt}{[}\left(
\E\! \left[\exp \left(\int_{\startT}^{s}
\frac{\funcUo(X_{\startT,t}^{\startX})}{T}\,dt \right) \right]\right)^{\frac{1}{8r}}
\left(
\E\! \left[\exp \left(\int_{\startT}^{s}
\frac{\funcUo(Y_{\startT,t}^{\theta,\startX})}{T}\,dt \right) \right]\right)^{\frac{1}{8r}}\nonumber \\
&\qquad\quad
\cdot \left(
\E\! \left[\exp \left(\int_{\startT}^{s}
\funcUi(X_{\startT,t}^{\startX})\,dt \right) \right]\right)^{\frac{1}{8r}}
\left(
\E\! \left[\exp \left(\int_{\startT}^{s}\1_{{\setD}_{\size{\theta}}}(Y_{\startT,\rdown{t}{\theta}}^{\theta,\startX})\funcUi(Y_{\startT,t}^{\theta,\startX})\,dt \right) \right]\right)^{\frac{1}{8r}}\vastr{25pt}{]}\nonumber \\
&\leq 
\left(\sup_{s\in[\startT,T]}
\E\! \left[\exp \left(\funcUo(Y_{\startT,{s}}^{\theta,\startX})+\int_{\startT}^{s}\1_{{\setD}_{\size{\theta}}}(Y_{\startT,\rdown{t}{\theta}}^{\theta,\startX})\funcUi(Y_{\startT,t}^{\theta,\startX})\,dt \right) \right]\right)^{\frac{1}{4r}}\nonumber \\&\qquad\qquad\qquad\qquad\qquad\qquad\qquad\qquad\cdot \left(\sup_{s\in[\startT,T]}
\E\! \left[\exp \left(\funcUo(X_{\startT,s}^{\startX})+\int_{\startT}^{s}\funcUi(X_{\startT,t}^{\startX})\,dt \right) \right]\right)^{\frac{1}{4r}}\nonumber \\
&\leq \left[\funcEta(\size{\theta})\right]^{\frac{1}{4r}}\left[
\exp \left(e^{\alpha(T-\startT)}\funcUo(\startX)\right)\right]^{\frac{1}{2r}}.\label{r08}
\end{align}}%
\sloppy
Next, \eqref{r29b} and \eqref{g01bb} imply that $(\1_{\{t<\stTau{\startT}{\startX}{\theta}\}} )_{t\in[\startT,T]}$ is $(\F_t)_{t\in[\startT,T]}$-predictable. Combining this,
\eqref{g01a}, \eqref{r02}, \eqref{r28},
the fact that $(a_t)_{t\in[\startT,T]}$, $(b_t)_{t\in[\startT,T]}$, and
$(\1_{\{\stTau{\startT}{\startX}{\theta}<t\}} )_{t\in[\startT,T]}$  are $(\F_t)_{t\in[\startT,T]}$-predictable (see \eqref{r18}), 
 the fact that $(X_{\startT,t}^{\startX})_{t\in[\startT,T]}$ and
$(Y_{\startT,t}^{\startX,\theta})_{t\in[\startT,T]}$ are $(\F_t)_{t\in[\startT,T]}$-adapted, 
and \cite[Corollary~2.12]{HJ20} (applied
for $s\in[t_0,T]$
with 
$H\defeq\R^d$,
$U\defeq\R^m$, 
$O\defeq\R^d$, 
$\mathcal{O}\defeq\R^d$,
$T\defeq (T-\startT)$,
$(\mathcal{F}_t)_{t\in[0,T]}
\defeq (\F_{\startT+t})_{t\in[0,T-\startT]}$,
$(W_{t})_{t\in [0,T]}\defeq (W_{\startT+t}-W_{\startT})_{t\in[0,T-\startT] } $,
$(e_k)_{k\in[1,m]\cap\N}\defeq(\unit{m}{k})_{k\in[1,m]\cap\N}$,
$\mu\defeq \mu$, 
$\sigma\defeq\sigma$,  
$\epsilon\defeq 1$,
$p\defeq 2r$,  $\tau\defeq((s\wedge\stTau{\startT}{\startX}{\theta})-\startT)$,
$(X_t)_{t\in[0,T]}\defeq (X_{\startT,\startT+t})_{t\in [0,T-\startT]}$,
$(Y_t)_{t\in[0,T]}\defeq (Y_{\startT,\startT+t}^{\theta,\startX})_{t\in [0,T-\startT]}$,
$(a_t)_{t\in[0,T]}\defeq (\1_{\{t+\startT<\stTau{\startT}{\startX}{\theta}\}}a_{\startT+t})_{t\in[T-\startT]}$,
$(b_t)_{t\in[0,T]}\defeq (\1_{\{t+\startT<\stTau{\startT}{\startX}{\theta}\}}b_{\startT+t})_{t\in[T-\startT]}$,
$\delta\defeq 1$, 
$\rho\defeq1$, 
$r\defeq r$, 
 $q\defeq 2r$ in the notation of \cite[Corollary~2.12]{HJ20}),
\eqref{r08}, \eqref{r20b} (applied with $q\defeq2r$ in the notation of \eqref{r20b}), 
\eqref{r20c} 
(applied with $q\defeq2r$ in the notation of \eqref{r20c}), and the assumption that $4r \betaT\leq p$  imply that\fussy
{\allowdisplaybreaks
\begin{align}
&\sup_{s\in[\startT,T]}\left(\E\!\left[\left\|\1_{{\setD}_{\size{\theta}}}(Y_{\startT,\rdown{T}{\theta}}^{\theta,\startX})(X_{t_0,s }^{\startX}-Y_{t_0,s }^{\theta,{\startX}})\right\|^r\right]\right)^{\frac{1}{r}}\nonumber \\
&=\sup_{s\in[\startT,T]}\left(\E\!\left[\left\|\1_{\{\stTau{\startT}{\startX}{\theta}=T\}}(X_{t_0,s\wedge\stTau{\startT}{\startX}{\theta} }^{\startX}-Y_{t_0,s\wedge\stTau{\startT}{\startX}{\theta} }^{\theta,{\startX}})\right\|^r\right]\right)^{\frac{1}{r}}
\leq \sup_{s\in[\startT,T]}\left(\E\!\left[\left\|X_{t_0,s\wedge\stTau{\startT}{\startX}{\theta} }^{\startX}-Y_{t_0,s\wedge\stTau{\startT}{\startX}{\theta} }^{\theta,{\startX}}\right\|^r\right]\right)^{\frac{1}{r}}
\nonumber \\
&\leq 
\sup_{s\in[\startT,T]}\vastl{32pt}{\{}\left|
\E\!\left[\exp\left(\int_{\startT}^{s\wedge\stTau{\startT}{\startX}{\theta}}
\left.2r\left[\tfrac{
{\left\langle x-y, \mu(x)-\mu(y)\right\rangle+(2r-1) \|\sigma(x)-\sigma(y)\|_{\HS}^2}}{\|x-y\|^2}+\tfrac{3}{2}-\tfrac{1}{r}\right]^+
\right|_{\substack{x=X_{\startT,t}^{\startX}\nonumber \\y=Y_{\startT,t}^{\theta,\startX}}}dt
\right)\right]\right|^{\frac{1}{2r}}\nonumber \\&\quad\cdot\sqrt{2(2r-1)}
\left[\left(\E\!\left[\int_{\startT}^{s}\left\|a_t-\mu(Y_{\startT,t}^{\theta,\startX})\right\|^{2r}dt\right]\right)^{\frac{1}{2r}}+
\left(\E\!\left[\int_{\startT}^{s}\left\|b_t-\sigma(Y_{\startT,t}^{\theta,\startX})\right\|^{2r}dt\right]\right)^{\frac{1}{2r}}\right]\vastr{32pt}{\}}\nonumber \\
&\leq e^{(T-\startT)(3/2-1/r)}
\left[\funcEta(\size{\theta})\right]^{\frac{1}{4r}}
\left[
\exp \left(e^{\alpha(T-\startT)}\funcUo(\startX)\right)\right]^{\frac{1}{2r}}\sqrt{2(2r-1)}\nonumber \\
&\qquad\qquad\cdot
\left[(T-\startT)^{1/2r}\left[
18(T\vee1)b^2e^{2\rho\betaT (T-\startT)/p}(\lyaV({\startX}))^{2\betaT /p}\sqrt{\size{\theta}}
\sqrt{2r(4r-1)}
\right]\right]\nonumber \\
&\leq e^{3(T-\startT)/2}\left[\funcEta(\size{\theta})\right]^{\frac{1}{4r}}\left[
\exp \left(e^{\alpha(T-\startT)}\funcUo(\startX)\right)\right]^{\frac{1}{2r}}\nonumber \\
&\qquad\qquad\cdot 
\sqrt{32r^3}e^{(T-\startT)/(2r)}18(T\vee1)b^2 e^{2\rho\betaT (T-\startT)/p}(\lyaV({\startX}))^{2\betaT /p}\sqrt{\size{\theta}}\nonumber \\
&\leq 72\sqrt{2}(T\vee1)r^{3/2}b^2e^{(2+2\rho\betaT /p)(T-\startT)}\left[
\exp \left(e^{\alpha(T-\startT)}\funcUo(\startX)\right)\right]^{\frac{1}{2r}}(\lyaV({\startX}))^{2\betaT /p}\left[\funcEta(\size{\theta})\right]^{\frac{1}{4r}}\sqrt{\size{\theta}}
.
\label{r26}
\end{align}}%
Furthermore,
 \eqref{r32},
 Markov's inequality,
and
\cref{r33}
 imply that
\begin{align}\begin{split}
&\left(\P\!\left[Y_{\startT,\rdown{T}{\theta}}^{\theta,\startX}\notin{\setD}_{\size{\theta}}\right]\right)^{\frac{1}{2r}}\leq
\left( \P\!\left[\funcUo(Y_{\startT,\rdown{T}{\theta}}^{\theta,\startX})>-r\log(\size{\theta})\right]\right)^{\frac{1}{2r}}\\&=\left( \P\!\left[\exp \left(\funcUo(Y_{\startT,\rdown{T}{\theta}}^{\theta,\startX})\right)>\size{\theta}^{-r}\right]\right)^{\frac{1}{2r}}\\
&
\leq \left(\E\!\left[\exp \left(\funcUo(Y_{\startT,\rdown{T}{\theta}}^{\theta,\startX})\right)\right]
\size{\theta}^r\right)^{\frac{1}{2r}}
\leq \Bigl(\funcEta(\size{\theta})
\exp \left(e^{\alpha(T-\startT)}
\funcUo(\startX)\right)\Bigr)^{\frac{1}{2r}}\sqrt{\size{\theta}}.\label{r11b}
\end{split}\end{align}
H\"older's inequality, the triangle inequality, and 
a combination of 
\eqref{r20} (applied with  $q\defeq 2r$ in the notation of \eqref{r20}) and  of the assumption that $4r\betaT \leq p$ therefore prove for all $s\in[\startT,T]$ that 
\begin{align}\begin{split}
& \left(\E\!\left[\left\|\1_{\R^d\setminus{\setD}_{\size{\theta}}}(Y_{\startT,\rdown{T}{\theta}}^{\theta,\startX})(X_{t_0,s }^{\startX}-Y_{t_0,s }^{\theta,{\startX}})\right\|^r\right]\right)^{\frac{1}{r}}\\
&\leq 
\left(\P\!\left[Y_{\startT,\rdown{T}{\theta}}^{\theta,\startX}\notin{\setD}_{\size{\theta}}\right]\right)^{\frac{1}{2r}}
\left[
\left(\E\!\left[
\left\|X_{t_0,s}^{\startX}\right\|^{2r}\right]\right)^{\frac{1}{2r}}+
\left(\E\!\left[
\bigl\|Y_{t_0,s}^{\theta,{\startX}}\bigr\|^{2r}\right]\right)^{\frac{1}{2r}}\right]\\
&\leq 
\left[\left[\funcEta(\size{\theta})
\exp \left(e^{\alpha(T-\startT)}\right)\funcUo(\startX)
\right]^{\frac{1}{2r}}\sqrt{\size{\theta}}\right]
\left[2be^{\rho(s-t_0)\betaT /p}(\lyaV({\startX}))^{\betaT /p}\right].
\end{split}\end{align}
This, the triangle inequality,  \eqref{r26},  the fact that $\frac{2\betaT}{p}\leq \frac{1}{2r}$, and the fact that $\lyaV\geq 1$ imply that
\begin{align}\begin{split}
&\sup_{s\in[\startT,T]}\left(\E\!\left[\left\|X_{t_0,s }^{\startX}-Y_{t_0,s }^{\theta,{\startX}}\right\|^r\right]\right)^{\frac{1}{r}} \leq 
\sup_{s\in[\startT,T]}\biggl[
\left(\E\!\left[\left\|\1_{\R^d\setminus{\setD}_{\size{\theta}}}(Y_{\startT,\rdown{T}{\theta}}^{\theta,\startX})(X_{t_0,s }^{\startX}-Y_{t_0,s }^{\theta,{\startX}})\right\|^r\right]\right)^{\frac{1}{r}}\\&\qquad\qquad\qquad\qquad\qquad\qquad\qquad\qquad\qquad+
\left(\E\!\left[\left\|\1_{{\setD}_{\size{\theta}}}(Y_{\startT,\rdown{T}{\theta}}^{\theta,\startX})(X_{t_0,s }^{\startX}-Y_{t_0,s }^{\theta,{\startX}})\right\|^r\right]\right)^{\frac{1}{r}}\biggr]
\\
&\leq \left[\left[\funcEta(\size{\theta})
\exp \left(e^{\alpha(T-\startT)}
\funcUo(\startX)\right)\right]^{\frac{1}{2r}}\sqrt{\size{\theta}}\right]
\left[2be^{\rho(T-t_0)\betaT /p}(\lyaV({\startX}))^{\betaT /p}\right]\\&\quad +72\sqrt{2}(T\vee1)r^{3/2}b^2e^{(2+2\rho\betaT /p)(T-\startT)}\left[
\exp \left(e^{\alpha(T-\startT)}
\funcUo(\startX)\right)\right]^{\frac{1}{2r}}(\lyaV({\startX}))^{2\betaT /p}\left[\funcEta(\size{\theta})\right]^{\frac{1}{4r}}\sqrt{\size{\theta}}\\
&\leq 74\sqrt{2}(T\vee1)r^{3/2}b^2\exp\left((2+\tfrac{\rho}{2r})(T-\startT)\right)\left[\funcEta(\size{\theta})
\exp \left(e^{\alpha(T-\startT)}
\funcUo(\startX)\right)
\lyaV({\startX})\right]^{\frac{1}{2r}}
\sqrt{\size{\theta}}.
\end{split}\end{align}
This completes the proof of \cref{g03}.
\end{proof}

The following corollary extends  \cref{g03} to a formulation
which we later need to apply \cite[Corollary~3.12]{HJKN20}.

\begin{corollary}\label{r17}Assume \cref{g01},
let $\startT\in [0,T)$, $\startX\in\R^d$, $t\in [\startT,T]$, $s\in [t,T]$,
and
let $\epsilon\colon \R\to[0,\infty)$ be the function which satisfies for all $a\in\R$ that
\begin{align}\small
\begin{split}\label{r52}
\epsilon(a)=74\sqrt{2}(T\vee1)r^{3/2}b^2
\exp\left((2+\tfrac{\rho}{2r})(T-\startT)\right)
\left[\funcEta(a)\exp \left(e^{\alpha(T-\startT)}
\funcUo(\startX)\right)\lyaV(\startX)\right]^{\frac{1}{2r}}
\sqrt{|a|}.
\end{split}\end{align}
 Then it holds for all $\theta\in \setCalP{\startT}{T}$ that
\begin{align}\label{r50}
\left(\E\!\left[\left\|X_{t,s }^{ Y^{{\theta},\startX}_{t_0,t}}-X_{t,s }^{ X_{t_0,t}^{\startX}}\right\|^r\right]\right)^{\frac{1}{r}} \leq   \epsilon(\size{\theta}).
\end{align}
\end{corollary}
\begin{proof}[Proof of \cref{r17}]
First, we consider the trivial cases $t=\startT$ and $t=T$.
Observe that \eqref{g01a} and \eqref{g01bb} imply that $ Y^{{\theta},\startX}_{\startT,\startT}=X_{\startT,\startT}^{\startX}=\startX$.
This shows 
in the case $t=\startT$ that \eqref{r50} is trivially true. 
Next, \cref{g03} and \eqref{r52} imply for all $\theta\in \setCalP{\startT}{T}$ that
\begin{align}\label{r54}
\left(\E\!\left[\left\|Y_{\startT,T}^{{\theta},\startX}-X^{\startX}_{\startT,T}\right\|^r\right]\right)^{\nicefrac{1}{r}}\leq \epsilon(\size{\theta}).
\end{align}
This shows in the case $t=T$ that \eqref{r50} holds.
For the rest of the proof we assume that $t\in (\startT,T)$.
The fact that the Brownian motion has independent increments, \eqref{g01a}, and \eqref{g01bb}  imply   that
\begin{align}\label{r55}\color{blue}\begin{split}
&\sigmaAlgebra\bigl(\bigl\{X_{t,s}^x,Y^{\eta,x}_{t,s} \colon
\eta\in\setCalP{t}{T},x\in\R^d
 \bigr\}\bigr)
\text{ and }
\sigmaAlgebra\bigl(\bigl\{X_{t_0,t}^x,Y^{\eta,x}_{t_0,t} \colon
\eta\in\setCalP{t}{T},x\in\R^d\bigr\}\bigr)\\&
\text{are independent.}
\end{split}\end{align}
This combined with disintegration (see, e.g., \cite[Lemma~2.2]{HJK+18}), \eqref{r40c},
 the flow property, \eqref{r54},  and the fact that
$\epsilon|_{[0,\infty)}$ is non-decreasing 
imply that
for all
$\theta,\tilde{\theta}\in \setCalP{\startT}{T}$,
$\eta \in \setCalP{t}{T}$,
$n\in\N$,
$k,\ell\in [0,n]\cap\N_0$,
 $t_1,t_2,\ldots,t_n,s_1,s_2,\ldots,s_\ell\in\R$ with $\theta= (\startT,t_1,\ldots,t_n)$, 
$\eta= (t,s_1,s_2,\ldots,s_\ell)$,
$t_k< t\leq t_{k+1}$,
$\tilde{\theta}= (\startT,t_1,\ldots,t_k,t,s_1,s_2,\ldots,s_\ell)$
it holds that
$\size{\tilde{\theta}}\leq \size{\theta}\vee\size{\eta}$ and
\begin{align}\begin{split}
&\left(\E\!\left[\left.\E\!\left[\left\|Y_{t,s }^{\eta,\tilde{y}}-X_{t,s }^{\tilde{x}}\right\|^r\right]
\right|_{\tilde{x}=X_{t_0,t}^{\startX},\tilde{y}= Y^{{\theta},\startX}_{t_0,t}}\right]
\right)^{\nicefrac{1}{r}}
=\left(\E\!\left[\left\|Y_{t,s }^{{\eta}, Y^{{\theta},\startX}_{t_0,t}}-X_{t,s }^{ X_{t_0,t}^{\startX}}\right\|^r\right]\right)^{\nicefrac{1}{r}}\\
&
=
\left(\E\!\left[\left\|Y_{\startT,T}^{\tilde{\theta},\startX}-X^{\startX}_{\startT,T}\right\|^r\right]\right)^{\nicefrac{1}{r}}
\leq \epsilon(\size{\tilde{\theta}})\leq 
\epsilon(\size{\theta}\vee\size{\eta}).\label{r51}
\end{split}\end{align}
This, \eqref{r55}
combined with a basic result on disintegration (see, e.g., Lemma 2.2 in~\cite{HJK+18}),
\eqref{r54},
Fatou's lemma, continuity of $\epsilon$, and
the fact that $\epsilon(0)=0$ imply for all $\theta\in \setCalP{\startT}{T}$ that
\begin{align}
\begin{split}
&\left(\E\!\left[\left\|X_{t,s }^{Y^{{\theta},\startX}_{t_0,t}}-X_{t,s }^{ X_{t_0,t}^{\startX}}\right\|^r\right]\right)^{\nicefrac{1}{r}}=
\left(\E\!\left[\left.\E\!\left[\left\|X_{t,s }^{\tilde{y}}-X_{t,s }^{\tilde{x}}\right\|^r\right]
\right|_{\tilde{x}=X_{t_0,t}^{\startX},\tilde{y}= Y^{{\theta},\startX}_{t_0,t}}\right]
\right)^{\nicefrac{1}{r}}\\
&=
\left(\E\!\left[\lim_{\nu\in \setCalP{t}{T},\size{\nu}\to0}\left.\E\!\left[\left\|Y_{t,s }^{\nu,\tilde{y}}-X_{t,s }^{\tilde{x}}\right\|^r\right]
\right|_{\tilde{x}=X_{t_0,t}^{\startX},\tilde{y}= Y^{{\theta},\startX}_{t_0,t}}\right]
\right)^{\nicefrac{1}{r}}\\
&
\leq\liminf_{\nu\in \setCalP{t}{T},\size{\nu}\to0}
\left(\E\!\left[\left.\E\!\left[\left\|Y_{t,s }^{\nu,\startX}-X_{t,s }^{\startX}\right\|^r\right]
\right|_{\tilde{x}=X_{t_0,t}^{\startX},\tilde{y}= Y^{{\theta},\startX}_{t_0,t}}\right]
\right)^{\nicefrac{1}{r}}\\
& \leq\liminf_{\nu\in \setCalP{t}{T},\size{\nu}\to0} 
\epsilon(\size{\theta}\vee\size{\nu}) =\epsilon(\size{\theta}).
\end{split}
\end{align}
 This implies \eqref{r50} and  completes the proof of \cref{r17}.
\end{proof}

\subsection{A Lyapunov-type function }
The following lemma, Lemma \ref{s28}, shows that
the functions $V\defeq(\R^d\ni x\mapsto(\|x\|^2+d)^{p+1}$, $d,p\in\N$,
are suitable Lyapunov-type functions.
\begin{lemma}\label{s28}
Consider the notation in~\cref{s05b}, let $d,m\in\N$,
$p\in [3/2,\infty)$, $a,c\in(0,\infty)$,  and
let 
$\mu\colon\R^d\to\R^d$,
$\sigma=(\sigma_1,\sigma_2,\ldots,\sigma_m)\colon \R^ {d\times m}\to\R$,
$\varphi\colon \R^d\to \R $ 
 satisfy for all $x\in\R^d$ that
\begin{align}
\varphi(x)= \left({a} +\|x\|  ^2 \right)^p\quad\text{and}\quad
\langle \mu(x),x \rangle  
+\tfrac{1}{2} (2p-1)\|\sigma(x)\|_{\HS}^2\leq c \varphi(x)^{1/p}.\label{s31}
\end{align}
Then 
it holds for all $x\in\R^d$, $i\in\{1,2,3\}$ that
$
\|(\totalD^i \varphi)(x)\|_{\mlinear{i}{\R^d}{\R}}   \leq (2p)^i\varphi(x)^{1-\frac{i}{2p}}
$ and
$
((\totalD \varphi)(x))(\mu(x))+\frac{1}{2}\sum_{k=1}^{m} ((\totalD^2 \varphi) (x))(\sigma_k(x),\sigma_k(x))\leq 2cp \varphi(x)
$.
\end{lemma}
\begin{proof}[Proof of \cref{s28}]First, 
\eqref{s31} and the Cauchy-Schwarz inequality
 show  that for all  $x,u,v,w\in\R^d$ it holds that
 $ \|x\|  \leq (\varphi(x))^{\frac{1}{2p}}$,
\begin{gather}\small
\begin{split}
[(\totalD \varphi)(x)](u)&= 
p\left[{a} +\|x\|  ^2 \right]^{p-1}2\langle x,u\rangle=2p(\varphi (x))^{\frac{p-1}{p}}\langle x,u\rangle\leq 
2p(\varphi (x))^{\frac{p-1}{p}}\|x\|\|u\|,\label{s32}
\end{split}
\\\small
\begin{split}
[(\totalD^2 \varphi)(x)](u,v)
&= 2p(p-1)\left[{a} +\|x\|  ^2 \right]^{p-2}2\langle x,v\rangle\langle x,u\rangle+2p\left[{a} +\|x\|  ^2 \right]^{p-1}\langle u,v\rangle,
\\\small
&\leq  4p(p-1)(\varphi(x))^{\frac{p-2}{p}}\|x\|^2\|u\|\|v\|+2p(\varphi (x))^{\frac{p-1}{p}}\|u\|\|v\|,\quad\text{and}\quad
\end{split}\\\small
\begin{split}
&[(\totalD^3V)(x)](u,v,w)= 
4p(p-1)(p-2)\left[{a} +\|x\|  ^2 \right]^{p-3}
2\langle x,w \rangle\langle x,v \rangle\langle x,u \rangle
\\&+4p(p-1)\left[{a} +\|x\|  ^2 \right]^{p-2}\bigl(\langle w,u \rangle\langle x,v\rangle  + \langle x,u\rangle \langle w,v\rangle\bigr)
+2p(p-1)\left[{a} +\|x\|  ^2 \right]^{p-2}2\langle x,w\rangle\langle u,v\rangle\\
&\leq 
8p(p-1)(p-2)(\varphi(x))^{\frac{p-3}{p}}
\|x\|^3\|w\|\|v\|\|u\|
\\&\quad+4p(p-1)(\varphi(x))^{\frac{p-2}{p}}\bigl(\|w\|\|u\|\|x\|\|v\| + \| x\|\|u\|\| w\| \|v\|\bigr)\\&\quad
+8p(p-1)(\varphi(x))^{\frac{p-2}{p}}2\| x\|\|w\|\| u\|\|v\|.
\end{split}
\end{gather}
This, \eqref{s36}, the triangle inequality, 
 and the assumption that $p\geq 3/2$
 imply for all  $x\in\R^d$  that
\begin{gather}
\begin{gathered}
\begin{split}
\|(\totalD \varphi)(x)\|   &\leq  2p (\varphi(x))^{\frac{p-1}{p}}\|x\|  \leq 2p (\varphi(x))^{1-\frac{1}{2p}},\end{split}
\\[3pt]
\begin{split}
\left\|(\totalD^2 \varphi)(x)\right\|_{\mlinear{2}{\R^d}{\R}}
&\leq 
(4p(p-1)+2p)(\varphi(x))^{1-\frac{2}{2p}}
\leq  2p(2p-1) (\varphi(x))^{1-\frac{2}{2p}},
\quad\text{and}\quad
\end{split}
\\[3pt]
\begin{split}
&\left\|(\totalD^3 \varphi)(x)\right\|_{\mlinear{3}{\R^d}{\R}}
 \leq 
\left[
8p(p-1)((p-2)\vee 0)+12p(p-1)\right](\varphi(x))^{1-\frac{3}{2p}}\\
&
\leq  p(p-1) [ 8(p-2)\vee 0+12 ](\varphi(x))^{1-\frac{3}{2p}}
\leq 
p(p-1) 8p(\varphi(x))^{1-\frac{3}{2p}}
\leq  (2p)^3(\varphi(x))^{1-\frac{3}{2p}}.
\end{split}\label{s29a}
\end{gathered}\end{gather}
This,  \eqref{s32}, and \eqref{s31}
 prove for all $x\in\R^d$ that
\begin{align}\small\begin{split}
&((\totalD \varphi)(x))(\mu(x))+\frac{1}{2}\sum_{k=1}^{m} ((\totalD^2 \varphi) (x))(\sigma_k(x),\sigma_k(x))\\
&\leq  2p(\varphi(x))^{\frac{p-1}{p}}\langle \mu(x),x \rangle  +\frac{1}{2}\|(\totalD^2V)(x)\|_{\mlinear{2}{\R^d}{\R}}\sum_{k=1}^{m}\|\sigma_k(x)\|_{\HS}^2\\
&= 2p(\varphi(x))^{\frac{p-1}{p}}
\left[\langle \mu(x),x \rangle  
+\frac{1}{2}(2p-1) \|\sigma(x)\|_{\HS}^2
\right]\leq 2p(\varphi(x))^{\frac{p-1}{p}} c(\varphi(x))^{\frac{1}{p}}= 2cp \varphi(x).
\end{split}\end{align}
This and \eqref{s29a}
complete the proof of \cref{s28}.
\end{proof}

\section{Error analysis for MLP approximations}
\renewcommand{\lyaV}{\varphi}
A central assumption of Theorem \ref{m01} below
is the local monotonicity condition \eqref{w01} where
$U,\bar{U}$ satisfy the exponential integrability condition \eqref{eq:exponential.integrability}.
These conditions are satisfied by many interesting SDEs from applications; see, e.g. \cite[Chapter 4]{CHJ14}.
The two main steps of the proof of Theorem~\ref{m01}
are to (a)
apply \cite[Corollary~3.12]{HJKN20} with the forward process given by  \eqref{v01}
to obtain \eqref{eq:V.v}
and to (b) apply \cite[Lemma 2.3]{HJKN20} to obtain the distance \eqref{v16}
between the exact PDE solution and the solution of the stochastic fixed-point equation
with respect to~\eqref{v01}.

\begin{theorem}
\newcommand{\bfdelta}{\boldsymbol{\delta}}
\label{m01}
Consider the notation in~\cref{s05b},
let $d,m\in\N$,
$T\in  (0,\infty)$, 
  $b,c,\betaT,\gamma\in [1,\infty)$,  $\LipConstF,\alpha\in[0,\infty)$,
$p\in[8\betaT,\infty)$, 
$\funcUi\in C(\R^{d},[0,\infty))$, 
$ 
\kappaA\in (0, p/(3\betaT+1)]$,
$\lyaV\in C^3(\R^d, [1,\infty))$,
$\funcUo\in C^3(\R^{d},[0,\infty))$,
$\funcG\in C( \R^d,\R)$, let
$\smallF\colon [0,T]\times \R^{d}\times \R \to\R$ be 
$\mathcal{B}([0,T]\times \R^{d}\times \R)/\mathcal{B}(\R) $-measurable,
let 
$\mu\in C( \R^{d} ,\R^{d})$, $\sigma=(\sigma_1,\sigma_2,\ldots,\sigma_m)\in C( \R^{d},\R^{d\times m})$,
let $\funcF\colon\R^{[0,T]\times\R^d}\to\R^{[0,T]\times\R^d}$ satisfy for all
$t\in [0,T]$,
$x\in \R^d$,
 $w\in \R^{[0,T]\times\R^d}$ that
\begin{align}\label{r40e}
(\funcF(w))(t,x)= f(t,x,w(t,x))
\end{align}
let
$({\setD}_{h})_{h\in(0,T]}\subseteq \mathcal{B}(\R^{d})$, 
let $\rho \in\R$ satisfy that
$\rho=(5c^{1+\frac{1}{\kappaA}}p)^{3p}$,
assume  for all $\ell\in\{1,2,3\}$, $x,y\in\R^d$, $t\in[0,T]$,  $w_1,w_2\in\R $  that
{\allowdisplaybreaks
\begin{gather}\label{v05}
\left\|(\totalD^\ell \lyaV)(x)\right\|_{\mlinear{\ell}{\R^d}{\R}}
\leq c| \lyaV(x)|^{1-\frac{\ell }{p}},
\\
\label{v06}
(\totalD \lyaV(x))(\mu(x))+\frac{1}{2} \sum_{k=1}^{m} (\totalD^2\lyaV(x))(\sigma_k(x),\sigma_k(x))\leq c\lyaV(x),
\\
\label{v04}
\max\left\{ |Tf(t,x,0)|,|g(x)|,\|\mu(x)\|, \|\sigma(x)\|^2, \|x\|\right\}\leq b(\lyaV(x))^{\frac{\betaT }{p}},
\\
\label{v07}
|f(t,x,w_1)-f(t,y,w_2)|\leq L|w_1-w_2|+ b T^{-\nicefrac{3}{2}}
(\lyaV(x)+\lyaV(y))^{\frac{\betaT }{p}}
\|x-y\|,
\\
|g(x)-g(y)|\leq b T^{-\nicefrac{1}{2}}
(\lyaV(x)+\lyaV(y))^{\frac{\betaT }{p}}
\|x-y\|\label{v07b}
\\
\max\left\{
 \|\mu(x)-\mu(y)\|,
\|\sigma(x)-\sigma(y)\|\right\}
\leq b\|x-y\|\left[(\lyaV(x))^{\frac{\betaT }{p}}+(\lyaV(y))^{\frac{\betaT }{p}}\right],\label{v03}
\\
\begin{split}\label{w01}
&{\left\langle x-y, \mu(x)-\mu(y)\right\rangle+3 \|\sigma(x)-\sigma(y)\|_{\HS}^2}\\&
\leq {\|x-y\|^2}
 \left[\frac{\funcUo(x)+\funcUo(y)}{16T}+\frac{\funcUi(x)+\funcUi(y)}{16}\right],
\end{split}
\\\label{w02}
\left\|(\totalD^\ell \funcUo)(x)\right\|_{\mlinear{\ell}{\R^d}{\R}}\leq c| \funcUo(x)|^{1-\frac{1 }{c}},\quad 
|\funcUi(x)|\leq c(1+|\funcUo(x)|^\gamma),
\\
|\funcUi (x)-\funcUi (y)|\leq 
c\left[1+|\funcUo (x)|^\gamma+|\funcUo (y)|^\gamma\right]\|x-y\|,\quad\text{and}\quad
\\
\begin{split}\label{eq:exponential.integrability}
(\totalD \funcUo (x))(\mu (x))&+\frac{1}{2} \sum_{k=1}^{m} (\totalD^2\funcUo (x))(\sigma_{k}(x),\sigma_{k}(x))\\&+
\frac{1}{2}e^{\alpha T}\left\|\sigma (x)^*(\nabla \funcUo) (x)\right\|^2+\funcUi (x)\leq \alpha\funcUo (x),
\end{split}
\end{gather}}%
 assume for all 
$h\in (0,T]$, $x\in {\setD}_{h}$
 that 
\begin{align}
\max\left\{\|\mu (x)\|,\|\sigma (x)\|,\funcUo (x),b^{3\kappaA}\lyaV(x)\right\}\leq
\min\Big\{ch^{-\frac{1}{28}},ch^{-\frac{1}{8+8\gamma}},ch^{-\kappaA},2\log\big(\tfrac{1}{h}\big)\Big\},
\end{align} 
let $\setCalP{0}{T}$ be the set which satisfies that
 \begin{align}
\setCalP{0}{T}=\{(t_0,t_1,\ldots,t_n)\in\R^{n+1}\colon n\in\N,0=t_0<t_1<\ldots<t_n=T\},
\end{align}
for every 
$\delta\in \setCalP{0}{T}$, 
$n\in\N$,
$(t_0,t_1,\ldots,t_n)\in\R^{n+1}$
 with $\delta=(t_0,t_1,\ldots,t_n)$ let
$\size{\delta}\in(0,T] $, $\rdown{\cdot}{\delta}\colon [t_0,t_n]\to \R $ satisfy 
for all $t\in(t_0,t_n]$
 that 
\begin{align}
 \size{\delta} =\max_{i\in [0,n-1]\cap\N_0} |t_{i+1}-t_i|,\quad
\rdown{t_0}{\delta}=t_0,
\quad\text{and}\quad\rdown{t}{\delta}=\sup(\{t_0,t_1,\ldots,t_n\}\cap[t_0,t)),
\end{align}
let $(\bfdelta_n)_{n\in\N}\subseteq\setCalP{0}{T}$ satisfy for all $n\in\N$ that 
\begin{align}\label{r02b}
\bfdelta_n=(0,1,\ldots,n^n) \tfrac{T}{n^n},
\end{align}
let 
$  \Theta = \bigcup_{ n \in \N } \Z^n$,
let $(\Omega,\mathcal{F},\P, (\F_t)_{t\in[0,T]})$ be a filtered probability space which satisfies the usual conditions,
let $\unif^\theta\colon \Omega\to[0,1]$, $\theta\in \Theta$, be independent random variables which are uniformly distributed on $[0,1]$, 
let $\uniform^\theta\colon [0,T]\times \Omega\to [0, T]$, $\theta\in\Theta$, satisfy 
for all $t\in [0,T]$, $\theta\in \Theta$ that 
$\uniform^\theta _t = t+ (T-t)\unif^\theta$,
 let $W^\theta\colon [0,T]\times\Omega \to \R^{m}$, $\theta\in\Theta$, be  standard $(\F_{t})_{t\in[0,T]}$-Brownian motions with continuous sample paths,
assume that $\sigmaAlgebra(\{\unif^\theta\colon\theta\in\Theta\})$ and
$\sigmaAlgebra(\{W^\theta_t\colon\theta\in\Theta,t\in[0,T]\})$ are independent,
for every 
$\theta\in\Theta$,
$t\in[0,T)$,  $x\in\R^d$,
$\delta\in \setCalP{0}{T}$  let
$
(Y^{\delta,\theta}_{t,s}(x,\omega))_{s\in[t,T],\omega\in\Omega}\colon[t,T]\times\Omega\to\R^d$ 
 satisfy  for all $s\in(t,T]$  that $Y_{t,t}^{\delta,\theta}(x)=x$ and 
\begin{align}\small\begin{split}\label{v01}
&Y_{t,s}^{\delta,\theta}(x)\\
& = \left[z + \1_{{\setD}_{\size{\delta\cup\{t\}}}}(z)\left(\mu(z)(s-\rdown{s}{\delta\cup\{t\}})+
\frac{\sigma(z)(W^\theta_{s} -W^\theta_{\rdown{s}{\delta\cup\{t\}}} )}{1+\left\|\sigma(z)(W_{s}^\theta -W^\theta_{\rdown{s}{\delta\cup\{t\}}} )\right\|^2}\right)\right]\Biggr|_{\begin{matrix}
z=Y_{t,\rdown{s}{\delta\cup\{t\}}}^{\delta,\theta}(x)
\end{matrix}},
\end{split}\end{align}
let
$ 
  {\bigV}_{ n,M}^{\delta,\theta} \colon [0, T] \times \R^d \times \Omega \to \R
$, $n,M\in\Z$, $\theta\in\Theta$, $\delta\in\setCalP{0}{T}$, satisfy
for all $n,M \in \N$, $\theta\in\Theta $, $\delta\in\setCalP{0}{T}$,
$ t \in [0,T]$, $x\in\R^d $
that ${\bigV}_{-1,M}^{\delta,\theta}(t,x)={\bigV}_{0,M}^{\delta,\theta}(t,x)=0$ and 
\begin{equation}\label{v02}
\begin{split}
  &{\bigV}_{n,M}^{\delta,\theta}(t,x)
=
  \frac{1}{M^n}
 \sum_{i=1}^{M^n} 
      \funcG \!\left(\sppr^{\delta,(\theta,0,-i)}_{t,T}(x)\right)
 \\
&  +
  \sum_{\ell=0}^{n-1} \frac{(T-t)}{M^{n-\ell}}
    \left[\sum_{i=1}^{M^{n-\ell}}
      \bigl(\funcF\bigl({\bigV}_{\ell,M}^{\delta,(\theta,\ell,i)}\bigr)-\1_{\N}(\ell)\funcF\bigl( {\bigV}_{\ell-1,M}^{\delta,(\theta,-\ell,i)}\bigr)\bigr)
      \Bigl(\uniform_t^{(\theta,\ell,i)},\sppr_{t,\uniform_t^{(\theta,\ell,i)}}^{\delta,(\theta,\ell,i)}(x)\Bigr)
    \right],
\end{split}
\end{equation}
let 
$\bigl(\FE^\delta_{n,M}\bigr)_{\delta\in \setCalP{0}{T},n\in\Z,M\in \N}$ satisfy for all $\delta\in\setCalP{0}{T}$, $n, M\in \N$ that
$\FE^{\delta}_{0,M}=0$ and
\begin{align}
\FE^{\delta}_{n,M}\leq \left(2+
\left\lceil \frac{T}{\size{\delta}} \right\rceil\right)M^n+\sum_{\ell=0}^{n-1}\left[M^{n-\ell}
\Biggl(3+
\left\lceil \frac{T}{\size{\delta}} \right\rceil+\FE^{\delta}_{\ell,M}+\1_{\N}(\ell)\FE^{\delta}_{\ell-1,M}
\Biggr)
\right],\label{c01}
\end{align} 
and
let $\funcEta\colon\R\to\R$ satisfy for all $a\in \R$ that
\begin{align}
\funcEta(a)=\exp \left(
\exp\left( 2\left[720 \max\{T,\alpha,1\}(ce^{\alpha T})^3 \right]^{ (720(ce^{\alpha T})^3\max\{T,1\}+7)\gamma}\right)|\min\{|a|,1\}|^{1/8}\right).\label{v13}
\end{align}
Then
\begin{enumerate}[i)]
\item \label{k20}
for every $t\in[0,T]$,  $x\in\R^d$, $\theta\in\Theta$
there exists an $(\F_{s})_{s\in[t,T]}$-adapted stochastic process
  $
(X_{t,s}^{\theta}(x,\omega))_{s\in[t,T],\omega\in\Omega}\colon[t,T]\times\Omega\to\R^d$ 
with continuous sample paths
which satisfies that for all $s\in[t,T]$ it holds  $\P$-a.s. that
\begin{align}
X_{t,s}^\theta(x)= x+\int_{t}^s \mu(X_{t,r}^\theta(x))\,dr +\int_t^s \sigma(X_{t,r}^\theta(x))\,dW_r^\theta,
\end{align}
\item \label{k21}there exists a unique
$\mathcal{B}([0,T]\times\R^d)/\mathcal{B}(\R)$-measurable function 
$u\colon [0,T]\times\R^d\to\R$ 
which satisfies that
$\sup_{t\in[0,T],x\in\R^d}[u(t,x)(\lyaV(x))^{-\betaT/p}]<\infty$ and 
 which satisfies 
 for all $t\in[0,T]$, $x\in\R^d$  that
$
\E\bigl[|g(X^0_{t,T}(x))|\bigr]+\int_{t}^{T}\E[|f(s,X_{t,s}^0(x),u(s,X_{t,s}^0(x)))| ]ds<\infty
$
and
\begin{align}
u(t,x)=\E\!\left[g(X^0_{t,T}(x))+\int_{t}^{T}f(s,X_{t,s}^0(x),u(s,X_{t,s}^0(x)))\,ds \right],
\end{align}
\item \label{k22}it holds for all $t\in[0,T]$, $x\in\R^d$, $\delta\in\setCalP{0}{T}$, $n,M\in\N$ that
\begin{align}\begin{split}\label{eq:error}
&\left(\E\!\left[|{\bigV}_{n,M}^{\delta,0}(t,x)-\smallU(t,x)|^2\right]\right)^{\!\nicefrac{1}{2}} \\
&\leq 
\left[\frac{e^{M/2}e^{2nLT}}{M^{n/2}}+\frac{\size{\delta}^{\nicefrac{1}{2}}}{T^{\nicefrac{1}{2}}}\right]
1184b^3e^{3T+\rho T+3LT(\funcEta(\size{\delta}))^{\nicefrac{1}{4}} +e^{\alpha T}U(x)/4}
(\lyaV(x))^{\frac{1}{2}},
\end{split}\end{align}
and 
\item\label{k23}there exist $(\mathsf{n}(\epsilon,x))_{x\in\R^d,\epsilon\in (0,1)}\subseteq\N$ such that
 for all 
 $\epsilon,\gamma\in (0,1)$,
$n\in[\mathsf{n}(\epsilon,x),\infty)\cap\N$,
 $t\in[0,T]$, $x\in\R^d$ it holds that $
\E\!\left[\bigl|{\bigV}_{n,n}^{\bfdelta_n,0}(t,x)-\smallU(t,x)\bigr|^2\right]<\epsilon^2
$ and  
\begin{align}\begin{split}\label{eq:effort}
&\left[\sum_{k=1}^{\mathsf{n}(\epsilon,x)}
\FE^{\bfdelta_k}_{k,k}\right]\epsilon^{4+\gamma}\\&\leq
 \sup_{n\in\N}\left[\frac{\left(5^ne^{2ncT}\right)^{\gamma+4}}{n^{\gamma n/2}}\right]
\left[10^4b^3e^{3T+\rho T+3LT(\funcEta(1))^{\nicefrac{1}{4}} +e^{\alpha T}U(x)/4}
(\lyaV(x))^{\frac{1}{2}}\right]^{\gamma+4}.
\end{split}\end{align}

\end{enumerate}
\end{theorem}

\begin{remark}
We discuss the 
computational effort of the approximation method in \cref{v02}.
For every $n\in\N_0$, $M\in \N$, $\delta\in \setCalP{0}{T}$, $t\in[0,T]$, $x\in \R^d$, $\theta\in\Theta$ we think of $\FE^{\delta}_{n,M}$ as an upper bound for the number of function evaluations which are required to compute one realization of $V_{n,M}^{\delta,\theta}(t,x)$ in \eqref{v02}. Then \eqref{c01} can be explained as follows.
For  every $n,M\in \N$, 
$i\in \{1,2,\ldots,M^n\}$,
$\delta\in \setCalP{0}{T}$, $t\in[0,T]$, $x\in \R^d$, $\theta\in\Theta$
we need no more than $(1+
\left\lceil {T}/{\size{\delta}} \right\rceil)$ function evaluation to get one realization of
$
\sppr^{\delta,(\theta,0,-i)}_{t,T}(x)$ and therefore
we need no more than 
$(2+
\left\lceil {T}/{\size{\delta}} \right\rceil)$
function evaluations to get one realization of
$g(
\sppr^{\delta,(\theta,0,-i)}_{t,T}(x))$.
Next, for  every $n,M\in\N$,
$\ell\in \{0,1,\ldots,n \}$,
$i\in \{1,2,\ldots,M^{n-\ell}\}$,
$\delta\in \setCalP{0}{T}$, $t,s\in[0,T]$, $x,y\in \R^d$, $\theta\in\Theta$  
we need (see \eqref{v01}) no more than
$
1+\lceil{T}/{\size{\delta}}\rceil$ function evaluations  to 
get one realization of
$ \sppr_{t,\uniform_t^{(\theta,\ell,i)}}^{\delta,(\theta,\ell,i)}(x)$, 
we need (see \eqref{v02}) no more than
$\FE^{\delta}_{\ell,M}$ to get one realization of
${\bigV}_{\ell,M}^{\delta,(\theta,\ell,i)} (s,y)$, 
we need (see \eqref{v02}) no more than
$
\1_{\N}(\ell)\FE^{\delta}_{\ell-1,M}$ function evaluations
to get one realization of
$ {\bigV}_{\ell-1,M}^{\delta,(\theta,-\ell,i)}(s,y)$,
and therefore we need no more than
$
3+
\left\lceil {T}/{\size{\delta}} \right\rceil+\FE^{\delta}_{\ell,M}+\1_{\N}(\ell)\FE^{\delta}_{\ell-1,M}
$ function evaluations
to get one realization of
$ (\funcF\bigl({\bigV}_{\ell,M}^{\delta,(\theta,\ell,i)})-\1_{\N}(\ell)\funcF( {\bigV}_{\ell-1,M}^{\delta,(\theta,-\ell,i)}))
      (\uniform_t^{(\theta,\ell,i)},\sppr_{t,\uniform_t^{(\theta,\ell,i)}}^{\delta,(\theta,\ell,i)}(x))$. 

\end{remark}

\begin{proof}[Proof of \cref{m01}]
\newcommand{\bfdelta}{\boldsymbol{\delta}}
\sloppy
Throughout this proof let 
$\Delta\subseteq [0,T]^2$ be the set which satisfies that $\Delta= \{(t,s)\in[0,T]^2\colon t\leq s\}$ and
let $\lyaPsi\in C([0,T]\times\R^d,[0,\infty))$ be the function which
 satisfies for all  $t\in[0,T]$, $x\in\R^d$ that
\begin{align}\label{v11}
\lyaPsi(t,x)=\left[ \exp \left(e^{\alpha(T-t)}
\funcUo(x)\right)\right]^ {\frac{1}{2}} \left[e^{\rho (T-t)}\lyaV(x)\right]^{\frac{1}{2}}.
\end{align}
First, 
\eqref{v03}  and the assumption that 
$\phi\in C(\R^d,[0,\infty))$ prove
 that $\mu$ and $\sigma$ are locally Lipschitz continuous. This, 
\eqref{v06}, \eqref{v04},
and a standard result on SDEs with locally Lipschitz continuous coefficients
(see, e.g., \cite[Corollary~2.6]{gk96b}) imply \eqref{k20}. 

For the proof of \eqref{k21} we are going to apply \cite[Proposition~2.2]{HJKN20}
and first check the assumptions.
Jensen's inequality,
the fact that $p\in[8\betaT,\infty) $,
$1\leq c\leq \rho$, $\betaT\in [1,\infty)$, and
a combination of the assumption that
$ 
\kappaA\in (0, p/(3\betaT+1)]$,
\eqref{v01},
\eqref{v05}, \eqref{v06}, \eqref{v04}, the fact that $\rho=(5c^{2+\nfrac{1}{\kappaA}}p)^{3p}$,
\cite[Lemma~2.2]{CHJ14} (applied 
for $t\in[0,T)$, 
$s\in[t,T]$, $x\in\R^d$, $\theta\in\Theta$
with 
$T\defeq T-t$,
$O\defeq\R^d$,
$V\defeq( [0,T-t]\times\R^d\ni(s,x)\mapsto\lyaV(x)\in[0,\infty))$, $\alpha\defeq ([0,T-t]\ni s\mapsto c\in [0,\infty) )$, 
$\tau\defeq s-t$,
$X\defeq (X^\theta_{t,t+r}(x))_{r\in[0,T-t]}$ in the notation of \cite[Lemma~2.2]{CHJ14}),
and of
\cref{s05} (applied 
for $\delta\in\setCalP{0}{T}$, $s\in [0,T]$, $x\in\R^d$, $\theta\in\Theta$
with 
$\betaT\defeq \betaT-1$,
$V\defeq \lyaV$,
$\theta\defeq (\delta\cup\{s\})| _{[s,T]}$, $Y^{\theta,x}\defeq Y^{\delta,\theta}(x)$ in the notation of 
\cref{s05})
imply
for all 
$\theta\in\Theta$, 
$\delta\in\setCalP{0}{T}$,
$\mathfrak{X}\in\{X^\theta,Y^{\delta,\theta}\}$,
$t\in[0,T]$, $s\in [t,T]$, $x\in\R^d$
 that
\begin{align}\label{v08}\begin{split}
&e^{\rho \betaT (T-s)/p}\E\!\left[\bigl(\lyaV (\mathfrak{X}_{t,s}(x)\bigr)^{\frac{\betaT}{p}}\right]\leq 
e^{\rho \betaT (T-s)/p}
\left[\E\!\left[\bigl(\lyaV (\mathfrak{X}_{t,s}(x)\bigr)^{\frac{2\betaT}{p}}\right]\right]^{\frac{1}{2}}\\& \leq e^{\rho\betaT(T-s)/p}
\bigl(
\E\bigl[\lyaV \bigl(\mathfrak{X}_{t,s}(x)\bigr)\bigr]\bigr)^{\frac{\betaT}{p}} 
\leq e^{\rho\betaT(T-s)/p}
 e^{\growrate\betaT(s-t)/p}(\lyaV(x))^{\frac{\betaT}{p}}\\
&
= e^{\rho\betaT(T-t)/p}
(\lyaV(x))^{\frac{\betaT}{p}}
.
\end{split}\end{align}
Furthermore, Lemmas~\ref{r33} and~\ref{r17} (applied 
for $\delta\in \setCalP{0}{T}$, $s\in[0,T]$
with $r\defeq 2$, $\theta\defeq (\delta\cup\{s\})|_{[s,T]}$ in the notation of Lemmas~\ref{r33} and~\ref{r17}) combined with the assumptions of this theorem, and the fact that $T\vee1\leq e^T$ show for all $\delta\in \setCalP{0}{T}$, $\mathfrak{X}\in \{Y^{\delta,0},X^0\}$,
$t\in[0,T]$,
$s\in[t,T]$, $r\in[s,T]$, $x\in \R^d$ that
\begin{align}
\E\!\left[
\exp \left(e^{\alpha(T-s)}
\funcUo(\mathfrak{X}_{t,s}(x))\right)\right]\leq 
\funcEta(\size{\delta})
\exp \left(e^{\alpha(T-t)}
\funcUo(x)\right)\label{v09b}
\end{align}
and
\begin{align}
\begin{split}
&
\left(
\E\!\left[\bigl\|X_{s,r}^0\left(X^0_{t,s}(x)\right)-X_{s,r}^0\bigl(Y_{t,s}^{\delta,0}(x)\bigr)\bigr\|^{2}\right]\right)^{\nicefrac{1}{2}}\\&\leq 74\sqrt{2}(T\vee1)2^{3/2}c^2
e^{2T+\rho T/4}
\left[\funcEta(\size{\delta})\exp \left(e^{\alpha(T-t)}
\funcUo(x)\right)\lyaV(x)\right]^{\frac{1}{4}}
\size{\delta}^{\nicefrac{1}{2}}\\
&\leq 296 b^2
e^{3T+\rho T/4}
(\funcEta(\size{\delta}))^{\nicefrac{1}{4}}
\size{\delta}^{\nicefrac{1}{2}}
\left(\psi(t,x)\right)^{\nicefrac{1}{2}}.
\end{split}\label{v09c}\end{align}
Next, continuity of $\mu,\sigma$, path continuity of $W^\theta$, $\theta\in\Theta$, the fact that $ (\setD_h)_{h\in[0,T]}\subseteq\mathcal{B}(\R^d)$, and Fubini's theorem imply for all
$\mathcal{B}([0,T]\times\R^d)/\mathcal{B}([0,\infty)) $-measurable functions
$\eta\colon  [0,T]\times\R^d\to[0,\infty)$ and all
$\delta\in\setCalP{0}{T}$,
$\theta\in\Theta$
that  
\begin{align}
\Delta\times\R^d \ni (t,s,x)\mapsto
\E\bigl[\eta\bigl(s,Y_{t,s}^{\delta,\theta}(x)\bigr)\bigr]\in[0,\infty]
\text{ is }\mathcal{B}(\Delta\times\R^d)/\mathcal{B}([0,\infty])\text{-measurable.}\label{v09d}
\end{align}
Moreover,
local Lipschitz continuity of $\mu,\sigma$, \eqref{v06}, \eqref{v04}, 
and \cite[Lemma~3.7]{BGHJ19} (applied with $\mathcal{O}\defeq\R^d$, $V\defeq(
[0,T]\times\R^d\ni (t,x)\mapsto
e^{-\rho t}\lyaV(x)\in(0,\infty) )$ in the notation of \cite[Lemma~3.7]{BGHJ19}) imply for all $\theta\in\Theta$
and all
$\mathcal{B}([0,T]\times\R^d)/\mathcal{B}([0,\infty)) $-measurable functions
$\eta\colon  [0,T]\times\R^d\to[0,\infty)$ 
that  
\begin{align}
\Delta\times\R^d \ni (t,s,x)\mapsto
\E\bigl[\eta\bigl(s,X_{t,s}^{\theta}(x)\bigr)\bigr]\in[0,\infty]
\text{ is }\mathcal{B}(\Delta\times\R^d)/\mathcal{B}([0,\infty])\text{-measurable.}
\end{align}
This,  \eqref{v09d}, \eqref{v08},
\eqref{v07}, \eqref{v04}, and
\cite[Proposition~2.2]{HJKN20} (applied 
for  $t\in [0,T]$, $x\in\R^d$, $\delta\in\setCalP{0}{T}$
with
$\mathcal{O}\defeq \R^d$, $ (X_{t,s}^x)_{s\in[t,T]}\defeq (Y_{t,s}^{\delta,0}(x))_{s\in[t,T]}
$, $ (X_{t,s}^x)_{s\in[t,T]}\defeq (X_{t,s}^{0}(x))_{s\in[t,T]} $, 
$
V\defeq ([0,T]\times\R^d\ni(t,x)\mapsto e^{\growrate\betaT  (T-t)/p}(\lyaV (x))^{\betaT /p}\in(0,\infty))
$
 in the notation of \cite[Proposition~2.2]{HJKN20})
 establish that 
\begin{enumerate}[a)]
\item 
 there exist unique
$\mathcal{B}([0,T]\times\R^d)/\mathcal{B}(\R)$-measurable functions
$\smallU, \smallV_\delta\colon [0,T]\times\R^d\to\R$, $\delta\in\setCalP{0}{T}$, which satisfy  
that
$\sup_{t\in[0,T],x\in\R^d,\delta\in\setCalP{0}{T}} \left[(|\smallU(t,x)|\vee|\smallV_\delta(t,x)|){(\lyaV(x))^{-\betaT /p}}\right]<\infty$
and which satisfy  for all $\delta\in \setCalP{0}{T}$, $t\in[0,T]$, $x\in\R^d$  that
\begin{align}\begin{split}
&\textstyle \E\!\left[\bigl|g\bigl(Y^{\delta,0}_{t,T}(x)\bigr)\bigr|+\int_{t}^{T}\bigl|f\bigl(s,Y_{t,s}^{\delta,0}(x),\smallV_\delta\bigl(s,Y_{t,s}^{\delta,0}(x)\bigr)\bigr)\bigr|\,ds \right]<\infty,
\\
&\textstyle\E\!\left[\bigl|g\bigl(X^{0}_{t,T}(x)\bigr)\bigr|+\int_{t}^{T}\bigl|f\bigl(s,X_{t,s}^{0}(x),\smallU\bigl(s,X_{t,s}^{0}(x)\bigr)\bigr)\bigr|\,ds \right]<\infty,
\\
&\smallV_\delta(t,x)=\E\!\left[g\bigl(Y^{\delta,0}_{t,T}(x)\bigr)+\int_{t}^{T}f\bigl(s,Y_{t,s}^{\delta,0}(x),\smallV_\delta(s,Y_{t,s}^{0,\delta}(x))\bigr)\,ds \right],
\quad\text{and}\quad\\
&\smallU(t,x)=\E\!\left[g\bigl(X^{0}_{t,T}(x)\bigr)+\int_{t}^{T}f\bigl(s,X_{t,s}^0(x),\smallU(s,X_{t,s}^0(x))\bigr)\,ds \right]
\end{split}\label{v15}\end{align}
and \item it holds for all $\delta\in \setCalP{0}{T}$  that 
\begin{align}\label{v14}\small\begin{split}
\sup_{t\in[0,T],x\in\R^d} \left[\frac{|\smallU(t,x)|\vee|\smallV_\delta(t,x)|}{e^{\growrate \betaT  (T-t)/p}(\lyaV(x))^{\betaT /p}}\right]\leq  
\sup_{t\in[0,T],x\in\R^d}\left[\frac{|g(x)|}{(\lyaV(x))^{\beta/p}}
+\frac{|Tf(t,x,0)|}{(\lyaV(x))^{\beta /p}}\right]e^{L T}\leq 
2be^{LT}.
\end{split}\end{align}
\end{enumerate}
This proves \eqref{k21}. 

For the proof of \eqref{k22} we are going to apply \cite[Proposition~3.12]{HJKN20}
and first check the assumptions.
Note that \eqref{v11}, the Cauchy-Schwarz inequality, \eqref{v09b}, and
\eqref{v08}
 show for all 
$\delta\in\setCalP{0}{T}$,
$\mathfrak{X}\in \{Y^{\delta,0},X^0\}$,
$t\in[0,T]$, $s\in[t,T]$, $x\in\R^d$ that
\begin{align}\begin{split}
&\E\!\left[\lyaPsi(s,\mathfrak{X}_{t,s}(x))\right]
= 
\E\!\left[ \left(\exp \left(e^{\alpha(T-s)}
\funcUo(\mathfrak{X}_{t,s}(x))\right)\right)^{\frac{1}{2}}\left( e^{\rho (T-s)}\lyaV(\mathfrak{X}_{t,s}(x))\right)^{\frac{1}{2}}\right]\\
&\leq  
\left(\E\!\left[
\exp \left(e^{\alpha(T-s)}
\funcUo(\mathfrak{X}_{t,s}(x))\right)\right]
\right)^{\frac{1}{2}}
\left(\E\!\left[e^{\rho (T-s)}\lyaV(\mathfrak{X}_{t,s}(x))\right]
\right)^{\frac{1}{2}}\\
&\leq \left[\funcEta(\size{\delta})
\exp \left(e^{\alpha(T-t)}
\funcUo(x)\right)
\right]^{\frac{1}{2}}
\left(e^{\rho (T-s)}e^{\rho (s-t)}\lyaV(x)
\right)^{\frac{1}{2}}=(\funcEta(\size{\delta}))^{\nicefrac{1}{2}}\lyaPsi(t,x).
\end{split}\label{v17}\end{align}
\sloppy
Lipschitz continuity of $\mu,\sigma$ (see \eqref{v03}) implies 
for all $t\in[0,T]$, $s\in[t,T]$,
$r\in[s,T]$,
 $x\in\R^d$   that
$
\P\circ (\exactpr_{s,r}^{\exactpr_{t,s}^x})^{-1}= \P\circ (\exactpr_{t,r}^x)^{-1}.
$
This combined with
 \eqref{k20}, \eqref{v01},  \eqref{v08},
\eqref{v17}, \eqref{v15},  \eqref{v04}, 
\eqref{v14}, \eqref{v07}, 
 \eqref{v07b},
\eqref{v09c},
 and with
\cite[Lemma~2.3]{HJKN20} (applied 
for $\delta\in\setCalP{0}{T}$
with 
$  \eta\defeq(\funcEta(\size{\delta}))^{\nicefrac{1}{2}}$,
$\delta\defeq   296b^2
e^{3T+{\rho T}/{4}}
(\funcEta(\size{\delta}))^{\nicefrac{1}{4}}
\size{\delta}^{\nicefrac{1}{2}}$, 
$
 p\defeq {p}/ \betaT $,
$ q\defeq 2$,
$\|\cdot\|\defeq\|\cdot\||_{\R^d}$,
$(X_{t,s}^{x,1})_{t\in[0,T],s\in[t,T],x\in\R^d}\defeq (X_{t,s}^0(x))_{t\in[0,T],s\in[t,T],x\in\R^d}$,
$(X_{t,s}^{x,2})_{t\in[0,T],s\in [t,T],x\in\R^d}\defeq (Y_{t,s}^{\delta,0}(x))_{t\in[0,T],s\in [t,T],x\in\R^d}$,
$V\defeq b^{p/\betaT}\lyaV$,
$u_1\defeq u$,
$u_2\defeq v_\delta$  in the notation of \cite[Lemma~2.3]{HJKN20}),
the fact that
$1+LT\leq e^{LT}$, 
a combination of
the fact that 
$\funcEta\geq1$ (see \eqref{v13})
and of the fact that
$\forall\,a\in[1,\infty)\colon ae^{LT}\leq e^{aLT}$, 
\eqref{v11},
and a combination of
the fact that $\lyaV\geq1$ and of
the fact that $ \betaT /{p}\leq 1/4$  imply 
for all $t\in[0,T]$, $x\in\R^d$, $\delta\in\setCalP{0}{T}$ that
\begin{align}\begin{split}
&|\smallU(t,x)-  \smallV_\delta(t,x)|\leq
4(1+LT)T^{-\nicefrac{1}{2}}e^{LT+\nfrac{\rho \betaT  T}{p}+LT(\funcEta(\size{\delta}))^{\nicefrac{1}{4}}}\\&\qquad\qquad\qquad\qquad\qquad\qquad (b^{\frac{p}{ \betaT }}\lyaV(x))^{\frac{ \betaT }{p}}  (\lyaPsi(t,x))^{\frac{1}{2}} \left[
296b^2
e^{3T+\nfrac{\rho T}{4}}
(\funcEta(\size{\delta}))^{\nicefrac{1}{4}}
\size{\delta}^{\nicefrac{1}{2}}\right]\\
&\leq 1184b^3T^{-\nicefrac{1}{2}}
e^{3T+\nfrac{\rho \betaT T}{p}+\nfrac{\rho T}{4}+3LT(\funcEta(\size{\delta}))^{\nicefrac{1}{4}}}
\\&\qquad\qquad\qquad\qquad \qquad\qquad
(\lyaV(x))^{\frac{ \betaT }{p}}  
\exp\left(\tfrac{1}{4}e^{\alpha T}U(x)\right)
e^{\rho T/4}(\lyaV(x))^{\frac{1}{4}}
\size{\delta}^{\nicefrac{1}{2}}\\
&\leq 
1184b^3T^{-\nicefrac{1}{2}}
e^{3T+\rho T+3LT(\funcEta(\size{\delta}))^{\nicefrac{1}{4}}}
(\lyaV(x))^{\frac{1}{2}}
\exp\left(\tfrac{1}{4}e^{\alpha T}U(x)\right)
\size{\delta}^{\nicefrac{1}{2}}.
\end{split}\label{v16}\end{align}
Next, \cite[Corollary~3.12]{HJKN20} (applied 
for $\delta\in\setCalP{0}{T}$, $t\in[0,T)$
with $\rho\defeq 2 \betaT \rho/p$, 
$(Y^\theta_{t,s})_{t\in[0,T],s\in[t,T],\theta\in\Theta}\defeq (Y^{\delta,\theta}_{t,s})_{t\in[0,T],s\in[t,T],\theta\in\Theta}$,
$(U^\theta_{n,M\in\Z})_{n,M\in\Z,\theta\in\Theta}\defeq (\bigV_{n,M\in\Z}^{\delta,\theta})_{n,M\in\Z,\theta\in\Theta}$,
$u\defeq \smallV^{\delta}$,
$\varphi\defeq \lyaV^{2 \betaT /p}$, $\tau\defeq t$ in the notation of \cite[Corollary~3.12]{HJKN20}), \eqref{r40e},
\eqref{v07}, 
\eqref{v15},
 \eqref{v04}, \eqref{v14},
 and  \eqref{v08}  imply for all $\delta\in \setCalP{0}{T}$, $t\in[0,T]$, $n,M\in\N$ that
\begin{align}\begin{split}\label{eq:V.v}
&\sup_{x\in\R^d}\left[\frac{
\E\!\left[|{\bigV}_{n,M}^{\delta,0}(t,x)-\smallV_\delta(t,x)|^2\right]}{(\lyaV(x))^{2 \betaT /p}}\right]^{\!\nicefrac{1}{2}}\leq
2 e^{M/2}M^{-n/2}(1+2TL)^{n-1}e^{\rho \betaT  T/p}
\\& \quad\cdot
\left[\sup_{s\in[0,T],x\in\R^d}\left[\frac{|T(\funcF(0))(s,x)|\vee|g(x)|}{(\lyaV(x))^{ \betaT /p}}\right]+TL\sup_{s\in[0,T],x\in\R^d}\left[
\frac{|\smallV_\delta(s,x)|}{(\lyaV(x))^{ \betaT /p}}\right]\right]\\
&\leq e^{M/2}M^{-n/2}(1+2TL)^{n-1}e^{\rho \betaT  T/p}\left[2b+4TLb e^{LT+\rho \betaT T/p}\right]\\
&\leq  2be^{M/2}M^{-n/2}(1+2TL)^{n} e^{LT+2\rho \betaT T/p}\leq 
2be^{M/2}M^{-n/2}e^{2nLT} e^{LT+2\rho \betaT T/p}.
\end{split}\end{align}
This, the triangle inequality,  \eqref{v16}, the fact that
$2 \betaT \leq p $,
the fact that $b,\funcEta,\lyaV\geq 1$,
and the fact that $\forall\,a\in \R\colon \funcEta(a)\leq \funcEta(1) $ (see \eqref{v13})
 show  for all
$n,M\in\N$, $t\in[0,T]$, $x\in\R^d$, $\delta\in\setCalP{0}{T}$  that
\begin{align}\begin{split}
&\left(\E\!\left[|{\bigV}_{n,M}^{\delta,0}(t,x)-\smallU(t,x)|^2\right]\right)^{\!\nicefrac{1}{2}} \leq 
\left(\E\!\left[|{\bigV}_{n,M}^{\delta,0}(t,x)-\smallV_\delta(t,x)|^2\right]\right)^{\!\nicefrac{1}{2}}+
|\smallU(t,x)-  \smallV_\delta(t,x)|\\
&
\leq 
2b e^{M/2}M^{-n/2}e^{2nLT+LT+2\rho \betaT T/p}
(\lyaV(x))^{ \betaT /p}\\&\quad+1184b^3T^{-\nicefrac{1}{2}}
e^{3T+\rho T+3LT(\funcEta(\size{\delta}))^{\nicefrac{1}{4}}}
(\lyaV(x))^{\frac{1}{2}}
\exp\left(\tfrac{1}{4}e^{\alpha T}U(x)\right)
\size{\delta}^{\nicefrac{1}{2}}\\
&
\leq 
\left({e^{M/2}e^{2nLT}}{M^{-n/2}}+{\size{\delta}^{\nicefrac{1}{2}}}{T^{-\nicefrac{1}{2}}}\right)
1184b^3e^{3T+\rho T+3LT(\funcEta(\size{\delta}))^{\nicefrac{1}{4}} +e^{\alpha T}U(x)/4}
(\lyaV(x))^{\frac{1}{2}}\\
&
\leq 
\left({e^{M/2}e^{2nLT}}{M^{-n/2}}+{\size{\delta}^{\nicefrac{1}{2}}}{T^{-\nicefrac{1}{2}}}\right)
1184b^3e^{3T+\rho T+3LT(\funcEta(1))^{\nicefrac{1}{4}} +e^{\alpha T}U(x)/4}
(\lyaV(x))^{\frac{1}{2}}
.
\end{split}\label{c02}\end{align}
This shows \eqref{k22}.

To prove \eqref{k23}
let $(\mathsf{n}(\epsilon,x))_{\epsilon\in (0,1), x\in\R^d}\subseteq[0,\infty]$ satisfy for all 
 $\epsilon\in (0,1)$, $x\in\R^d$
that 
\begin{align}\label{r33f}\begin{split}
\mathsf{n}(\epsilon,x)= \inf\left(\left\{n\in\N\colon 
\sup_{k\in[n,\infty)\cap\N,t\in[0,T]}
\E\Bigl[\bigl|{\bigV}_{k,k}^{\bfdelta_k,0}(t,x)-\smallU(t,x)\bigr|^2\Bigr]<\epsilon^2
\right\}\cup\{\infty\}\right).
\end{split}\end{align}
The fact that 
$\lim_{n\to\infty}(e^{n/2}e^{2ncT}n^{-n/2})=0$ and
 \eqref{r21} then show for all $x\in\R^d$, $\epsilon\in(0,1)$ that \begin{align}
\mathsf{n}(\epsilon,x)\in\N.\label{r01}
\end{align} Next,
a combination of 
\eqref{c01} and of
\cite[Lemma~3.6]{HJK+18} (applied 
with $d\defeq (2+\lceil T/\size{\delta}\rceil)$, $(\mathrm{RV}_{n,M})_{n,M\in\Z}\defeq (\FE_{n,M})_{n,M\in\Z}$ in the notation of \cite[Lemma~3.6]{HJK+18})
and the fact that
$\forall \,\delta\in \setCalP{0}{T}\colon 2+ \left\lceil {T}/{\size{\delta}} \right\rceil\leq 4T/\size{\delta}$
show  for all $\delta\in\setCalP{0}{T}$, $n,M\in\N$ that
$
\FE^{\delta}_{n,M}\leq \left(2+ \left\lceil {T}/{\size{\delta}}\right\rceil\right) (5M)^n\leq 4T\size{\delta}^{-1}(5M)^n.
$ This, \eqref{r02}, and the fact that $\forall\,n\in\N\colon n^3\leq 3^n$ show 
for all 
$n\in\N$,
$k\in[1,n+1]\cap\N$ that
\begin{align}\begin{split}
2n\FE^{\bfdelta_{k}}_{k,k}
\leq 
8n{T}{\size{\bfdelta_{k}}}^{-1}(5k)^{k}
\leq 
8n(5k^2)^k\leq 8n (20n^2)^{n+1}=160n^320^{n}n^{2n}\leq 
160n^{2n}60^n.\label{r01b}
\end{split}\end{align}
This implies for all 
 $\gamma\in (0,\infty)$,
$n\in\N$, $t\in[0,T]$, $x\in\R^d$
 that
$\sum_{k=1}^{n+1}
\FE^{\bfdelta_{k}}_{k,k}\leq 160n^{2n}60^n$ and 
\begin{align}\begin{split}
& \left[\sum_{k=1}^{n+1}
\FE^{\bfdelta_{k}}_{k,k}\right]\left[
\left(\E\Bigl[\bigl|{\bigV}_{n,n}^{\bfdelta_n,0}(t,x)-\smallU(t,x)\bigr|^2\Bigr]\right)^{\frac{4+\gamma}{2}}\right]\leq 
160n^{2n}60^n\\&\quad\qquad\left[\left({e^{n/2}e^{2nLT}}{n^{-n/2}}+{\size{\bfdelta_n}^{\nicefrac{1}{2}}}{T^{-\nicefrac{1}{2}}}\right)
1184b^3e^{3T+\rho T+3LT(\funcEta(1))^{\nicefrac{1}{4}} +e^{\alpha T}U(x)/4}
(\lyaV(x))^{\frac{1}{2}}\right]^{\gamma+4}\\
&= 
160n^{2n}60^n\left[\left({e^{n/2}e^{2nLT}}{n^{-n/2}}+n^{-n/2}\right)
1184b^3e^{3T+\rho T+3LT(\funcEta(1))^{\nicefrac{1}{4}} +e^{\alpha T}U(x)/4}
(\lyaV(x))^{\frac{1}{2}}\right]^{\gamma+4}\\
&\leq 160n^{-\gamma n/2}60^n
\left[2{e^{n/2}e^{2nLT}}
1184b^3e^{3T+\rho T+3LT(\funcEta(1))^{\nicefrac{1}{4}} +e^{\alpha T}U(x)/4}
(\lyaV(x))^{\frac{1}{2}}\right]^{\gamma+4}
\\
&
\leq n^{-\gamma n/2}
\left[160^{\nicefrac{1}{4}}2{e^{n/2}60^{{n}/{4}}e^{2nLT}}
1184b^3e^{3T+\rho T+3LT(\funcEta(1))^{\nicefrac{1}{4}} +e^{\alpha T}U(x)/4}
(\lyaV(x))^{\frac{1}{2}}\right]^{\gamma+4}\\
&
\leq n^{-\gamma n/2}
\left[5^n e^{2nLT}
10^4b^3e^{3T+\rho T+3LT(\funcEta(1))^{\nicefrac{1}{4}} +e^{\alpha T}U(x)/4}
(\lyaV(x))^{\frac{1}{2}}\right]^{\gamma+4}.
\end{split}\end{align}
This,  \eqref{r01b}, and \eqref{r01}
imply that for all 
 $\epsilon,\gamma\in (0,1)$,
 $t\in[0,T]$, $x\in\R^d$ it holds 
in the case  $\mathsf{n}(\epsilon,x)=1$ that
$
\sum_{k=1}^{\mathsf{n}(\epsilon,x)}
\FE^{\bfdelta_k}_{k,k}\epsilon^{4+\gamma}\leq 
\FE^{\bfdelta_1}_{1,1}\leq 160\cdot 60
$ and it holds
in the case $\mathsf{n}(\epsilon,x)\geq 2$ 
that
\begin{align}\begin{split}
&
\left[\sum_{k=1}^{\mathsf{n}(\epsilon,x)}
\FE^{\bfdelta_k}_{k,k}\right]\epsilon^{4+\gamma}\leq 
\left[
\left[\sum_{k=1}^{n+1}
\FE^{\bfdelta_k}_{k,k}\right]
\left(\E\Bigl[\bigl|{\bigV}_{n,n}^{\bfdelta_n,0}(t,x)-\smallU(t,x)\bigr|^2\Bigr]\right)^{\frac{4+\gamma}{2}}\right]\Biggr|_{n=\mathsf{n}(\epsilon,x)-1}\\
&
\leq \sup_{n\in\N}\left[n^{-\gamma n/2}\left(5^ne^{2ncT}\right)^{\gamma+4}\right]
\left[10^4b^3e^{3T+\rho T+3LT(\funcEta(1))^{\nicefrac{1}{4}} +e^{\alpha T}U(x)/4}
(\lyaV(x))^{\frac{1}{2}}\right]^{\gamma+4}.
\end{split}\end{align}
This, the fact that $b,c,\lyaV\geq 1$, and the fact that $160\cdot 60\leq (10^4)^4$ prove that 
\begin{align}
\left[\sum_{k=1}^{\mathsf{n}(\epsilon,x)}
\FE^{\bfdelta_k}_{k,k}\right]\epsilon^{4+\gamma}\leq
 \sup_{n\in\N}\left[\tfrac{\left(5^ne^{2ncT}\right)^{\gamma+4}}{n^{\gamma n/2}}\right]
\left[10^4b^3e^{3T+\rho T+3LT(\funcEta(1))^{\nicefrac{1}{4}} +e^{\alpha T}U(x)/4}
(\lyaV(x))^{\frac{1}{2}}\right]^{\gamma+4}.
\end{align} 
This, \eqref{r33f}, and \eqref{r01}
imply \eqref{k23}.
The proof of \cref{m01}
is thus completed.
\end{proof}
The error estimate \cref{eq:error} shows that the $L^2$-distance between
our approximations and the PDE solution at a fixed space-time point converges
to $0$ as $n\to\infty$ for every fixed $M>e^{2LT}$.
To get an computational effort of order $4+$ as in \cref{eq:effort}, however,
we need to let $M\to\infty$ as $n\to\infty$.
We typically choose $M=n$. In this case, the error decays up to exponential factors 
like $n^{-n/2}$ (resp.\ like $1/\sqrt{n!}$)
and the number of function evaluations grows up to exponential factors like
$n^{2n}$ (resp.\ like $(n!)^2$) as $n\to\infty$; cf.\ \cref{r01b}.

In the following \cref{e01} we show that
 semilinear PDEs corresponding to competitive
Lotka-Volterra equations can be approximated without curse of
dimensionality under suitable assumptions on the parameters.
We note that the coefficients of \eqref{e02} are not
globally Lipschitz continuous.

\begin{example}\label{e01}
Consider the notation from \cref{s05b},
let $T \in (0,\infty)$, 
$
r = (r_i)_{i\in\N} \in [0,\infty)^{\N}
$,
$a= (a_{ij})_{i,j \in \N} \in [0,\infty)^{\N \times \N}$
satisfy that $\sup_{i \in \N} r_i<\infty$
and $\sup_{i\in \N} \sum_{j\in\N} \lvert a_{ij}\rvert^2 <\infty$, 
for every 
$x\in\R$ let $x^+\in \R$ satisfy that
$x^+=\max\{0,x\}$,
and for every $d \in \N$ let
$u_d \colon [0,T] \times \R^d \to\R$ be a viscosity solution of the~PDE
\begin{align}\label{e02}\begin{split}
&
\left(\frac{\partial}{\partial t}u_d\right)(t,x)+
\sum_{i=1}^{d}\left[
r_i x_i \left(1- \sum_{j=1}^{d}a_{ij}x_j^+ \right)
\left(
\frac{\partial }{\partial x_i}u_d\right)(t,x)\right]
+ \sum_{i=1}^{d}
\frac{\lvert x_i\rvert^2(\frac{\partial^2}{\partial x_i^2}u_d)(t,x)}{2(1+\lVert x\rVert^2)^2}\\
&=
-\sin (\lVert x \rVert+u_d(t,x))
\quad\text{and}\quad u_d(T,x)=\lVert x\rVert^2\quad \forall\,
(t,x)\in[0,T]\times  \R^d.
\end{split}
\end{align}
Then $(u_d(0,0))_{d \in \N}$ can be approximated without suffering from the curse of dimensionality.

To prove this statement we apply \cref{m01} and show that all assumptions are satisfied.
For this  fix $d\in\N$, 
  and
let
$\alpha,\bar{r},\bar{a}\in [0,\infty)$,
$\bar{U}\in C(\R^d,[0,\infty))$,
$\varphi\in C^3(\R^d,[1,\infty))$,
$U\in C^3(\R^d,[0,\infty))$,
$g\in C(\R^d,\R)$,
$f\in C([0,1]\times\R^d\times\R,\R)$,
 $\mu=(\mu_1,\mu_2, \ldots,\mu_d)\in C(\R^d,\R^d)$,
 $\sigma=(\sigma_1,\sigma_2,\ldots,\sigma_d)\in C(\R^d,\R^{d\times d})$
 satisfy for all 
$i,j\in \{1,2,\ldots,d\}$, 
$x=(x_1,x_2,\ldots,x_d)$,
$t\in[0,T] $, $w\in\R$
that  
\begin{align}
\mu_i(x)=  r_i x_i \left(1- \sum_{j=1}^{d}a_{ij}x_j^+ \right),\quad 
\sigma(x)= \mathrm{diag}\left(\frac{x_1}{1+\lVert x\rVert^2} ,\frac{x_2}{1+\lVert x\rVert^2},\ldots,\frac{x_d}{1+\lVert x\rVert^2}\right)
  ,\label{d01}
\end{align}
\begin{align}\label{d03}
f(t,x,w)= \sin (\lVert x \rVert +w),\quad g(x)=\lVert x\rVert^2,
\end{align}
\begin{align}\label{d04}
\varphi(x)=(1+\lVert x\rVert^2)^8,\quad 
U(x)= 8T\left(8T
e^{\alpha T}\bar{r}^2\bar{a}^2+3+\bar{r}\right)+e^{-\alpha T}\lVert x\rVert^2, \quad \bar{U}(x)=0,
\end{align}
\begin{align}\label{d02}
\bar{r}=\sup_{r\in\N}r_i,\quad 
\bar{a}=\left(
\sup_{i\in \N} \sum_{j\in\N} \lvert a_{ij}\rvert^2\right)^{\frac{1}{2}},\quad \alpha = 2\bar{r}+3.
\end{align}
%
First, \eqref{d01} and \eqref{d02} show for all $x=(x_1,x_2,\ldots,x_d)\in\R^d$ that 
\begin{align}
 \langle x,\mu(x)\rangle=\xeqref{d01} \sum_{i=1}^{d} \left(r_i x_i^2 \left(1- \sum_{j=1}^{d}a_{ij}x_j^+ \right)\right)\leq \xeqref{d02}\bar{r}\lVert x\rVert^2
\quad\text{and}\quad\lVert \sigma(x)\rVert \leq\xeqref{d01} \lVert x\rVert
.\label{d06}
\end{align}
This and \eqref{d04}
show for all $x\in\R^d$ that
$
\langle \mu(x),x\rangle+7.5\lVert\sigma (x)\rVert^2
\leq (\bar{r}+7.5)\lVert x\rVert^2\leq 
(\bar{r}+7.5)(\varphi (x))^{\frac{1}{8}}
.
$
This and \cref{s28} (applied with $m\gets d$, $p\gets 8$, $a\gets 1$,
$c\gets\bar{r}+7.5$ in the notation of \cref{s28}) show  for all $x\in\R^d$, $i\in\{1,2,3\}$ that
\begin{align}\label{a01}
\|(\totalD^i \varphi)(x)\|_{\mlinear{i}{\R^d}{\R}}   \leq 16^i\varphi(x)^{1-\frac{i}{16}}
\end{align} and
\begin{align}\label{a02}
((\totalD \varphi)(x))(\mu(x))+\frac{1}{2}\sum_{k=1}^{d} ((\totalD^2 \varphi) (x))(\sigma_k(x),\sigma_k(x))\leq 16(\bar{r}+7.5) \varphi(x)
. 
\end{align}
Furthermore, \eqref{d03}, 
the fact that 
$ \forall\,x,y\in\R\colon \lvert\sin(x)-\sin(y)\rvert\leq \lvert x-y\rvert$, and
the triangle inequality  show for all $t\in[0,T]$, $x,y\in\R^d$, $w_1,w_2\in\R$ that
\begin{align}\begin{split}
&
\lvert
f(t,x,w_1)-f(t,y,w_2)\rvert=\xeqref{d03}\lvert \sin (\lVert x\rVert+w_1)-\sin (\lVert y\rVert+w_2)\rvert\\
&\leq \bigl\lvert (\lVert x\rVert + w_1)- (\lVert y\rVert+w_2)\bigr\rvert\leq \lvert w_1-w_2\rvert+\lVert x-y\rVert.
\end{split}\label{a03}\end{align}
Next, \eqref{d03},  the fact that 
$ \forall\,x\in\R\colon x\leq \frac{1+x^2}{2}$, and 
\eqref{d04}
 show for all $t\in[0,T]$, $x,y\in\R^d$ that
\begin{align}\begin{split}
&
\lvert g(x)-g(y)\rvert
=\xeqref{d03}
\left\lvert\lVert x\rVert^2-\lVert y\rVert^2\right\rvert
= \left
(\lVert x\rVert+\lVert y\rVert\right)\left(\lVert x\rVert-\lVert y\rVert\right)
\\
&
\leq 0.5 (1+\lVert x\rVert^2+1+\lVert y\rVert^2)\lVert x-y\rVert\leq\xeqref{d04} 0.5\left[ (\varphi(x))^{\frac{2}{16}} + (\varphi(y))^{\frac{2}{16}}\right]\lVert x-y\rVert.\end{split}
\label{a04}
\end{align}
In addition, \eqref{d01}, the triangle inequality, 
\eqref{d02}, 
the Cauchy--Schwarz inequality,
and
the fact that
$\forall\,x\in\R\colon \lvert x^+\rvert\leq \lvert x\rvert$  show for all
$x=(x_1,x_2,\ldots,x_d)\in\R^d$ that
\begin{align}
&\lVert
\mu(x)
\rVert= \xeqref{d01}\left( \sum_{i=1}^{d}\left\lvert r_i x_i \left(1- \sum_{j=1}^{d}a_{ij}x_j^+ \right)\right\rvert^2\right)^{\frac{1}{2}}\leq 
\bar{r}\lVert
x\rVert+
\left( \sum_{i=1}^{d}\left\lvert r_i x_i  \sum_{j=1}^{d}a_{ij}x_j^+ \right\rvert^2\right)^{\frac{1}{2}}\nonumber\\
&\leq \bar{r}\lVert x\rVert+
\bar{r}\left( \sum_{i=1}^{d}\left\lvert  x_i  \right\rvert^2\right)^{\frac{1}{2}}\left(\max_{i\in\N}\sum_{j=1}^{d}a_{ij}x_j^+ \right)\leq 
\bar{r}\lVert x\rVert+\bar{r}\lVert x\rVert
\left(\max_{i\in\N}\sum_{j=1}^{d}\lvert a_{ij}\rvert^2 \right)^{\frac{1}{2}}
\left(\sum_{j=1}^{d}
\lvert x_j^+\rvert^2\right)^{\frac{1}{2}}\nonumber\\
&\leq \bar{r}\lVert x\rVert+\bar{r}\bar{a}\lVert x\rVert^2.
\end{align}
This, \eqref{d01}--\eqref{d04}, and the fact that
$\forall\,x\in\R\colon x\leq 0.5(1+x^2)$ show for all
$x\in \R^d$ that
\begin{align}
\begin{split}
&
 |Tf(t,x,0)|+|g(x)|+\|\mu(x)\|+ \|\sigma(x)\|^2+ \|x\| \\
&\leq 
  T+ \lVert x\rVert^2+
\left(
\bar{r}\lVert x\rVert+\bar{r}\bar{a}\lVert x\rVert^2\right)+
1+\lVert x\rVert= T+ 1+
(1+\bar{r})\lVert x\rVert+
(1+\bar{r}\bar{a})\lVert x\rVert^2
\\
&\leq T+0.5(3+\bar{r}) +
\left[
0.5(1+\bar{r})+(1+\bar{r}\bar{a})\right]\lVert x\rVert^2\leq
\left[T+
0.5(3+\bar{r})+(1+\bar{r}\bar{a})\right](\varphi(x))^{\frac{1}{8}}.\label{a06}
\end{split}\end{align}
Next, 
the triangle inequality, the Cauchy--Schwarz inequality,
\eqref{d02}, and
the fact that
$\forall\,x,y\in\R\colon \lvert x^+-y^+\rvert\leq \lvert x-y\rvert$  show for all
$x=(x_1,x_2,\ldots,x_d), y=(y_1,y_2,\ldots,y_d)\in\R^d$ that
{\allowdisplaybreaks\begin{align}
&\left(\sum_{i=1}^{d}\left\lvert
\left(x_i
\sum_{j=1}^{d} a_{ij}x_j^+\right)
-\left(
y_i
\sum_{j=1}^{d} a_{ij}y_j^+\right)\right\rvert^2\right)^{\frac{1}{2}}\nonumber\\
&=\left(\sum_{i=1}^{d}
  \left\lvert(x_i-y_i)\left(
\sum_{j=1}^{d} a_{ij}x_j^+
\right)
+y_i\sum_{j=1}^{d} a_{ij}(x_j^+-y_j^+)\right\rvert^2\right)^{\frac{1}{2}}\nonumber\\
&\leq 
\left(
\sum_{i=1}^{d}(x_i-y_i)^2\right)^{\frac{1}{2}}
\left(\sup_{i\in\N}\left\lvert
\sum_{j=1}^{d}a_{ij}x_j^+\right\rvert\right)
+
\left(\sum_{i=1}^{d}\lvert y_i\rvert^2\right)^{\frac{1}{2}}
\left(
\sup_{i\in\N}\left\lvert
\sum_{j=1}^{d} a_{ij}(x_j^+-y_j^+)\right\rvert^2\right)^{\frac{1}{2}}\nonumber
\\
&\leq \lVert x-y\rVert  
\left(
\sup_{i\in \N} \sum_{j\in\N} \lvert a_{ij}\rvert^2\right)^{\frac{1}{2}}
\left(\sum_{j=1}^{d}\lvert x_j^+\rvert^2\right)^{\frac{1}{2}} + \lVert y\rVert  \left(
\sup_{i\in \N} \sum_{j\in\N} \lvert a_{ij}\rvert^2\right)^{\frac{1}{2}}
\left(\sum_{j=1}^{d}\lvert x_j^+-y_j^+\rvert^2\right)^{\frac{1}{2}}\nonumber \\
&\leq\xeqref{d02} \lVert x-y\rVert \bigl(\bar{a}\lVert x\rVert+\bar{a}\lVert y\rVert\bigr).
\label{l03}
\end{align}}%
Furthermore, \eqref{d01} shows for all
$x=(x_1,x_2,\ldots,x_d), y=(y_1,y_2,\ldots,y_d)\in\R^d$ that
\begin{align}\begin{split}
&
\mu_i(x)-\mu_i(y)=
r_i x_i \left(1- \sum_{j=1}^{d}a_{ij}x_j^+ \right)-
r_i y_i \left(1- \sum_{j=1}^{d}a_{ij}y_j^+ \right)\\
&= r_i(x_i-y_i)-r_i
\left[\left(x_i
\sum_{j=1}^{d} a_{ij}x_j^+\right)
-\left(
y_i
\sum_{j=1}^{d} a_{ij}y_j^+\right)
\right].\end{split}\label{d05}
\end{align}
This,  the triangle inequality, \eqref{d02},  \eqref{l03},  
the fact that $\forall\, x,y\in\R\colon xy\leq 0.5x^2+0.5y^2 $,
and \eqref{d04}
 show for all
$x,y\in\R^d $ that
\begin{align}\begin{split}
&\lVert\mu(x)-\mu(y)\rVert\leq \xeqref{d02}\bar{r}\lVert x-y\rVert
+\xeqref{l03}\bar{r}\lVert x-y\rVert (\bar{a}\lVert x\rVert+\bar{a}\lVert y\rVert)\\
&
= 
\bar{r}\lVert x-y\rVert
(1+\bar{a}\lVert x\rVert+\bar{a}\lVert y\rVert)\leq 
\bar{r}\lVert x-y\rVert
\left(1+0.5\bar{a}+0.5\bar{a}\lVert x\rVert^2+0.5\bar{a}+0.5\bar{a}\lVert y\rVert^2\right)\\
&\leq \bar{r}(1+0.5\bar{a})\lVert x-y\rVert
\left[ 1+\lVert x\rVert^2+1+\lVert y\rVert^2\right]\\
&
\leq \bar{r}(1+0.5\bar{a})\lVert x-y\rVert\xeqref{d04}\left[ (\varphi(x))^{\frac{1}{8}} + (\varphi(y))^{\frac{1}{8}}\right].\end{split}\label{d07}
\end{align}
This, \eqref{d05}, the Cauchy--Schwarz inequality, and \eqref{l03} show for all
$x,y\in\R^d$ that
\begin{align}
&
\langle
x-y,\mu(x)-\mu(y)
\rangle=\xeqref{d05}\sum_{i=1}^{d}\left[
r_i(x_i-y_i)^2-r_i(x_i-y_i) 
\left[\left(x_i
\sum_{j=1}^{d} a_{ij}x_j^+\right)
-\left(
y_i
\sum_{j=1}^{d} a_{ij}y_j^+\right)
\right]\right]\nonumber
\\
&
\leq \bar{r}\lVert x-y\rVert^2
+\bar{r}\lVert x-y\rVert\cdot \lVert x-y\rVert (\bar{a}\lVert x\rVert+\bar{a}\lVert y\rVert)\nonumber\\
&\leq \bar{r} (1+\bar{a}\lVert x\rVert+\bar{a}\lVert y\rVert)\lVert x-y\rVert^2.
\end{align}
This, the fact that
$\forall\,x,y\in\R\colon xy\leq \frac{1}{2}8 T e^{\alpha T}x^2+\frac{1}{2}\frac{y^2}{8 T e^{\alpha T}}$,
and \eqref{d04} show for all $x,y\in\R^d$ that
\begin{align}\begin{split}
&
\langle
x-y,\mu(x)-\mu(y)
\rangle
+3\lVert \sigma(x)-\sigma(y)\rVert^2\leq
\Bigl[3+
\bar{r} +\bar{r}\bar{a}\lVert x\rVert+\bar{r}\bar{a}\lVert y\rVert
\Bigr]
\lVert x-y\rVert^2\\
&\leq \left[
3+\bar{r} +
\frac{1}{2}8 T
e^{\alpha T}\bar{r}^2\bar{a}^2+
 \frac{1}{2}\frac{\lVert x\rVert^2}{8 Te^{\alpha T}}+
\frac{1}{2}8 T
e^{\alpha T}\bar{r}^2\bar{a}^2+
 \frac{1}{2}\frac{\lVert y\rVert^2}{8 Te^{\alpha T}}
\right]\lVert x-y\rVert^2\\
&=
\left[
\frac{1}{2}\left(8 T
e^{\alpha T}\bar{r}^2\bar{a}^2+3+\bar{r}\right)+
 \frac{1}{2}\frac{\lVert x\rVert^2}{8 Te^{\alpha T}}+
\frac{1}{2}\left(8 T
e^{\alpha T}\bar{r}^2\bar{a}^2+3+\bar{r}\right)+
 \frac{1}{2}\frac{\lVert y\rVert^2}{8 Te^{\alpha T}}
\right]\lVert x-y\rVert^2\\
&=\xeqref{d04} \frac{U(x)+U(y)}{16T}\lVert x-y\rVert^2.\end{split}\label{a05}
\end{align}
Next, \eqref{d04} shows for all
$x,y,z\in\R^d$
 that 
\begin{align}\label{a07}
(\nabla U) (x) = 2e^{-\alpha T } x\quad\text{and}\quad ((\totalD^2 U)(x))(y,z)= 2e^{-\alpha T }\langle y,z\rangle.
\end{align}
This, \eqref{d01}, \eqref{d06}, \eqref{d02}, and \eqref{d04} show for all $x\in\R^d$ that
\begin{align}
\begin{split}
&
(\totalD \funcUo (x))(\mu (x))+\frac{1}{2} \sum_{k=1}^{d} (\totalD^2\funcUo (x))(\sigma_{k}(x),\sigma_{k}(x))+
\frac{1}{2}e^{\alpha T }\left\|\sigma (x)^*(\nabla \funcUo) (x)\right\|^2+\funcUi (x)\\
&\leq \xeqref{a07}\xeqref{d01}
\langle
2e^{-\alpha T}x,\mu(x)\rangle +\frac{1}{2}\sum_{k=1}^{d} 2e^{-\alpha T} \lVert\sigma_k (x)\rVert^2
+\frac{1}{2}e^{\alpha T }4e^{-2\alpha T}\lVert x\rVert^2
\\
&\leq \xeqref{d06}2e^{-\alpha T}\bar{r}\lVert x\rVert^2
+e^{-\alpha T} \lVert x\rVert^2 +2e^{-\alpha T}\lVert x\rVert^2= e^{-\alpha T}(2\bar{r}+3)\lVert x\rVert^2\leq\xeqref{d02}\xeqref{d01} \alpha  U(x).
\end{split}
\end{align}
Combining this,
\eqref{a01},
\eqref{a02},
\eqref{a03}, \eqref{a04}, \eqref{a06}, \eqref{d07}, \eqref{a05}, and \eqref{a07}
yields that the conditions \eqref{v05}--\eqref{eq:exponential.integrability} are satisfied 
(with $p\gets 16$, $\beta\gets 2$, $L\gets 1$, 
$\gamma\gets 1$,
and suitably large enough $b,c$ which do not depend on $d$).

\end{example}

\subsubsection*{Acknowledgements}
 This work has been
funded by the Deutsche Forschungsgemeinschaft (DFG, German Research Foundation) through
the research grant HU1889/6-2.

{\small
\bibliographystyle{acm}
\bibliography{bibfile-Xnonglob25}
}

\end{document}